\documentclass[11pt]{article}

\usepackage[margin=.7in]{geometry}%modify the margins
\usepackage{amsfonts}%for mathcal, mathbb etc..
\usepackage[mathcal]{euscript}%for mathcal, mathbb etc..
\usepackage{amsmath}%for eqref
\usepackage{amsthm}%for proofs
\usepackage{graphicx}

\newcommand{\luoz}{\lambda_{1,0}}%one zero\newcommand{\luzo}{\lambda_{0,1}}%zero one
\newcommand{\luzo}{\lambda_{0,1}}%one zero
\newcommand{\lu}[2]{\lambda_{#1,#2}}%\underscore

\newcommand{\luij}{\lambda_{ij}}
\newcommand{\dx}{\partial_x}
\newcommand{\dy}{\partial_y}
\newcommand{\mfl}{\mathfrak{L}}%MathFrak L
\newcommand{\mfn}{M}%{\mathfrak n}%MathFrak N
\newcommand{\Lal}{\mathcal L_{\mfn,\alpha}}

\newcommand{\LA }{\mathcal L_{\mfn,\alpha}^A}
\newcommand{\LL }{\mathcal L_{\mfn,\alpha}^L}
\newcommand{\LN }{\mathcal L_{\mfn,\alpha}^N}

\newcommand{\Dop}[2]{\mathcal D^{(#1,#2)}}

\newcommand{\matC}{\mathsf C}
\newcommand{\matM}{\mathsf A^{\mathcal F}}
\newcommand{\VGPW}{\mathbb V_{\alpha,p,q}^{0}}%\mfn,
\newcommand{\cpxi}{\mathrm i}

\newcommand{\RHS}{N}

\newcommand{\defOp}{Consider a point $(x_0,y_0)\in\mathbb R^2$, a given $q\in\mathbb N^*$, a given $\mfn\in\mathbb N$, $\mfn\geq 2$, a given set of complex-valued functions $\alpha=\{ \alpha_{k,l}\in\mathcal C^{q-1} \text{ at }(x_0,y_0),0\leq k+l\leq \mfn \}$, and the corresponding partial differential operator $\Lal$. }
\newcommand{\defp}{Let  $p\in\mathbb N^*$ be the number of desired basis functions. }
\newcommand{\defParThm}{Consider an open set $\Omega\subset\mathbb R^2$, $(x_0,y_0)\in\Omega$, $\mfn=2$, a given $(n,p,q)\in\mathbb (N^*)^3$, $n\geq\mfn$, $q\geq n-1$ and a given set of complex-valued functions $\alpha=\{ \alpha_{k_1,k_2}\in\mathcal C^{n}(\Omega),0\leq k_1+k_2\leq \mfn \}$, the corresponding partial differential operator $\Lal$ and set of GPWs $\VGPW$. }
\newcommand{\capfig}[1]{GPW approximation of $u_#1$ by $u_a\in\VGPW$ with $p=2n+1$. We represent the $L^\infty$ error $\max_{(x_0,y_0)\in\Omega}\|u_{#1}-u_a\|_{L^\infty(\{(x,y)\in\mathbb R^2, |(x,y)-(x_0,y_0)|<h \})}$, for 50 random points $(x_0,y_0)\in\Omega_{#1}$.
We compare results for parameters satisfying Theorem \ref{thethm} hypotheses i.e. $q=\max (1,n-1)$ (Left panel), and for varying $q$ with fixed $n$ (Right panel).}

\newtheorem{thm}{Theorem}
\newtheorem{cor}{Corollary}
\newtheorem{prop}{Proposition}
\newtheorem{dfn}{Definition}
\newtheorem{lmm}{Lemma}
\newtheorem{hyp}{Hypothesis}
\newtheorem{rmk}{Remark}

\usepackage{algorithmicx}
\usepackage{algorithm}% http://ctan.org/pkg/algorithms
\usepackage[noend]{algpseudocode}% http://ctan.org/pkg/algorithmicx

\usepackage{framed}
\usepackage{xcolor} 
\usepackage{tikz,pgfplots}
\pgfplotsset{compat=newest,
legend style={
font=\footnotesize,
rounded corners=2pt,
}}
\pgfplotscreateplotcyclelist{mycolorlist}{%
blue,every mark/.append style={fill=blue!80!black},mark=*\\%
red,every mark/.append style={fill=red!80!black},mark=square*\\%
brown!60!black,every mark/.append style={fill=brown!80!black},mark=otimes*\\%
black,mark=star\\%
blue,every mark/.append style={fill=blue!80!black},mark=diamond*\\%
red,densely dashed,every mark/.append style={solid,fill=red!80!black},mark=*\\%
brown!60!black,densely dashed,every mark/.append style={
solid,fill=brown!80!black},mark=square*\\%
black,densely dashed,every mark/.append style={solid,fill=gray},mark=otimes*\\%
blue,densely dashed,mark=star,every mark/.append style=solid\\%
red,densely dashed,every mark/.append style={solid,fill=red!80!black},mark=diamond*\\%
}

\title{A roadmap for Generalized Plane Waves 
and their interpolation properties
}

\author{Lise-Marie Imbert-G\'erard, Guillaume Sylvand}

\begin{document}
\maketitle
%\tableofcontents

%%%%%%%%%%%%%%%%%%
\begin{abstract}
This work focuses on the study of partial differential equation (PDE) based basis function for Discontinuous Galerkin methods to solve numerically  wave-related boundary value problems with variable coefficients. To tackle problems with constant coefficients, wave-based methods have been widely studied in the literature: they rely on the concept of Trefftz functions, i.e. local solutions to the governing PDE, using oscillating basis functions rather than polynomial functions to represent the numerical solution. Generalized Plane Waves (GPWs) are an alternative developed to tackle problems with variable coefficients, in which case Trefftz functions are not available. In a similar way, they incorporate information on the PDE, however they are only approximate Trefftz functions since they don't solve the governing PDE exactly, but only an approximated PDE. 
Considering a new set of PDEs beyond the Helmholtz equation, we propose to set a roadmap for the construction and study of local interpolation properties of GPWs. Identifying carefully the various steps of the process, we provide an algorithm to summarize the construction of these functions, and establish necessary conditions to obtain high order interpolation properties of the corresponding basis.

%say what restrictions including smoothness
\end{abstract}

%%%%%%%%%%%%%%%%%%%%%%
\section{Introduction}
Trefftz methods are Galerkin type of methods that rely on function spaces of local solutions to the governing partial differential equations (PDEs). They were initially introduced in \cite{Trefftz,transl}, and the original idea was to use trial functions which satisfy the governing PDE to derive error bounds. They have been widely used in the engineering community \cite{kita} since the 60's, for instance for Laplace's equation \cite{mik}, to  the  biharmonic equation \cite{rokt} and to  elasticity \cite{Kup}.
 Later the general idea of taking advantage of analytical knowledge about the problem to build a good approximation space was used to develop numerical methods: in
the presence of corner and interface singularities \cite{fix,Sfix}, boundary layers, rough coefficients, elastic interactions \cite{PUM96,PUM97,Mel,babz}, wave propagation \cite{PUM97,DEM}.
In the context of boundary value problems (BVPs) for time-harmonic wave propagation, several methods have been proposed following the idea of  functions that solves the governing PDE, \cite{survey}, relying on incorporating oscillating functions in the function spaces to derive and discretize a weak formulation.
 Wave-based numerical methods have received attention from several research groups around the world, from the theoretical \cite{survey} and computational \cite{monk3} point of view, and the pollution effect of plane wave Discontinuous Galerkin (DG) methods was studied in \cite{GH14}. 
Such methods have also been implemented in industry codes\footnote{http://www.waveller.com/Waveller\_Acoustics/waveller\_acoustics.shtml}, for acoustic applications.
The use of Plane Wave (PW) basis functions has been the most popular choice, while an attempt to use Bessel functions was reported in \cite{Teemu}. In \cite{comp}, the authors present an interesting comparison of performance between high order polynomial and wave-based methods. More recently, application to space-time problems have been studied in \cite{MP,DM1,tsuk,DM2,K+}.
 
In this context, numerical methods rely on discretizing a weak formulation via a set of exact solutions of the governing PDE. When no exact solutions to the governing PDE are available, there is no natural choice of basis functions to discretize the weak formulation. This is in particular the case for variable coefficient problems. In order to take advantage of Trefftz-type methods for problems with variable coefficients, Generalized Plane Waves (GPWs) were introduced in \cite{LMD}, as basis functions that are local approximate solutions - rather than exact solutions - to the governing PDE. GPWs were designed adding higher order terms in the phase of classical PWs, choosing these higher order terms to ensure the desired approximation of the governing PDE.
 In \cite{LMinterp}, the construction and interpolation properties of GPWs were studied for the Helmholtz equation
\begin{equation}\label{Helm}
-\Delta u +\beta(x,y))u = 0,
\end{equation}
with a particular interest for the case of a sign-changing coefficient $\beta$, including propagating solutions ($\beta<0$), evanescent solutions ($\beta>0$), smooth transition between them ($\beta=0$) called cut-offs in the field of plasma waves. The interpolation properties of a set $\mathbb V$ spanned by resulting basis functions, namely $\| (I-P_{\mathbb V}) u \|$ where $P_{\mathbb V}$ is the orthogonal projector on $\mathbb V$ while $u$ is the solution to the original problem, play a crucial role in the error estimation of the corresponding numerical method \cite{cess}. For this same equation the error analysis of a modified Trefftz method discretized with GPWs was presented in \cite{IGM}.
In \cite{MC}, Generalized Plane Waves (GPWs) were used for the numerical simulation of mode conversion modeled by the following equation:
\begin{equation}\label{eq:2ndOFnt}% Second Order F no tilde
\left( \dx^2 +(d+\overline d)\dx\dy +|d|^2\dy^2\right) F + (d-\overline d)x\partial_y F-\left(1+\frac{1}{\mu}+x(x+y)\right)F = 0.
\end{equation} 

In the present work, we answer questions related to extending the work on GPW developed in \cite{LMinterp} - the construction of GPWs on the one hand, and their interpolation properties on the other hand - from the Helmholtz operator $-\Delta+\beta$ to a wide range of partial differential operators. A construction process valid for some operators of order two or higher is presented, while a proof of interpolation properties is limited to some operators of order two. We propose a road map to identify crucial steps in our work:
\begin{enumerate}
\item Construction of GPWs $\varphi$ such that $\mathcal L\varphi\approx 0$
\begin{enumerate}
\item Choose an ansatz for $\varphi$ (Section \ref{sec:constr}).
\item Identify the corresponding $N_{dof}$ degrees of freedom and $N_{eqn}$ constraints (Subsection \ref{ssec:TE2s}).
\item Choose the number of degrees of freedom adequately $N_{dof}\geq N_{eqn}$ (Subsection \ref{ssec:TE2s}).
\item Study the structure of the resulting system and identify $N_{dof}-N_{eqn}$ additional constraints (Subsections \ref{ssec:NL2L} and \ref{ssec:hier}).
\item Compute the remaining $N_{eqn}$ degrees of freedom at minimal computational cost (Subsection \ref{ssec:algo}).
\end{enumerate}
\item Interpolation properties
\begin{enumerate}
\item Study the properties of the remaining $N_{eqn}$ degrees of freedom with respect to the $N_{dof}-N_{eqn}$ additional constraints
\item Identify a simple reference case depending only on the $N_{dof}-N_{eqn}$ additional constraints (Section \ref{sec:norm}).
\item Study the interpolation properties of this reference case (Subsection \ref{sec:PWinterp}).
\item Relate the general case to the reference case (Subsections \ref{ssec:every} and \ref{ssec:each}).
\item Prove the interpolation properties of the GPWs from those of the reference case (Subsection \ref{sec:GPWinterp}).
\end{enumerate}
\end{enumerate}
%\begin{tabular}{cc}
%Construction & Interpolation
%\end{tabular}
We will consider linear partial differential operators with variable coefficients, defined as follows.
\begin{dfn}
A linear partial differential operator of order $\mfn \geq 2$, in two dimensions, with a given set of complex-valued coefficients  $\alpha = \{\alpha_{k,\ell-k},(k,\ell)\in\mathbb N^2, 0\leq k\leq\ell\leq \mfn\}$ will be denoted hereafter as
$$ \Lal := \sum_{\ell = 0}^{\mfn} \sum_{k=0}^\ell \alpha_{k,\ell-k}\left( x , y \right) \partial_x^k \partial_y^{\ell-k}.$$ 
\end{dfn}
Our goal is to build a basis of functions well suited to approximate locally any solution $u$ to a given homogeneous variable-coefficient partial differential equation
$$
 \Lal u = 0 \text{ on a domain } \Omega\subset\mathbb R^2,
$$
where by locally we mean piecewise on a mesh $\mathcal T_h$ of $\Omega$. Such interpolation properties are a building block for the convergence proof of Galerkin methods.
For a constant coefficient operator, it is natural to use the same basis on each element $K\in\mathcal T_h$. However, with variable coefficients, it cannot be optimal to expect a single basis to have good approximation properties on the whole domain $\Omega\subset\mathbb R^2$. For instance, for the Helmholtz equation with a sign-changing coefficient, it can not be optimal to look for a single basis that would give a good approximation of solutions both in the propagating region and in the evanescent region. Therefore it is natural to think of local bases defined on each $K\in\mathcal T_h$: with GPWs we focus on local properties around a given point $(x_0,y_0)\in\mathbb R^2$ rather than on a given domain $\Omega$. 
A simple idea would then be freezing the coefficients of the operator, that is to say studying, instead of $\Lal$, the constant coefficient operator $\mathcal L_{\mfn,\bar\alpha}$ with constant coefficients $\bar \alpha = \{\alpha_{k,l}(x_0,y_0),0\leq k+l\leq \mfn\}$.
However, as observed in \cite{LMinterp,Waves2015}, this leads to low order approximation properties, while we are interested in high order approximation properties. This is why new functions are needed to handle variable coefficients. This work will focus on two aspects: the construction and the interpolation properties of GPWs. 

%Keeping in mind problems related to wave propagation, w
We follow the GPW design proposed in \cite{LMinterp,LMD}. Retaining the oscillating feature while aiming for higher order approximation, GPW were designed with Higher Order Terms ($HOT$) in the phase function of a plane wave. These higher order terms are to be defined to ensure that a GPW function $\varphi$ is an approximate solution to the PDE: 
\begin{equation}\label{eq:c/GPW}
\left\{\begin{array}{l}
\phi(x,y)=\exp i\kappa(\cos\theta x + \sin\theta y )\\
\left[-\Delta-\kappa^2\right]\phi=0
\end{array}\right.
\text{ versus }
\left\{\begin{array}{l}
\varphi(x,y)= \exp( i\kappa(\mathbf \cos\theta x + \sin\theta y )+HOT)\\
\Lal \varphi\approx0
\end{array}\right.
\end{equation}
In Section \ref{sec:constr}, the construction of a GPW $ \varphi(x,y)=e^{P(x,y)}$ will be described in detail, then a precise definition of GPW will be provided under the following hypothesis:
\begin{hyp}\label{hyp}
Consider a given point $(x_0,y_0)\in\mathbb R^2$, 
a given approximation parameter $q\in\mathbb N$, $q\geq 1$,  %a given parameter $n\in\mathbb N$, $n\geq 2$,  
a given $\mfn\in\mathbb N$, $\mfn\geq 2$, and a partial differential operator $\Lal$ defined by a given set of complex-valued coefficients $\alpha = \{\alpha_{k,l},0\leq k+l\leq \mfn\}$, defined in a neighborhood of $(x_0,y_0)$, satisfying
\begin{itemize}
\item $\alpha_{k,l}$ is $\mathcal C^{q-1}$ % $\mathcal C^{n-\mfn}$
 at $(x_0,y_0)$ for all $(k,l)$ such that $0\leq k+l\leq \mfn$,
\item  $\alpha_{\mfn,0}(x_0,y_0)\neq 0$.
\end{itemize}
\end{hyp}
%where $n$ will be related to the expected order of convergence in the interpolation properties,
\noindent This construction is equivalent to the construction of the bi-variate polynomial 
$$\displaystyle P(x,y) = \sum_{0\leq i+j\leq d P} \luij (x-x_0)^i(y-y_0)^j,$$
and is performed by choosing the degree $dP$, and providing an explicit formula for the set of complex coefficients $\{\luij\}_{\{ (i,j)\in\mathbb N^2, 0\leq i+j\leq d P \}}$, in order for $\varphi$ to satisfy $\Lal \varphi(x,y)=O\left(\| (x,y)-(x_0,y_0) \|^q\right)$.  An algorithm to construct a GPW is provided. In Section \ref{sec:norm} properties of the $\luij$s are studied, while the interpolation properties of the corresponding set of basis functions are studied for the case $\mfn=2$ in Section \ref{sec:int}, under the following hypothesis:
\begin{hyp}\label{hyp2}
Under Hypothesis \ref{hyp} we consider only operators $\Lal$ such that $\mfn$ is even and the terms of order $\mfn$ satisfy
$$
\sum_{k=0}^\mfn \alpha_{k,\mfn-k} (x_0,y_0)X^kY^{\mfn-k}
=
(\gamma_1X^2+\gamma_2XY+\gamma_3Y^2)^{\frac{\mfn}{2}}
$$
%
%SO EITHER WE WRITE SOMETHING LIKE
%for some complex numbers $(\gamma_1,\gamma_2,\gamma_3)$ with 
%$\Gamma=\begin{pmatrix}
%\gamma_1 &\gamma_2/2\\\gamma_2/2&\gamma_3
%\end{pmatrix}
%$ 
%hermitian with non-zero eigenvalues, $\mu_1,\mu_2$, $\exists S$ unitary such that $S\Gamma S^*=\begin{pmatrix}
%\mu_1&0\\0&\mu_2
%\end{pmatrix}=:D$ and 
%
%OR SOMETHING LIKE
for some complex numbers $(\gamma_1,\gamma_2,\gamma_3)$ such that there exists $(\mu_1,\mu_2)\in\mathbb C^2$, $\mu_1\mu_2\neq 0$, a non-singular matrix $A\in\mathbb C^{2\times2}$ satisfying
$\Gamma = A^tDA$ where
$\Gamma=\begin{pmatrix}
\gamma_1 &\gamma_2/2\\\gamma_2/2&\gamma_3
\end{pmatrix}$ and
$D=\begin{pmatrix}
\mu_1 &0\\0&\mu_2
\end{pmatrix}$
, and therefore
$$
\sum_{k=0}^\mfn \alpha_{k,\mfn-k} (x_0,y_0)X^kY^{\mfn-k}
=
\left( \mu_1(A_{11}X+A_{12}Y)^2+\mu_2(A_{21}X+A_{22}Y)^2\right)^{\frac{\mfn}2}.
$$
\end{hyp}
For instance, these matrices are $\Gamma=D=Id$ for $\mathcal L_H:=-\Delta-\kappa^2(x,y)$ or  $\mathcal L_B:=\Delta\mathcal L_H$, and $\Gamma=D=c(x_0,y_0)Id$ for $\mathcal L_C:=-\nabla\cdot(c(x,y)\nabla )-\kappa^2(x,y)$. Note that if $\Gamma$ is real, this is simply saying that its eigenvalues are non-zero. %, then $A$ is real, but if $\Gamma$ is not real then this is non standard since it is not the conjugate but only the transpose of $A$ that appears.
\noindent
Finally, corresponding numerical results are then provided, for various operators $\Lal$ of order $\mfn =2$ in Section \ref{sec:NR}. 
%Numerical results are then provided, for operators $\Lal$ of order $\mfn =2$ in Section \ref{sec:NR2} and for an operator $\Lal$ of order $\mfn =4$ in Section \ref{sec:NR4}.

%Announce the THM here so that I can use it in my comments on PW before the proof?? but it would not be stated for PWs anyway... at least announce that I want to study approximation properties and say that it's equivalent to matching Taylor exp of the solution u to be approx and a linear comb of our brand new BF 

Our previous work was limited to % elliptic equations - such as Helmholtz equation - 
the Helmholtz equation \eqref{Helm} for propagating and evanescent regions, transition between the two, absorbing regions, as well as caustics. 
The interpolation properties presented here cover more general second order equations, in particular equations that can be written as
\begin{equation}\label{eq:2O}
\nabla \cdot (A\nabla u) + \mathbf d\cdot\nabla u + k^2 m u = 0,
\end{equation}
with variable coefficients $A$ matrix-valued, real and symmetric with non-zero eigenvalues, $\mathbf d$ vector-valued and $m$ scalar-valued.
It includes for instance
\begin{itemize}
\item Helmholtz equation with absorption corresponding to $A = I$ with $\Re(m)>0$ and $\Im(m)\neq 0$ ;
\item the mild-slop equation \cite{mild-slope} modeling the amplitude of the free-surface water waves corresponding to $m=c_pc_g$ being the product of $c_p$ the phase speed of the waves and $c_g$ the group speed of the waves with  $A = m Id$ ;
\item if $\mu$ is the permeability and $\epsilon$ the permittivity, then the transverse-magnetic mode of Maxwell's equations for $A=\frac1\mu I$ and $m=\epsilon$, while  the transverse-electric mode of Maxwell's equations for $A=\frac1\epsilon I$ and $m=\mu$.
\end{itemize}

%- convected Helmholtz
%$$
%-\nabla\cdot(\rho(\nabla\phi-(\bfm\cdot\nabla\phi)\bfm+\mi k \phi\bfm))
%-\rho(k^2\phi+\mi k \bfm\cdot\nabla\phi)=0
%$$
%- mode conversion
%\begin{equation}%\label{eq:2ndOFnt}% Second Order F no tilde
%\left( \dx^2 +(d+\overline d)\dx\dy +|d|^2\dy^2\right) F + (d-\overline d)x\partial_y F-\left(1+\frac{1}{\mu}+x(x+y)\right)F = 0,
%\end{equation} 
%- mild-slope equation
%$$
%\nabla \cdot( c_p c_g\nabla \eta)+k^2 c_p c_g \eta = 0
%$$

Throughout this article, we will denote by $\mathbb N$ the set of non-negative integers, by $\mathbb N^*$ the set of positive integers, by $\mathbb R^{+} = [0;+\infty)$ the set of non-negative real numbers, and by $\mathbb C[z_1,z_2]$ the space of complex polynomials with respect to the two variables $z_1$ and $z_2$. As the first part of this work is dedicated to finding the coefficients $\luij$, we will reserve the word unknown to refer to the $\lambda_{i,j}$s. The length of the multi-index $(i,j)\in\mathbb N^2$ of an unknown $\luij$,  $|(i,j)|=i+j$, will play a crucial role in what follows. % and will hereafter be referred to as \totind.

%%%%%%%%%%%%%%%%
\section{Construction of a GPW}\label{sec:constr}
The task of constructing a GPW is attached to a homogeneous PDE, it is not global on $\mathbb R^2$ but it is local as it is expressed in terms of a Taylor expansion. It consists in finding a polynomial $P\in\mathbb C[x,y]$ such that the corresponding GPW, namely $ \varphi:=e^{P}$, is locally an approximate solution to the PDE.

Consider $\mfn = 2$, $\beta = \{ \beta_{0,0}, \beta_{0,1}=\beta_{1,0}=\beta_{1,1}=0, \beta_{0,2}\equiv -1,\beta_{2,0}\equiv -1 \}$, and the corresponding the operator $\mathcal L_{2,\beta} = -\partial_x^2  -\partial_y^2 + \beta_{0,0}(x)$. Then for any polynomial $P\in\mathbb C[x,y]$:
$$\mathcal L_{2,\beta}e^{P(x,y)} = \left(-\partial_x^2 P - (\partial_x P)^2 -\partial_y^2 P - (\partial_y P)^2 + \beta_{0,0}(x,y) \right) e^{P(x,y)}
,$$
so the construction of an exact solution to the PDE would be equivalent to the following problem:
\begin{equation}\label{eq:PLH1}
\text{Find } P\in\mathbb C[x,y] \text{ such that }
\partial_x^2 P(x,y) + (\partial_x P)^2(x,y) +
\partial_y^2 P(x,y) + (\partial_y P)^2(x,y) 
= \beta_{0,0}(x,y).
\end{equation}
Consider then the following examples.
\begin{itemize}
\item If $\beta_{0,0}(x,y)$ is constant, then it is straightforward to find a polynomial of degree one satisfying Problem \eqref{eq:PLH1}; $\beta_{0,0}$ being negative this would correspond to a classical plane wave.
\item   If $\beta_{0,0}(x,y)=x$, then there is no solution to \eqref{eq:PLH1}, since the total degree of $\partial_x^2 P + (\partial_x P)^2+\partial_y^2 P + (\partial_y P)^2$ is always even.
\item If $\beta_{0,0}(x,y)$ is not a polynomial function, it is also straightforward to see that no polynomial $P$ can satisfy Problem \eqref{eq:PLH1}.
\end{itemize}

From these trivial examples we see that in general there is no such function, $\varphi(x,y)=e^{P(x,y)}$, $P$ being a complex polynomial, solution to a variable coefficient partial differential equation exactly. It could seem that the restriction for $P$ to be a polynomial is very strong. 
However since we are interested in approximation and smooth coefficients, rather than looking for a more general phase function we restrict the identity $\mathcal L \varphi=0$ on $\Omega$ into an approximation on a neighborhood of $(x_0,y_0)\in\mathbb R^2$ in the following sense. We replace the too restrictive cancellation of $\Lal e^{P(x,y)}$ by the cancellation of the lowest terms of its Taylor expansion around $(x_0,y_0)$.  So this section is dedicated to the construction of a polynomial $P\in\mathbb C[x,y]$, under Hypothesis \ref{hyp}, to ensure that the following local approximation property
\begin{equation}\label{eq:Lhq}
\Lal e^{P(x,y)} = O(\|(x-x_0,y-y_0)\|^q)
\end{equation} 
is satisfied. The parameter $q$ will denote throughout this work the order of approximation of the equation to which the GPW is designed. In summary, the construction is performed:
\begin{itemize}
\item for a partial differential operator $\Lal$ of order $\mfn$ defined by a set of smooth coefficients $\alpha$,
\item at a point $(x_0,y_0)\in\mathbb R^2$,
\item at order $q\in\mathbb N^*$,
\item to ensure that
$
\Lal e^{P(x,y)} = O(|(x-x_0,y-y_0)|^q).
$
\end{itemize}

Even though the construction of a GPW will involve a non-linear system %to be satisfied by the unknowns $\{\luij\}_{\{ (i,j)\in\mathbb N^2, 0\leq i+j\leq d P \}}$, 
we propose to take advantage of the structure of this system to construct a solution via an explicit  formula.
In this way, even though a GPW $ \varphi:=e^{P}$ is a PDE-based function, the polynomial $P$ can be constructed in practice from this formula, and therefore the function can be constructed without solving numerically any non-linear - or even linear - system. This remark is of great interest with respect to the use of such functions in a Discontinuous Galerkin method to solve numerically boundary value problems.

In order to illustrate the general formulas that will appear in this section, we will use the specific case $\mathfrak L_{2,\gamma}$ where $\gamma = \{ \gamma_{0,0},\gamma_{1,0},\gamma_{0,1},\gamma_{2,0} \equiv -1,\gamma_{1,1} , \gamma_{0,2}\}$, for which we can write explicitly many formulas is a compact form. %[EX]
In order to simplify certain expressions that will follow we propose the following definition.
\begin{dfn}
Assume $(i,j)\in\mathbb N^2$ and $(x_0,y_0)\in\mathbb R^2$. We define the linear partial differential operator $D^{(i,j)}$  by
$$
\Dop{i}{j}: f\in\mathcal C^{i+j} \mapsto \frac1{i!j!}\dx^i\dy^j f.
$$
\end{dfn}

A precise definition of GPW will be provided at the end of this section.

%\subsection{A non-linear system/main aspects}
\subsection{From the Taylor expansion to a non-linear system}
\label{ssec:TE2s}
We are seeking a polynomial 
$\displaystyle P(x,y) = \sum_{0\leq i+j\leq d P} \luij (x-x_0)^i(y-y_0)^j$ satisfying the Taylor expansion \eqref{eq:Lhq}. Defining such a polynomial is equivalent to defining the set $ \{\luij ; (i,j)\in\mathbb N^2,0\leq i+j\leq d P \} $, and therefore we will refer to the $\luij$s as the unknowns throughout this construction process. The goal of this subsection is to identify the set of equations to be satisfied by these unknowns to ensure that $P$ satisfies the Taylor expansion \eqref{eq:Lhq}, and in particular to choose the degree of $P$ so as to guarantee the presence of linear terms in each equation of the system. 

According to the Faa di Bruno formula, the action of the partial differential operator $\Lal$ on a function  $\varphi (x,y)=  e^{P(x,y)}$ is given by 
\begin{align*}
 \Lal e^{P( x , y )} =& e^{P(x,y)} \Bigg(\alpha_{0,0}(x,y)
\\&+ \sum_{\ell = 1}^{\mfn} \sum_{k=0}^\ell \alpha_{k,\ell-k}\left( x , y \right)
k!{(\ell-k)}! \sum_{1\leq\mu\leq \ell} \sum_{s=1}^{\ell} \sum_{p_s((k,\ell-k),\mu)} \prod_{m=1}^s \frac{1}{k_m!}\left(
 \Dop{i_m}{j_m}%\frac{1}{i_m!j_m!} \dx^{i_m}\dy^{j_m}  
 P(x,y)
 \right)^{k_m}\Bigg),
\end{align*} 
where the linear order $\prec$ on $\mathbb N^2$ is defined by
$$
\forall (\mu,\nu)\in(\mathbb N^2)^2, \mu\prec \nu \Leftrightarrow 
\begin{array}{l}
1.\ \mu_1+\mu_2 < \nu_1+\nu_2 ; \text{ or }\\
2.\ \mu_1+\mu_2=\nu_1+\nu_2 \text{ and } \mu_1 <\mu_2,
\end{array}
$$
and where $p_s((i,j),\mu)$ is equal to
$$
\begin{array}{l}
\Bigg\{
(k_1,\cdots,k_s;(i_1,j_1),\cdots,(i_s,j_s)):k_i>0,0\prec(i_1,j_1)\prec\cdots\prec(i_s,j_s),\\
\displaystyle \phantom{\Bigg\{}
\sum_{l=1}^s k_l = \mu,
\sum_{l=1}^s k_li_l = i,
\sum_{l=1}^s k_lj_l = j
\Bigg\}.
\end{array}
$$

%% MAYBE GIVE EXAMPLES FOR A  $\mathfrak n=2$ and a $\mathfrak n = 4$ CASE
%[EX] 
For the operator $\mathfrak L_{2,\gamma}$ the Faa di Bruno formula becomes
%\begin{align*}
%\mathfrak L_{2,\gamma} e^{P( x , y )} =& 
%e^{P(x,y)} \Bigg(
%-\dx^2 P(x,y) + \gamma_{1,1}(x,y)\dx\dy P(x,y) + \gamma_{0,2}(x,y)\dy^2 P(x,y)
%\\&\phantom{e^{P(x,y)} \Bigg(}
%-(\dx P(x,y))^2 + \gamma_{1,1}(x,y)\dx P(x,y)\dy P(x,y) + \gamma_{0,2}(x,y)(\dy P(x,y))^2
%\\&\phantom{e^{P(x,y)} \Bigg(}
%+ \gamma_{1,0}(x,y)\dx P(x,y) + \gamma_{0,1}(x,y)\dy P(x,y) + \gamma_{0,0}(x,y)
%\Bigg).
%\end{align*}
\begin{align*}
\mathfrak L_{2,\gamma} e^{P} =& 
e^{P} \Bigg(
-\dx^2 P + \gamma_{1,1}\dx\dy P + \gamma_{0,2}\dy^2 P
-(\dx P)^2 + \gamma_{1,1}\dx P\dy P + \gamma_{0,2}(\dy P)^2
\\&\phantom{e^{P} \Bigg(}
+ \gamma_{1,0}\dx P + \gamma_{0,1}\dy P + \gamma_{0,0}
\Bigg).
\end{align*} 
 
In order to single out the terms depending on $P$ in the right hand side, this leads to the following definition.
\begin{dfn}\label{def:LA}
Consider a given $\mfn\in\mathbb N$, $\mfn\geq 2$, a given set of complex-valued functions $\alpha=\{ \alpha_{k,l},0\leq k+l\leq \mfn \}$, and the corresponding partial differential operator $\Lal$. 
We define the partial differential operator $\LA$ associated to $\Lal$ as
\begin{equation*}
\LA = \sum_{\ell = 1}^{\mfn} \sum_{k=0}^\ell k!{(\ell-k)}! \alpha_{k,\ell-k}
\sum_{1\leq\mu\leq \ell} \sum_{s=1}^{\ell} \sum_{p_s((k,\ell-k),\mu)} \prod_{m=1}^s \frac{1}{k_m!}\left(  \Dop{i_m}{j_m}(\cdot) \right)^{k_m},
\end{equation*}
or equivalently, since the exponential of a bounded quantity is bounded away from zero:
$$\LA:f\in\mathcal C^\mfn\mapsto 
\frac{\Lal e^f}{e^f}-\alpha_{0,0}.
$$
\end{dfn}

%% FOLLOW UP EXAMPLES FOR A  $\mathfrak n=2$ and a $\mathfrak n = 4$ CASE
%[EX] 
For the operator $\mathfrak L_{2,\gamma}$  this gives
\begin{equation*}
\mathfrak L_{2,\gamma}^A P= -\dx^2 P + \gamma_{1,1}\dx\dy P + \gamma_{0,2}\dy^2 P
-(\dx P)^2 + \gamma_{1,1}\dx P\dy P + \gamma_{0,2}(\dy P)^2
+ \gamma_{1,0}\dx P + \gamma_{0,1}\dy P.
\end{equation*}
 
Since, for any polynomial $P$, the function $e^P$ is locally bounded, and since ${\Lal[ e^P]}=\big(\LA{e^P}+\alpha_{0,0}\big)e^P$, then for a polynomial $P$ to satisfy the approximation property \eqref{eq:Lhq}, it is sufficient to satisfy
\begin{equation}\label{eq:Nhq}
\LA P (x,y) =- \alpha_{0,0}(x,y) + O(|(x-x_0,y-y_0)|^q).
\end{equation}
Therefore, the problem to be solved 
%for the unknowns $\{\lambda_{i,j}\}_{\{ (i,j),-\leq i+j\leq dP \}}$
is now:
\begin{equation}\label{thesystdeg}
\text{Find } P\in\mathbb C[x,y], \text{ s.t. }
\forall (I,J)\in\mathbb N^2, 0\leq I+J <q,
\Dop{I}{J} \LA P (x_0,y_0)
=
-\Dop{I}{J} \alpha_{0,0}(x_0,y_0).
\end{equation}

In order to define a polynomial $\displaystyle P(x,y) = \sum_{0\leq i+j\leq d P} \luij (x-x_0)^i(y-y_0)^j$, the degree $dP$ of the polynomial determines the number of unknowns: there are $N_{dof} = \frac{(dP+1)(dP+2)}{2}$ unknowns to be defined, namely the $\{\lambda_{i,j}\}_{\{(i,j)\in\mathbb N,0\leq i+j\leq dP}$.
In order to design a polynomial $P$ satisfying Equation \eqref{eq:Nhq}, the parameter $q$ determines the number of equations to be solved: there are $N_{eqn} = \frac{q(q+1)}{2}$ terms to be canceled from the Taylor expansion. 
The first step toward the construction of a GPW is to define the value of $dP$ for a given value of $q$. 

%%%%%%%
At this point it is clear that if $dP\leq q-1$, then the resulting system is over-determined.
Our choice for the polynomial degree $dP$ relies on a careful examination of the linear terms in $\LA P$.
We can already notice that, under Hypothesis \ref{hyp}, in $\LA P$ there is at least one non-zero linear term, namely $\alpha_{\mfn,0}(x_0,y_0)\dx^\mfn P$, and there is at least one non-zero non-linear term, namely $\alpha_{\mfn,0}(x_0,y_0)(\dx P)^\mfn$. This non-linear term corresponds to the following parameters from the Faa di Bruno formula: $\mu=\mfn$, $s=1$, $(k_1,(i_1,j_1)) = (\mfn,(1,0))$.
 The linear terms can only correspond to $s=1$, $\mu=1$ and $p_1((k,\ell-k),1) = \{ (1,(k,\ell-k)) \}$, see Definition \ref{def:LA}. We can then split $\LA$ into its linear and non-linear parts. 
\begin{dfn}
Consider a given $\mfn\in\mathbb N$, $\mfn\geq 2$, a given set of complex-valued functions $\alpha=\{ \alpha_{k,l},0\leq k+l\leq \mfn \}$, and the corresponding partial differential operator $\Lal$. 
The linear part of the partial differential operator $\LA$ is defined by $\LL:=\Lal-\alpha_{0,0} \dx^{0}\dy^{0}$, or equivalently
\begin{equation*}
\LL = \sum_{\ell = 1}^{\mfn} \sum_{k=0}^\ell \alpha_{k,\ell-k} \dx^{k}\dy^{\ell-k},
\end{equation*}
and its non-linear part $\LN:=\LA -\LL$ can equivalently be defined by
\begin{align*}
\LN = \sum_{\ell = 1}^{\mfn} \sum_{k=0}^\ell k!{(\ell-k)}! \alpha_{k,\ell-k}
\sum_{2\leq\mu\leq \ell} \sum_{s=1}^{\ell} \sum_{p_s((k,\ell-k),\mu)} \prod_{m=1}^s \frac{1}{k_m!}\left( \Dop{i_m}{j_m}  (\cdot) \right)^{k_m}.
\end{align*}
\end{dfn}

%% FOLLOW UP EXAMPLES FOR A  $\mathfrak n=2$ and a $\mathfrak n = 4$ CASE
%[EX] 
For the operator $\mathfrak L_{2,\gamma}$  this gives respectively
\begin{equation*}
\left\{\begin{array}{l}\displaystyle
\mathfrak L_{2,\gamma}^L P= -\dx^2 P + \gamma_{1,1}\dx\dy P + \gamma_{0,2}\dy^2 P
+ \gamma_{1,0}\dx P + \gamma_{0,1}\dy P,\\\displaystyle
\mathfrak L_{2,\gamma}^N P= 
-(\dx P)^2 + \gamma_{1,1}\dx P\dy P + \gamma_{0,2}(\dy P)^2.
\end{array}\right.
\end{equation*}

Consider the $(I,J)$ coefficients of the Taylor expansion of $\LL P$ for $(I,J)\in\mathbb N^2$ and $0\leq I+J <q$:
\begin{align*}
\Dop{I}{J} \left[\LL P\right] (x_0,y_0)
&= \sum_{\ell = 1}^{\mfn} \sum_{k=0}^\ell 
\Dop{I}{J} \left[\alpha_{k,\ell-k} \dx^{k}\dy^{\ell-k}P\right] (x_0,y_0),
\end{align*}
so that in order to isolate the derivatives of highest order, i.e. of order $\mfn+I+J$, we can write
\begin{equation}\label{eq:dxidyj}
\begin{array}{l}
\displaystyle
\Dop{I}{J} \left[\LL P\right] (x_0,y_0)
%\\\displaystyle
%\phantom == \frac{1}{I!J!} \sum_{k=0}^{\mfn} 
%\alpha_{k,{\mfn}-k} (x_0,y_0) \dx^{k+I}\dy^{{\mfn}-k+J}P (x_0,y_0)
%\\
%\phantom{==}
%\displaystyle+\frac{1}{I!J!} \sum_{k=0}^{\mfn} 
%\sum_{\tilde i=0}^{I-1}\sum_{\tilde j=0}^{J-1}
%\begin{pmatrix}
%I\\\tilde i
%\end{pmatrix}
%\begin{pmatrix}
%J\\\tilde j
%\end{pmatrix}
%\dx^{I-\tilde i}\dy^{J-\tilde j}\alpha_{k,{\mfn}-k} (x_0,y_0) \dx^{k+\tilde i}\dy^{{\mfn}-k+\tilde j}P (x_0,y_0)\\
%\phantom{==}
%\displaystyle+\frac{1}{I!J!}\sum_{\ell = 1}^{\mfn-1} \sum_{k=0}^\ell 
%\sum_{\tilde i=0}^{I}\sum_{\tilde j=0}^{J}
%\begin{pmatrix}
%I\\\tilde i
%\end{pmatrix}
%\begin{pmatrix}
%J\\\tilde j
%\end{pmatrix}
%\dx^{I-\tilde i}\dy^{J-\tilde j}\alpha_{k,\ell-k} (x_0,y_0) \dx^{k+\tilde i}\dy^{\ell-k+\tilde j}P (x_0,y_0).
\\\displaystyle
\phantom == \frac{1}{I!J!} \sum_{k=0}^{\mfn} 
\alpha_{k,{\mfn}-k} (x_0,y_0) \dx^{k+I}\dy^{{\mfn}-k+J}P (x_0,y_0)
\\
\phantom{==}
\displaystyle+ \sum_{k=0}^{\mfn} 
\sum_{\tilde i=0}^{I-1}\sum_{\tilde j=0}^{J-1}
\frac{1}{\tilde i!\tilde j!}
\Dop{I-\tilde i}{J-\tilde j}\alpha_{k,{\mfn}-k} (x_0,y_0) \dx^{k+\tilde i}\dy^{{\mfn}-k+\tilde j}P (x_0,y_0)\\
\phantom{==}
\displaystyle+\sum_{\ell = 1}^{\mfn-1} \sum_{k=0}^\ell 
\sum_{\tilde i=0}^{I}\sum_{\tilde j=0}^{J}
\frac{1}{\tilde i!\tilde j!}
\Dop{I-\tilde i}{J-\tilde j}\alpha_{k,\ell-k} (x_0,y_0) \dx^{k+\tilde i}\dy^{\ell-k+\tilde j}P (x_0,y_0).
\end{array}
\end{equation}
Back to Problem \eqref{thesystdeg}, the $(I,J)$ terms \eqref{eq:dxidyj} a priori depend on the unknowns $\{\lambda_{i,j},(i,j)\in\mathbb N^2,0\leq i+j\leq dP\}$. 
Since
$$
\forall (i,j)\in\mathbb N^2, 
\Dop{i}{j}P (x_0,y_0) = 
\left\{\begin{array}{ll}
\lambda_{i,j} &\text{ if } i+j\leq dP, \\
0  & \text{otherwise},
\end{array}\right.
$$
then under Hypothesis \ref{hyp} any $(I,J)$ term in System \eqref{thesystdeg} has at least one non-zero linear term, as long as $I+J\leq dP-\mfn$, namely 
$\frac{(\mfn +I)!}{I!}\alpha_{\mfn,0}(x_0,y_0)\lambda_{\mfn+I,J}$, while it does not necessarily have any linear term as soon as $I+J>dP-\mfn$.
Avoiding equations with no linear terms is natural, and it will be crucial for the construction process described hereafter.
 
Choosing the polynomial degree to be $dP=\mfn +q-1$ therefore guarantees the existence of at least one linear term in every equation of System \eqref{thesystdeg}. Therefore, from now on the polynomial $P$ will be of degree $dP=\mfn +q-1$ and the new problem to be solved is
\begin{equation}\label{thesyst}
\begin{array}{c}
\displaystyle
\text{Find } 
\{\lambda_{i,j},(i,j)\in\mathbb N^2,
                0\leq i+j\leq \mfn +q-1\} 
 \text{ such that }\hfill\\
\displaystyle
P(x,y):=\sum_{i=0}^{\mfn +q-1}\sum_{j=0}^{\mfn +q-1-i}
\lambda_{i,j} (x-x_0)^i(y-y_0)^j
\in\mathbb C[x,y], \text{ satisfies }\\
\displaystyle
\forall (I,J)\in\mathbb N^2, 0\leq I+J <q,
\Dop{I}{J} \LA P (x_0,y_0)
=
-\Dop{I}{J} \alpha_{0,0} (x_0,y_0).
\end{array}
\end{equation}
As a consequence the number of unknowns is $N_{dof} = \frac{(\mfn+q)(\mfn+q+1)}{2}$, and the system is under-determined : $N_{dof}-N_{eqn} = \mfn q+ \frac{\mfn(\mfn+1)}{2}$. See Figure \ref{fig:indCOMB} for an illustration of the equation and unknown count.
%HERE FIG tikzindq tikzindP
\begin{figure}
\begin{center}
\includegraphics[width=.45\textwidth,trim={5cm 12cm 2cm 4cm},clip]{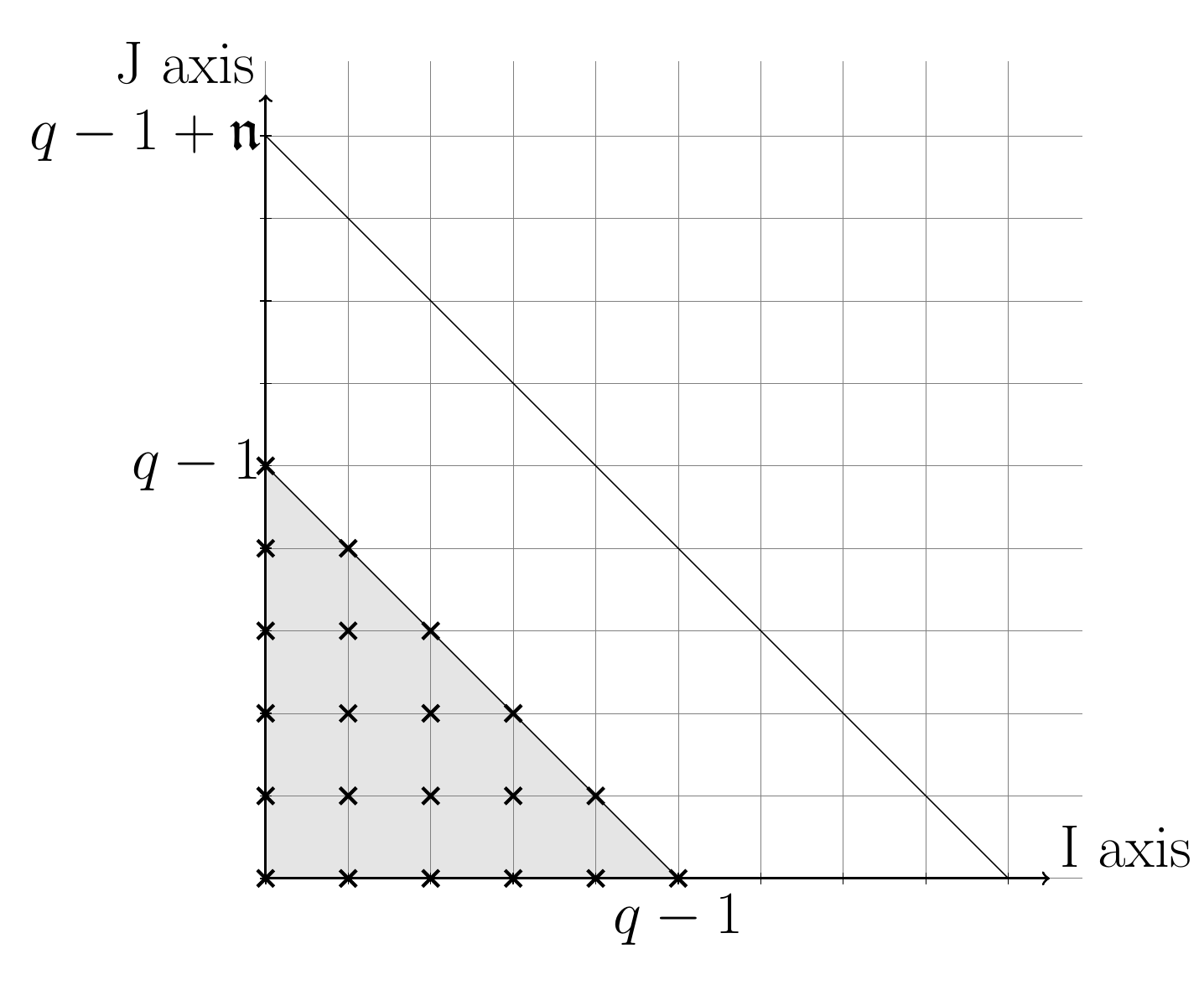}
\includegraphics[width=.45\textwidth,trim={5cm 12cm 2cm 4cm},clip]{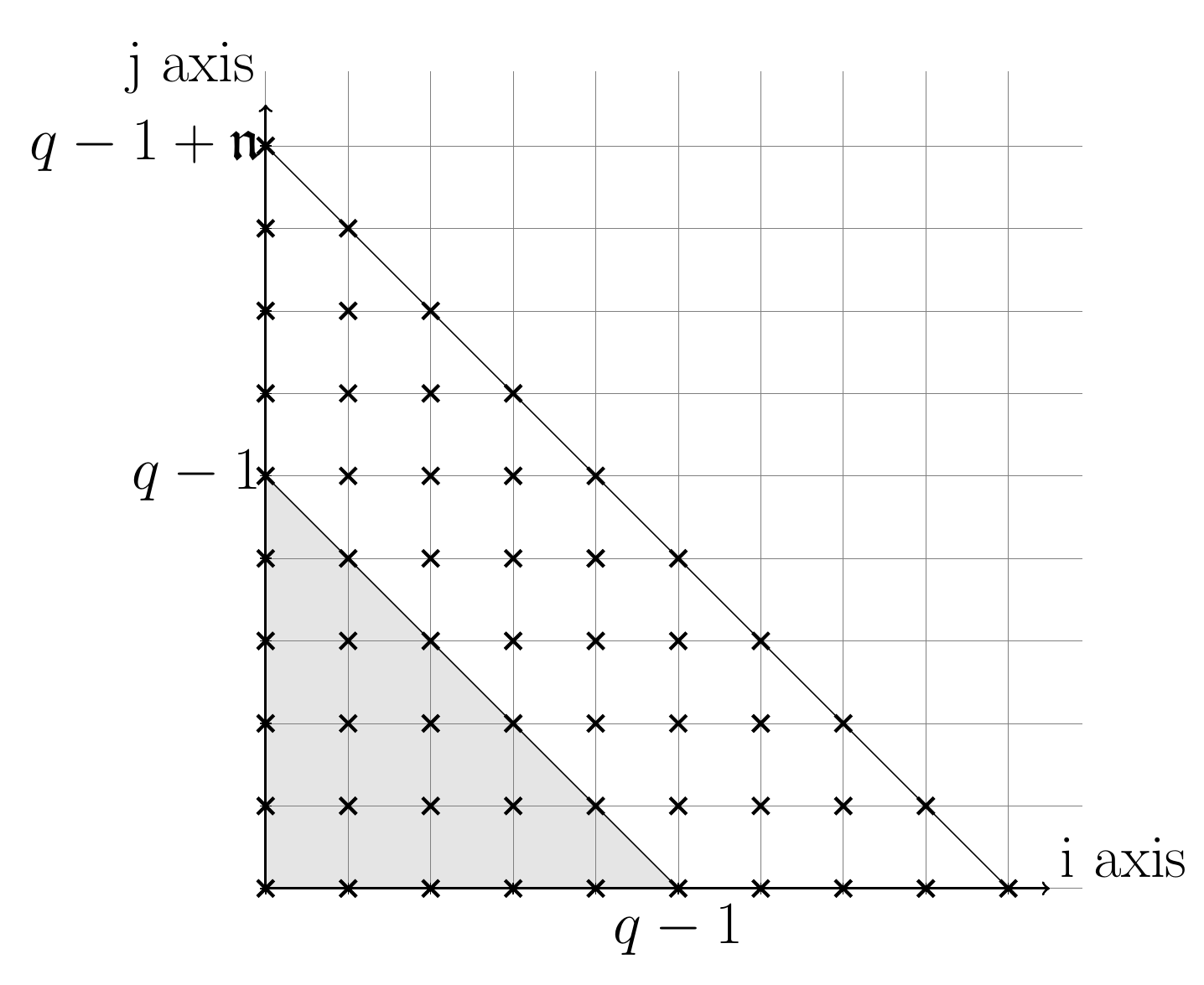}
\end{center}
\caption{Representation of the indices involved in the nonlinear system \eqref{thesyst}, for $q=6$ and $\mathfrak n=4$. Each cross in the $(I,J)$ plane corresponds to the equation $(I,J)$ in System \eqref{thesyst} (Left panel), while each cross in the $(i,j)$ plane corresponds to the unknown $\luij$ (Right panel).
% the indices $(I,J)$ correspond to the system's equations (Left panel), while the indices $(i,j)$ correspond to the system's unknowns.
%In the $(i,j)$ plane, for $q=6$ and $\mathfrak n=4$, cross markers represent the indices of the unknowns $\lambda_{i,j}$ while circle markers represent the indices of the $(i,j)$ equation.
}
\label{fig:indCOMB}
\end{figure}
 
Note that this system is always non-linear. Indeed, under Hypothesis \ref{hyp}, the $(0,0)$ equation of the system always includes the non-zero non-linear term $\alpha_{\mfn,0}(x_0,y_0)(\lambda_{1,0})^{\mfn}$, corresponding to the following parameters from the Faa di Bruno formula: $\mu=\mfn$, $s=1$, $(k_1,(i_1,j_1)) = (\mfn,(1,0))$.

The key to the construction procedure proposed next is a meticulous gathering of unknowns $\lambda_{i,j}$ with respect the length of their multi-index $i+j$. As we will now see, this will lead to splitting the system into a hierarchy of simple linear sub-systems.

\subsection{From a non-linear system to linear sub-systems}\label{ssec:NL2L}
%\subsection{Inspection of the linear and non-linear terms}
The different unknowns appearing in each equation of System \eqref{thesyst} can now be studied. A careful inspection of the linear and non-linear terms will reveal the underlying structure of the system, and will lead to identify a hierarchy of simple linear subsystems.

The inspection of the linear terms is very straightforward thanks to Equation \eqref{eq:dxidyj}. 
The description of the unknowns in the linear terms is summarized here.
\begin{lmm}\label{lm:lterms}
\defOp
In each equation $(I,J)$ of System \eqref{thesyst}, the linear terms can be split as follows:
\begin{itemize}
\item a set of unknowns with  length of the multi-index equal to $\mfn+I+J$, corresponding to $\ell=\mfn$ and $(\tilde{i},\tilde{j})=(I,J)$,
\item a set of unknowns with  length of the multi-index at most equal to  $\mfn+I+J-1$.
\end{itemize}
Under Hypothesis \ref{hyp}, both sets are never empty.
\end{lmm}

%% FOLLOW UP EXAMPLES FOR A  $\mathfrak n=2$ and a $\mathfrak n = 4$ CASE
 
\begin{proof}
In terms of unknowns $\{\lambda_{i,j},(i,j)\in\mathbb N^2,
                0\leq i+j\leq \mfn +q-1\} $, Equation \eqref{eq:dxidyj} reads : 
                
%for $I=J=0$ 
\begin{align}\label{eq:lin1}
\begin{split}
\dx^0\dy^0&\left[\LL P\right] (x_0,y_0)\\ = 
&\sum_{k=0}^{\mfn} (k)!({\mfn}-k)!
\alpha_{k,{\mfn}-k} (x_0,y_0) \lambda_{k,{\mfn}-k}\\
&+\sum_{\ell = 1}^{\mfn-1} \sum_{k=0}^\ell 
 \left(k\right)!\left({\ell}-k\right)!
\alpha_{k,\ell-k} (x_0,y_0)  \lambda_{k,{\ell}-k};
\end{split}
\end{align}
                
%for $I=0$ and $J>0$ 
\begin{align}\label{eq:lin2}
\begin{split}
\forall J>0,\
&\Dop{0}{J}\left[\LL P\right] (x_0,y_0)\\
&= \frac{1}{J!} \sum_{k=0}^{\mfn} k!({\mfn}-k+J)!
\alpha_{k,{\mfn}-k} (x_0,y_0) \lambda_{k,{\mfn}-k+J}
\\
&\phantom =+ \sum_{k=0}^{\mfn} 
\sum_{\tilde j=0}^{J-1}
 k!\frac{\left({\mfn}-k+\tilde j\right)!}{\tilde j!}
\Dop{0}{J-\tilde j}\alpha_{k,{\mfn}-k} (x_0,y_0) \lambda_{k,{\mfn}-k+\tilde j}\\
&\phantom =+\sum_{\ell = 1}^{\mfn-1} \sum_{k=0}^\ell 
\sum_{\tilde j=0}^{J}
 k!\frac{\left({\ell}-k+\tilde j\right)!}{\tilde j!}
\Dop{0}{J-\tilde j}\alpha_{k,\ell-k} (x_0,y_0)  \lambda_{k,{\ell}-k+\tilde j} ; 
\end{split}
\end{align}
                
%for $I> 0$ and $J=0$  
\begin{align}\label{eq:lin3}
\begin{split}
\forall I>0,\
&\Dop{I}{0}\left[\LL P\right] (x_0,y_0)\\
&= \frac{1}{I!} \sum_{k=0}^{\mfn} (k+I)!({\mfn}-k)!
\alpha_{k,{\mfn}-k} (x_0,y_0) \lambda_{k+I,{\mfn}-k}
\\
&\phantom =+ \sum_{k=0}^{\mfn} 
\sum_{\tilde i=0}^{I-1}
\frac{ \left(k+\tilde i\right)!}{\tilde i !}\left({\mfn}-k\right)!
\Dop{I-\tilde i}{0}\alpha_{k,{\mfn}-k} (x_0,y_0) \lambda_{k+\tilde i,{\mfn}-k}\\
&\phantom =+\sum_{\ell = 1}^{\mfn-1} \sum_{k=0}^\ell 
\sum_{\tilde i=0}^{I}
\frac{ \left(k+\tilde i\right)!}{\tilde i!}\left({\ell}-k\right)!
\Dop{I-\tilde i}{0}\alpha_{k,\ell-k} (x_0,y_0)  \lambda_{k+\tilde i,{\ell}-k} ;
\end{split}
\end{align} 
                
%and finally for $I>0$ and $J>0$  
\begin{align}\label{eq:lin4}
\begin{split}
\forall (I,J), IJ\neq 0,\ &\Dop{I}{J}\left[\LL P\right] (x_0,y_0)\\
&= \frac{1}{I!J!} \sum_{k=0}^{\mfn} (k+I)!({\mfn}-k+J)!
\alpha_{k,{\mfn}-k} (x_0,y_0) \lambda_{k+I,{\mfn}-k+J}
\\
&\phantom =+ \sum_{k=0}^{\mfn} 
\sum_{\tilde i=0}^{I-1}\sum_{\tilde j=0}^{J-1}
\frac{ \left(k+\tilde i\right)!}{\tilde i!}\frac{\left({\mfn}-k+\tilde j\right)!}{\tilde j!}
\Dop{I-\tilde i}{J-\tilde j}\alpha_{k,{\mfn}-k} (x_0,y_0) \lambda_{k+\tilde i,{\mfn}-k+\tilde j}\\
&\phantom =+\sum_{\ell = 1}^{\mfn-1} \sum_{k=0}^\ell 
\sum_{\tilde i=0}^{I}\sum_{\tilde j=0}^{J}
\frac{ \left(k+\tilde i\right)!}{\tilde i!}\frac{\left({\ell}-k+\tilde j\right)!}{\tilde j!}
\Dop{I-\tilde i}{J-\tilde j}\alpha_{k,\ell-k} (x_0,y_0)  \lambda_{k+\tilde i,{\ell}-k+\tilde j}.
\end{split}
\end{align}
The result is immediate for $I=J=0$ from \eqref{eq:lin1}. The following comments are valid for the right hand sides of \eqref{eq:lin2}, \eqref{eq:lin3}, and  \eqref{eq:lin4}:
the third term only contains unknowns with a length of the multi-index equal to $\ell+\tilde i+\tilde j\leq \mfn -1 +I+J$, while 
the second term only contains unknowns with a length of the multi-index equal to $\mfn+\tilde i+\tilde j\leq \mfn +I+J-2$ ; as to
the first term, it only contains unknowns with a length of the multi-index equal to $\mfn+I+J$. This proves the claim.
\end{proof}

We then focus on the inspection of the non-linear terms. Each non-linear term in $\LA P$ reads from the definition of $\LN $
\begin{equation}\label{eq:NLtOp}
\alpha_{k,\ell-k} \prod_{m=1}^s \left( \dx^{i_m}\dy^{j_m}  P \right)^{k_m}\text{ with } \sum_{m=1}^s k_m>1
\end{equation}
 and yields a sum of non-linear terms with respect to the unknowns $\{\luij\}_{\{(i,j),0\leq i+j\leq \mfn+q-1 \}}$, implicitly given by the following formula:
\begin{align}\label{eq:NLt0}
\begin{split}
&\Dop{I}{J}\left[\alpha_{k,\ell-k} \prod_{m=1}^s \left( \dx^{i_m}\dy^{j_m} P \right)^{k_m}\right](x_0,y_0)
\\&= 
\sum_{\tilde i=0}^{I}\sum_{\tilde j=0}^{J}
\Dop{I-\tilde i}{J-\tilde j}\alpha_{k,\ell-k}(x_0,y_0)
\Dop{\tilde i}{\tilde j}\left[\prod_{m=1}^s \left( \dx^{i_m}\dy^{j_m} P \right)^{k_m}\right](x_0,y_0).
\end{split}
\end{align}
Therefore coming from the term \eqref{eq:NLtOp}, only a restricted number of unknowns contribute to the $(I,J)$ equation of Problem \eqref{thesyst}.

In order to identify the unknowns contributing to \eqref{eq:NLt0}, here are two simple yet important reminders are provided in Appendix \ref{PolFor}.

It is now straightforward to describe the unknowns $\lambda_{i,j}$ appearing in the non-linear terms of the equation $(I,J)$ of System \eqref{thesyst}, unwinding formula \eqref{eq:NLt0}. 
\begin{lmm}\label{lm:nlterms}
\defOp
In each equation $(I,J)$ of System \eqref{thesyst}, the unknowns $\lambda_{i,j}$ appearing in the non-linear terms have a length of the multi-index $i+j<\mfn+I+J$.
\end{lmm}

%%FOLLOW UP EXAMPLES FOR A  $\mathfrak n=2$ and a $\mathfrak n = 4$ CASE
 
\begin{proof}
Each term $\dx^{\tilde i}\dy^{\tilde j}\left[\prod_{m=1}^s \left( \dx^{i_m}\dy^{j_m} P \right)^{k_m}\right]$ in $\LA P$ is a polynomial, and its constant coefficient contains  coefficients of the polynomial $\prod_{m=1}^s \left( \dx^{i_m}\dy^{j_m} P \right)^{k_m}$ with a length of the multi-index length of the multi-index at most equal to $\tilde i+\tilde j$,
that is to say coefficients of the polynomials $\dx^{i_m}\dy^{j_m} P$ with a length of the multi-index length of the multi-index  at most equal to $\tilde i+\tilde j$ for every $(i_m,j_m)$ from the Faa di Bruno's formula,
so coefficients $\lambda_{i,j}$ of the polynomial $P$ with a length of the multi-index  at most equal to $\tilde i+\tilde j+i_m+j_m$.
%\begin{itemize}
%\item coefficients of the polynomial $\prod_{m=1}^s \left( \dx^{i_m}\dy^{j_m} P \right)^{k_m}$ of \totind at most equal to $\tilde i+\tilde j$,
%\item so coefficients of the polynomials $\dx^{i_m}\dy^{j_m} P$ of \totind  at most equal to $\tilde i+\tilde j$ for every $(i_m,j_m)$ from the Faa di Bruno's formula,
%\item so coefficients $\lambda_{i,j}$ of the polynomial $P$ of \totind  at most equal to $\tilde i+\tilde j+i_m+j_m$.
%\end{itemize}
Since the indices are such that $\tilde i\leq I$, $\tilde j\leq J$, and $i_m+j_m\leq \ell<\mfn$, the unknowns $\lambda_{i,j}$ appearing in each term $\dx^{\tilde i}\dy^{\tilde j}\left[\prod_{m=1}^s \left( \dx^{i_m}\dy^{j_m} P \right)^{k_m}\right](x_0,y_0)$ have a length of the multi-index at most equal to $\mfn+I+J-1$. It is therefore true for any linear combination such as \eqref{eq:NLt0}.
\end{proof}

From the two previous Lemmas, we see that, in each equation $(I,J)$ of System \eqref{thesyst}, unknowns with a length of the multi-index equal to $\mfn+I+J$ appear only in linear terms, namely in
$$
\sum_{k=0}^n 
\frac{(k+I)!}{I!}\frac{(\mfn-k+J)!}{J!}\alpha_{k,\mfn-k}(x_0,y_0)\lambda_{k+I,\mfn-k+J},
$$
 whereas all the remaining unknowns have a length of the multi-index at most equal to $\mfn+I+J-1$.  It is consequently natural to subdivide the set of unknowns with respect to the length of their multi-index $\mfn+\mfl$, for $\mfl$ between $0$ and $q-1$ in order to take advantage of this linear structure. %Indeed, for any $\mfl$ between $0$ and $q-1$, the sub-system formed by the equations $(I,\mfl-I)$ for I between $0$ and $\mfl$

\subsection{Hierarchy of triangular linear systems}\label{ssec:hier}
%\subsection{Hierarchy of linear subsystems}
Our goal is now to construct a solution to the non-linear system \eqref{thesyst}, and our understanding of its linear part will lead to an explicit construction of such a solution without any need for any approximation.

 The crucial point of our construction process is to take advantage of the underlying layer structure with respect to the length of the multi-index: it is only natural now to gather into subsystems all equations $(I,\mfl-I)$ for $I$ between $0$ and $\mfl$, while gathering similarly all unknowns with length of the multi-index equal to $\mfn+\mfl$. 
 %In the subsystem of layer $\mfl$, we know that the unknowns of \totind equal to $\mfn+I+J$ only appear in linear terms, whereas all the remaining unknowns have a \totind at most equal to $\mfn+I+J-1$.  Therefore, we now rewrite each equation $(I,J)$ as
 In the subsystem of layer $\mfl$, we know that the unknowns with a length of the multi-index equal to $\mfn+I+J$ only appear in linear terms, and we rewrite each equation $(I,J)$ as
 $$
 \begin{array}{l}
 \displaystyle
\sum_{k=0}^n 
\frac{(k+I)!}{I!}\frac{(\mfn-k+\mfl-I)!}{(\mfl-I)!}\alpha_{k,\mfn-k}(x_0,y_0)\lambda_{k+I,\mfn-k+\mfl-I}\\
 \displaystyle
=-\Dop{I}{J} \alpha_{0,0} (x_0,y_0)
- \Dop{I}{J} \LA P (x_0,y_0)
+\sum_{k=0}^n 
\frac{(k+I)!}{I!}\frac{(\mfn-k+\mfl-I)!}{(\mfl-I)!}\alpha_{k,\mfn-k}(x_0,y_0)\lambda_{k+I,\mfn-k+\mfl-I}.
\end{array}
 $$

 For the sake of clarity, the resulting right-hand side terms can defined as follows.
\begin{dfn}\label{dfn:\RHS}
\defOp
We define the quantity $\RHS_{I,J}$ from Equation $(I,J)$ from \eqref{thesyst} as
%   
%for $I=J=0$ 
\begin{align}%\label{eq:lin1}
\begin{split}
\RHS_{0,0}&:=-\sum_{\ell = 1}^{\mfn-1} \sum_{k=0}^\ell 
 \left(k\right)!\left({\ell}-k\right)!
\alpha_{k,\ell-k} (x_0,y_0)  \lambda_{k,{\ell}-k}
\\
&\phantom = - \LN P (x_0,y_0) - \alpha_{0,0}(x_0,y_0);
\end{split}
\end{align}
%                
%for $I=0$ and $J>0$ 
\begin{align}%\label{eq:lin2}
\begin{split}
\forall J>0,\
\RHS_{0,J}
&:=\sum_{k=0}^{\mfn} 
\sum_{\tilde j=0}^{J-1}
 \left(k+\tilde i\right)!\frac{\left({\mfn}-k+\tilde j\right)!}{\tilde j!}
\Dop{0}{J-\tilde j}\alpha_{k,{\mfn}-k} (x_0,y_0) \lambda_{k,{\mfn}-k+\tilde j}\\
&\phantom =+\sum_{\ell = 1}^{\mfn-1} \sum_{k=0}^\ell 
\sum_{\tilde j=0}^{J}
 \left(k\right)!\frac{\left({\ell}-k+\tilde j\right)!}{\tilde j!}
\Dop{0}{J-\tilde j}\alpha_{k,\ell-k} (x_0,y_0)  \lambda_{k,{\ell}-k+\tilde j} 
\\
&\phantom = - \Dop{0}{J}\left[ \LN P \right](x_0,y_0) - \Dop{0}{J}\alpha_{0,0}(x_0,y_0); 
\end{split}
\end{align}
 %            
%for $I> 0$ and $J=0$  
\begin{align}%\label{eq:lin3}
\begin{split}
\forall I>0,\
\RHS_{I,0}
&:=\sum_{k=0}^{\mfn} 
\sum_{\tilde i=0}^{I-1}
 \frac{\left(k+\tilde i\right)!}{\tilde i!}\left({\mfn}-k\right)!
\Dop{I-\tilde i}\alpha_{k,{\mfn}-k} (x_0,y_0) \lambda_{k+\tilde i,{\mfn}-k}\\
&\phantom =+\sum_{\ell = 1}^{\mfn-1} \sum_{k=0}^\ell 
\sum_{\tilde i=0}^{I}
\frac{ \left(k+\tilde i\right)!}{\tilde i!}\left({\ell}-k\right)!
\Dop{I-\tilde i}{0}\alpha_{k,\ell-k} (x_0,y_0)  \lambda_{k+\tilde i,{\ell}-k} 
\\
&\phantom = - \Dop{I}{0}\left[ \LN P \right](x_0,y_0) - \Dop{I}{0}\alpha_{0,0}(x_0,y_0);
\end{split}
\end{align} 
%         
%and finally for $IJ\neq 0$  
\begin{align}
\begin{split}
\forall (I,J),\ IJ\neq0,\ &
\RHS_{I,J}\\&:=-\sum_{k=0}^{\mfn} 
\sum_{\tilde i=0}^{I-1}\sum_{\tilde j=0}^{J-1}
 \frac{\left(k+\tilde i\right)!\left({\mfn}-k+\tilde j\right)!}{\tilde i!\tilde j !}
\Dop{I-\tilde i}{J-\tilde j}\alpha_{k,{\mfn}-k} (x_0,y_0) \lambda_{k+\tilde i,{\mfn}-k+\tilde j}\\
&\phantom =-\sum_{\ell = 1}^{\mfn-1} \sum_{k=0}^\ell 
\sum_{\tilde i=0}^{I}\sum_{\tilde j=0}^{J}
 \frac{\left(k+\tilde i\right)!\left({\ell}-k+\tilde j\right)!}{\tilde i !\tilde j!}
\Dop{I-\tilde i}{J-\tilde j}\alpha_{k,\ell-k} (x_0,y_0)  \lambda_{k+\tilde i,{\ell}-k+\tilde j}
\\
&\phantom = -\Dop{I}{J}\left[ \LN P \right](x_0,y_0) - \Dop{I}{J}\alpha_{0,0}(x_0,y_0).
\end{split}
\end{align}
\end{dfn}

%% FOLLOW UP EXAMPLES FOR A  $\mathfrak n=2$ and a $\mathfrak n = 4$ CASE
[EX] 
For the operator $\mathfrak L_{2,\gamma}$  the non-linear terms in $\RHS_{0,0}$, $\RHS_{1,0}$ and $\RHS_{0,1}$  are respectively
\begin{equation*}
\mathfrak L_{2,\gamma}^N P(x_0,y_0)= 
-\lambda_{1,0}^2
+\gamma_{1,1}(x_0,y_0)\lambda_{1,0}\lambda_{0,1}
+\gamma_{0,2}(x_0,y_0)\lambda_{0,1}^2,
\end{equation*}
\begin{equation*}
\begin{array}{rl}
\dx [\mathfrak L_{2,\gamma}^N P](x_0,y_0)=&
-2\lambda_{2,0}\lambda_{1,0}
+\gamma_{1,1}(x_0,y_0)\left( 2\lambda_{2,0}\lambda_{0,1}+\lambda_{1,0}\lambda_{1,1} \right)
+2\gamma_{0,2}(x_0,y_0)\lambda_{1,1}\lambda_{0,1}
\\
&
+\dx \gamma_{1,1}(x_0,y_0)\lambda_{1,0}\lambda_{0,1}
+\dx \gamma_{0,2}(x_0,y_0)\lambda_{0,1}^2,
\end{array}
\end{equation*}
\begin{equation*}
\begin{array}{rl}
\dy [\mathfrak L_{2,\gamma}^N P](x_0,y_0)=&
-2\lambda_{1,1}\lambda_{1,0}
+\gamma_{1,1}(x_0,y_0)\left( \lambda_{1,1}\lambda_{0,1}+2\lambda_{1,0}\lambda_{2,0} \right)
+2\gamma_{0,2}(x_0,y_0)\lambda_{0,2}\lambda_{0,1}
\\
&+\dy \gamma_{1,1}(x_0,y_0)\lambda_{1,0}\lambda_{0,1}
+\dy \gamma_{0,2}(x_0,y_0)\lambda_{0,1}^2.
\end{array}
\end{equation*}

We now consider the following subsystems for given $\mfl$ between $0$ and $q-1$:
\begin{equation}\label{ssyst}
\begin{array}{c}
\displaystyle
\text{Find } 
\{\lambda_{i,j},(i,j)\in\mathbb N^2,
                i+j=\mfn +\mfl\} 
 \text{ such that }\hfill\\
\displaystyle
\forall (I,J)\in\mathbb N^2, I+J=\mfl,
\sum_{k=0}^{\mfn} \frac{(k+I)!({\mfn}-k+J)!}{I!J!} 
\alpha_{k,{\mfn}-k} (x_0,y_0) \lambda_{k+I,{\mfn}-k+J}
=\RHS_{I,J}.
\end{array}
\end{equation}
The layer structure follows from our understanding of the non-linearity of the original system:
\begin{cor}\label{cor:RIJ}
\defOp
For any $(I,J)\in\mathbb N^2$ such that $I+J<q$, the quantity $\RHS_{I,J}$ only depends on unknowns $\lambda_{i,j}$ with length of the multi-index at most equal to $\mfn+I+J-1$.
\end{cor}
 
\begin{proof}
The result is straightforward from Lemmas \ref{lm:lterms} and \ref{lm:nlterms}.
\end{proof}

Assuming that all unknowns  $\lambda_{i,j}$ with length of the multi-index at most equal to $\mfn+I+J-1$ are known, then \eqref{ssyst} is a well-defined linear under-determined system with
 \begin{itemize}
 \item $\mfl$ linear equations, namely the $(I,J) =(I,\mfl-I)$ equations from System \eqref{thesyst} for $I$ between $0$ and $\mfl$;
 \item $\mfn+\mfl+1$ unknowns, namely the $\luij$  for $i+j=\mfn+\mfl$.
 \end{itemize}
Therefore, if  all unknowns  $\lambda_{i,j}$ with length of the multi-index at most equal to $\mfn+I+J-1$ are known, we expect to be able to compute a solution to the subsystem $\mfl$ ; this is the layer structure of our original problem \eqref{thesyst}. Figure \ref{fig:indL} highlights the link between the layers of unknowns and equations of the initial nonlinear system on the one hand, and  the layers unknowns and equations of the linear subsystems on the other hand. 
%HERE FIG tikzinds
\begin{figure}
\begin{center}
\includegraphics[width=.45\textwidth,trim={5cm 12cm 2cm 4cm},clip]{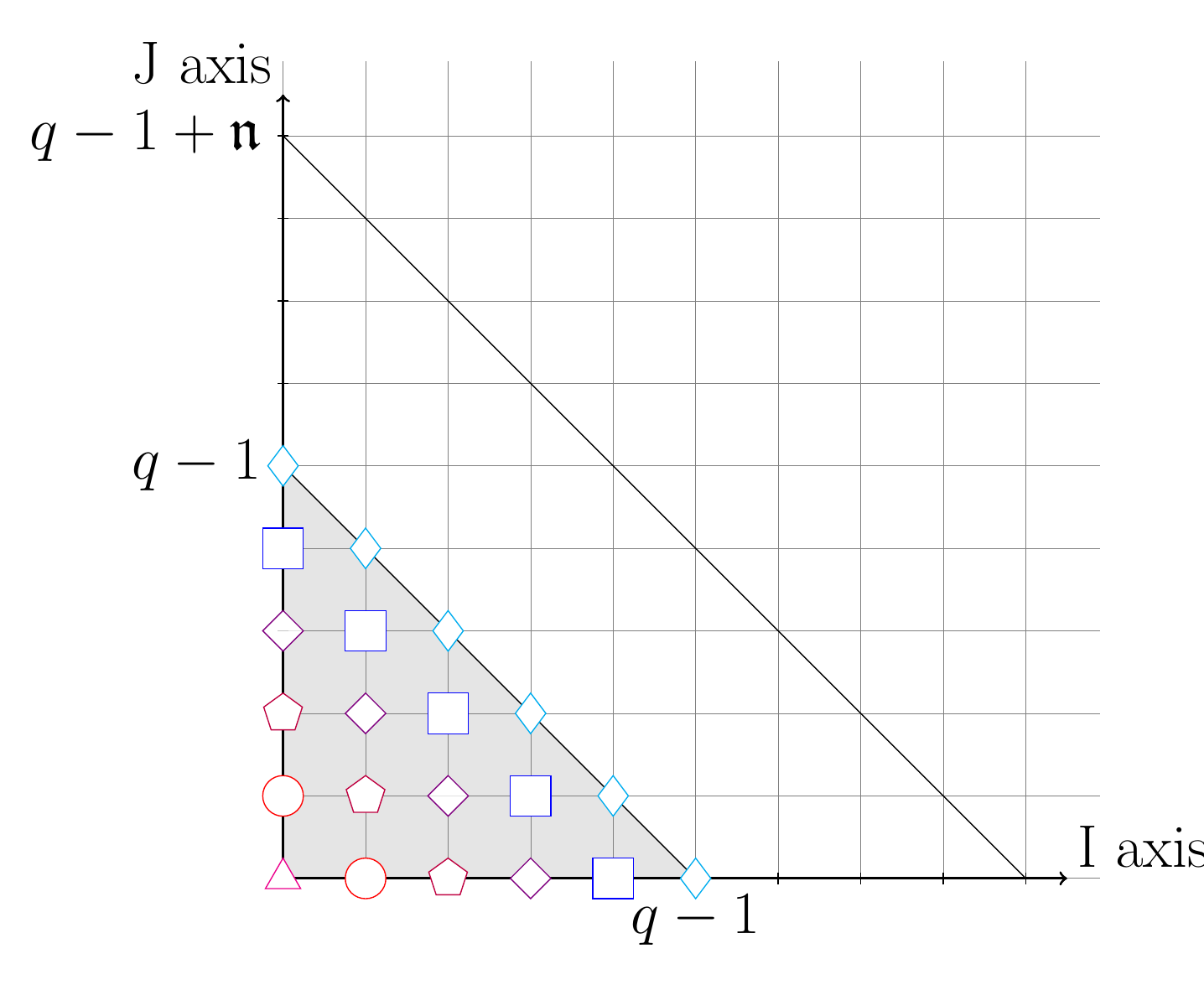}
\includegraphics[width=.45\textwidth,trim={5cm 12cm 2cm 4cm},clip]{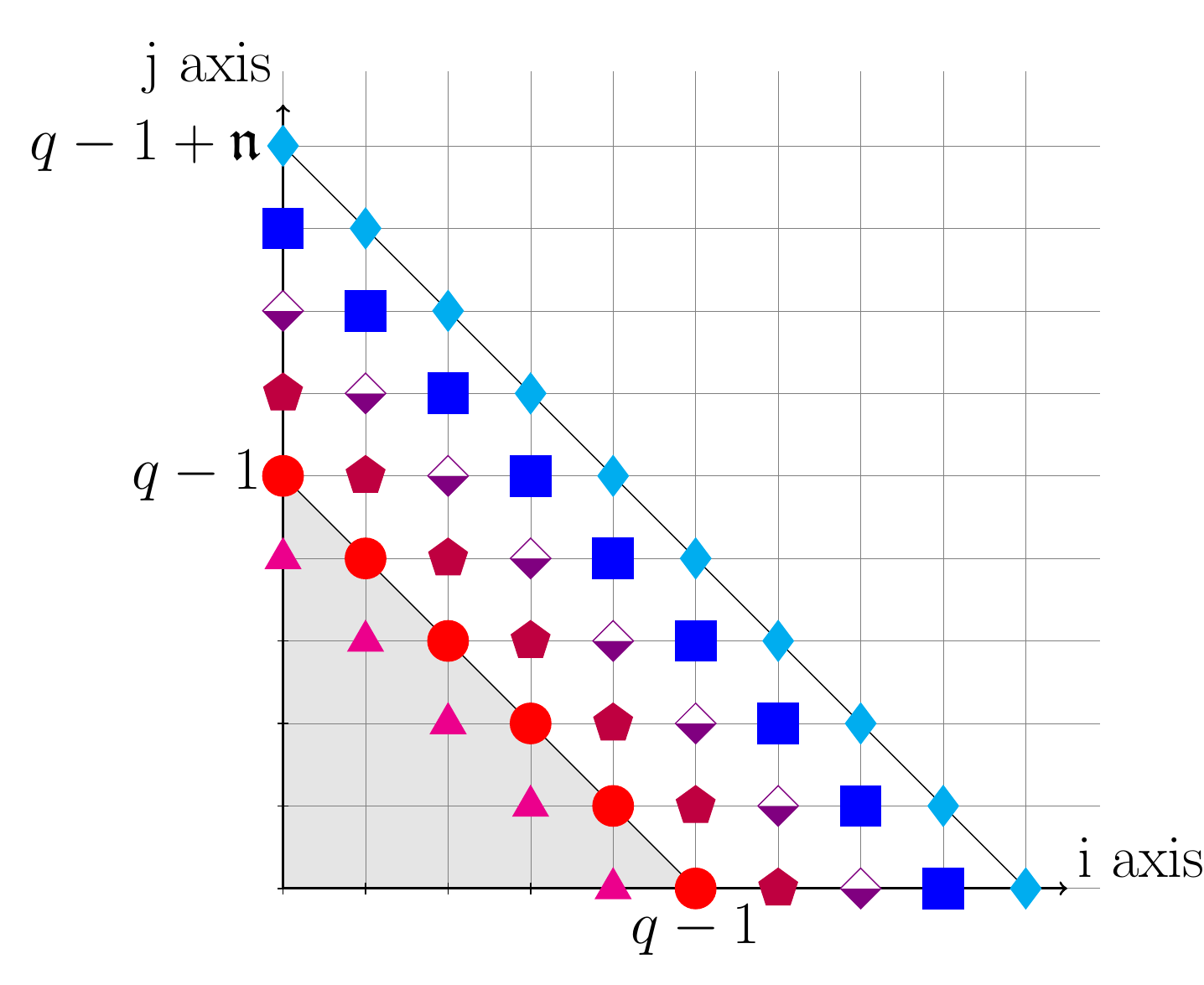}
\end{center}
\caption{Representation of the indices of equations and unknowns from the initial nonlinear system \eqref{thesyst} divided up into linear subsystems \eqref{ssyst}. 
%The numbering scheme is inherited from the initial nonlinear system. 
For $q=6$ and $\mfn=4$, each shape of marker corresponds to one value of $\mfl$: the indices $(I,J)$ satisfying $I+J=\mfl$ correspond to the subsystem's equations (Left panel), while the indices  $(i,j)$ satisfying $i+j=\mfl+\mfn$ correspond to the subsystem's unknowns (Right panel).}
\label{fig:indL}
\end{figure}

At this stage, we have identified a hierarchy of under-determined linear subsystems, for increasing values of $\mfl$ from $0$ to $q-1$, and we are now going to propose one procedure to build a solution to each subsystem. There is no unique way to do so, however if either $\alpha_{\mfn,0}(x_0,y_0)\neq 0$ or $\alpha_{0,\mfn}(x_0,y_0)\neq 0$ it provides a natural way to proceed. Indeed, the unknowns involved in an equation $(I,J)=(I,\mfl-I)$ are $\{\lambda_{i,\mfn+\mfl-i} ; i\in\mathbb N, I\leq i \leq I+\mfn  \}$ ; and the coefficient of the unknown $\lambda_{I+\mfn,\mfl-I}$ is proportional to $\alpha_{\mfn,0}(x_0,y_0)$, which is non-zero under Hypothesis \ref{hyp}. Figure \ref{fig:indLex} provides two examples, in the $(i,j)$ plane, of the indices of one equation's unknowns: for each equation, the coefficient of the term corresponding to the rightmost marker is non-zero.
\begin{figure}
\begin{center}
\includegraphics[width=.45\textwidth,trim={5cm 12cm 2cm 4cm},clip]{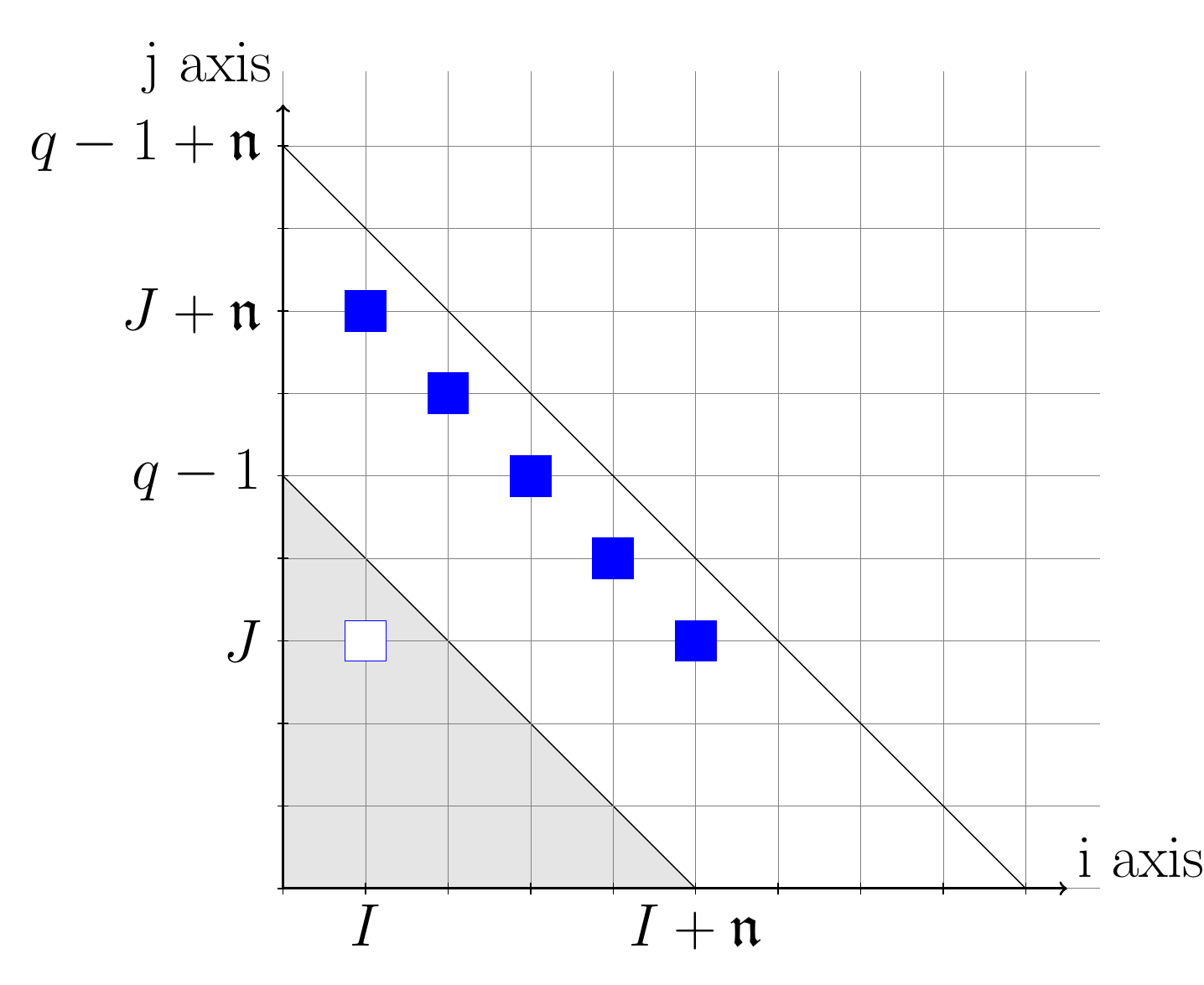}
\includegraphics[width=.45\textwidth,trim={5cm 12cm 2cm 4cm},clip]{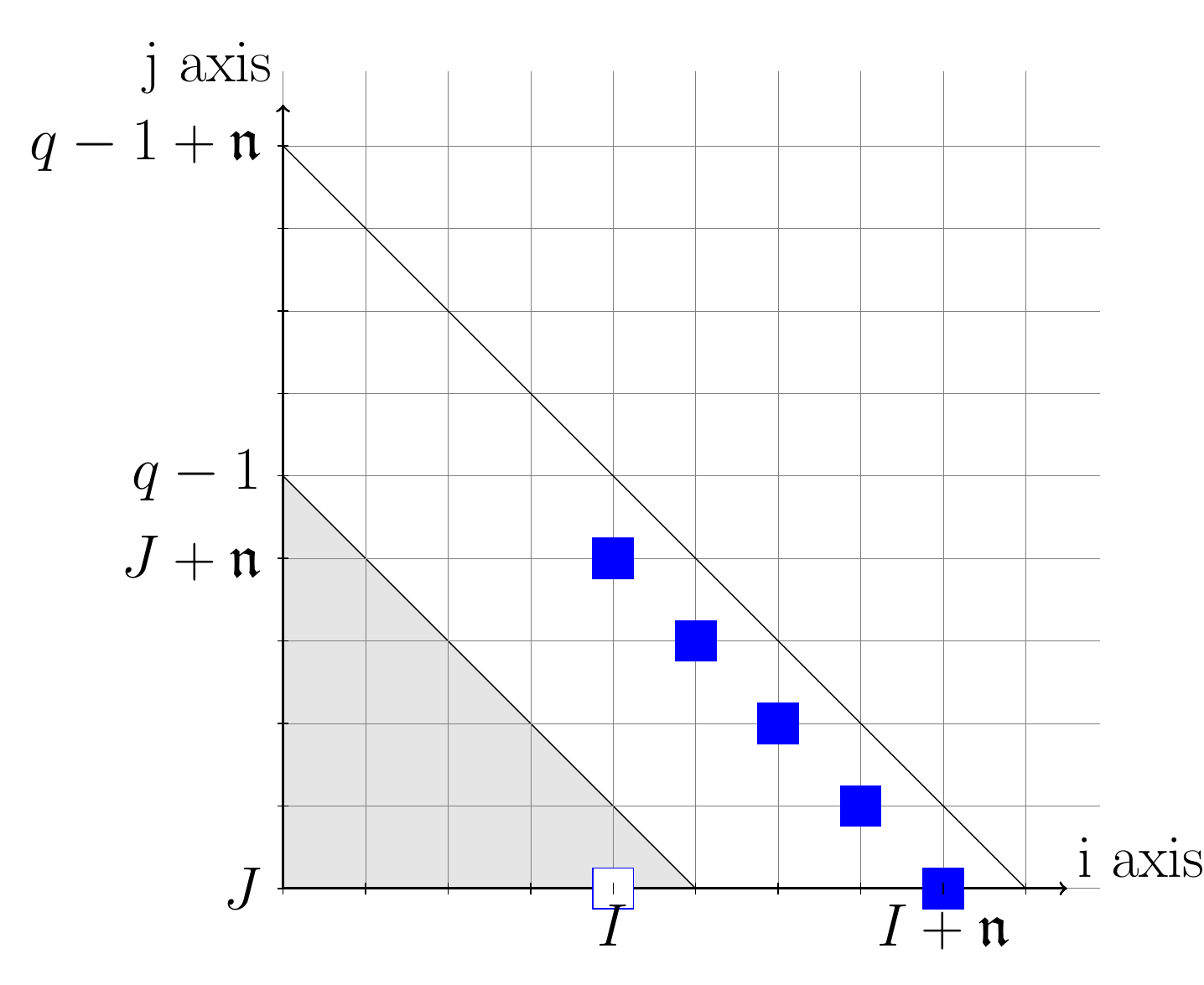}
\end{center}
\caption{Representation of the indices of unknowns involved in two equations $(I,J)$ of the subsystem \eqref{ssyst}. For $q=6$, for $\mfn = 4$, and $\mfl=4$, each  filled blue square marker corresponds  in the $(i,j)$ plane to an unknown $\luij$, involved in the $(I,J) = (1,3)$ equation (Left panel), or in the $(I,J) = (4,0)$ equation (Right panel).}
\label{fig:indLex}
\end{figure}
By adding $\mfn$ constraints corresponding to fixing the values of $\lambda_{i,\mfn+\mfl-i}$ for $0\leq i <\mfn$, that is the unknowns corresponding in the $(i,j)$ plane to first $\mfn$ markers on the left at level $\mfn+\mfl$, we therefore guarantee that for increasing values of $I$ from $0$ to $\mfl$ we can compute successively $\lambda_{I+\mfn,\mfl-I}$.

We can easily recast this in terms of matrices. At each level $\mfl$, numbering the equations with increasing values of $I$ and the unknowns with increasing values of $i$ highlights the band-limited structure of each subsystem, while the entries of the $\mfn$th super diagonal are all proportional to $\alpha_{\mfn,0}(x_0,y_0)$, and therefore non-zero under Hypothesis \ref{hyp}. The matrix of the square linear system at level $\mfl$ is then constructed from the first $\mfn$ lines of the identity, corresponding to the additional $\mfn$ constraints, placed on top of the matrix of the subsystem. 

\begin{dfn}
\defOp
For a given level $\mfl\in\mathbb N$ with $\mfl<q$, we define the matrix of the square system of level $\mfl$,  $\mathsf T^\mfl\in\mathbb C^{(\mfn +\mfl+1)\times(\mfn +\mfl+1)}$, as
\begin{equation*}
\left\{
\begin{array}{lll}
{ \mathsf{T}}_{k+1,k+1}^\mfl &= 1,&\forall k \ s.t. \ 0\leq k\leq \mfn-1,\\
{ \mathsf{T}}_{I+\mfn+1,I+k+1}^\mfl &\displaystyle = 
\frac{ (I+k)!(\mfn-k+\mfl-I)!}{I!(\mfl-I)!}
\alpha_{k,\mfn-k}(x_0,y_0),&\forall (k,I) \ s.t. \ 0\leq k\leq \mfn, \ 0\leq I\leq \mfl,\\
\mathsf T^\mfl_{k,k'} & =0,&\text{otherwise},
\end{array}
\right.
\end{equation*}
or equivalently
$$
\mathsf T^\mfl
:=
\begin{bmatrix}
1 &  & & && \\
 & \ddots & && & \\
 &  &1 & && \\\hline
\Pi_0^{0,\mfl}A_0 & \cdots & \cdots&\Pi_{\mfn}^{0,\mfl} A_\mfn&&  \\
 & \ddots &\ddots &\ddots&\ddots &  \\
 &  &\Pi_0^{\mfl,\mfl}A_0 &\cdots&\cdots& \Pi_{\mfn}^{\mfl,\mfl}  A_\mfn
\end{bmatrix}
\text{ with }
 \left\{
\begin{array}{lll}
\Pi_k^{i,\mfl}:= \frac{(k+i)!(\mfn - k +\mfl-i)! }{i!(\mfl-i)!},\\
A_{k} := \alpha_{k,\mfn-k}(x_0,y_0).
\end{array}\right.
$$
\end{dfn}

Assuming that all unknowns  $\lambda_{i,j}$ with length of the multi-index at most equal to $\mfn+I+J-1$ are known, then, as expected, a solution to the linear under-determined system \eqref{ssyst} can be computed as follows.

\begin{prop}\label{prop:CS}%Complemented System
\defOp
For a given level $\mfl\in\mathbb N$ with $\mfl<q$, under Hypothesis \ref{hyp}, the matrix $\mathsf T^\mfl\in\mathbb C^{(\mfn +\mfl+1)\times(\mfn +\mfl+1)}$ is non-singular.

We now assume that the unknowns $\{\lambda_{i,j}, (i,j)\in\mathbb N^2, i+j<\mfn +\mfl\}$ are known, so that the terms $N_{I,\mfl-I}$ for $I$ from $0$ to $\mfl$ can be computed. Consider any vector $\mathsf B^\mfl\in\mathbb C^{\mfn +\mfl+1}$ satisfying
\begin{equation*}
%\left\{
%\begin{array}{lll}
%{ \mathsf{B}}_{k+1}^\mfl &= v_k,&\forall k \ s.t. \ 0\leq k\leq \mfn-1,\\
{ \mathsf{B}}_{\mfn+1+I}^\mfl %&\displaystyle 
= 
\RHS_{I,\mfl-I}
,%&
\forall I \ s.t. \ 0\leq I\leq \mfl.
%\end{array}
%\right.
\end{equation*}
Then independently of the first $\mfn$ components of $\mathsf B^\mfl$, solving the linear system
 \begin{equation}\label{trisyst}
 \mathsf T^\mfl \mathsf X^\mfl = \mathsf B^\mfl
 \end{equation}
  by forward substitution provides a solution to \eqref{ssyst} for
$$
\lambda_{i,\mfn+\mfl-i} = \mathsf X^\mfl_{i+1},\ \forall i\in\mathbb N \text{ such that }0\leq  i\leq \mfn+\mfl.
$$
\end{prop}
\begin{proof}
The matrix $\mathsf T^\mfl$ is lower triangular, therefore its determinant is
$$
\det \mathsf T^\mfl
 = \prod_{I=0}^{\mfl} \left(\frac{ (I+\mfn)!(\mfl-I)!}{I!(\mfl-I)!}
\alpha_{\mfn,0}(x_0,y_0)\right)
= \left(\prod_{I=0}^{\mfl} \frac{ (I+\mfn)!}{I!}\right)
\big(\alpha_{\mfn,0}(x_0,y_0)\big)^{\mfl+1},
$$
which can not be zero under Hypothesis \ref{hyp}. The second part of the claim derives directly from the definition of $ \mathsf T^\mfl$ and $ \mathsf B^\mfl$ and the fact that the system is lower triangular, and can be illustrated as follows:
$$\underbrace{\left[
\begin{array}{cccccc}
1 &  & & && \\
 & \ddots & && & \\
 &  &1 & && \\\hline
\Pi_0^{0,\mfl}A_0 & \cdots & \cdots&\Pi_{\mfn}^{0,\mfl} A_\mfn&&  \\
 & \ddots &\ddots &\ddots&\ddots &  \\
 &  &\Pi_0^{\mfl,\mfl}A_0 &\cdots&\cdots& \Pi_{\mfn}^{\mfl,\mfl}  A_\mfn
\end{array}\right]}_{\mathsf T^\mfl}
\underbrace{\begin{bmatrix}
{\lambda_{0,\mfl+\mfn}}\\\vdots\\
{\lambda_{\mfn-1,\mfl+1}}\\\hline
{\lambda_{\mfn,\mfl}}\\
{\vdots}\\
{\lambda_{\mfl+\mfn,0}}
\end{bmatrix}}_{\mathsf X^\mfl}
=
\underbrace{\begin{bmatrix}
*\\\vdots\\
*\\\hline
N_{0,\mfl}\\
\vdots\\
N_{\mfl,0}
\end{bmatrix}}_{\mathsf B^\mfl}
$$
\end{proof}

To summarize, we have defined for increasing values of $\mfl$ a hierarchy of linear systems, each of which has the following characteristics:
\begin{itemize}
\item its unknowns are $\{\lambda_{i,\mfn+\mfl-i} ;\ \forall i\in\mathbb N \text{ such that }0\leq  i\leq \mfn+\mfl\}$;
\item its matrix $\mathsf T^\mfl\in\mathbb C^{(\mfn +\mfl+1)\times(\mfn +\mfl+1)}$ is a square, non-singular, and triangular ;
%\begin{itemize}
%\item square, of size $\mfn+\mfl+1$ at level $\mfl$;
%\item non-singular;
%\item lower triangular, therefore 
%\end{itemize}
\item its right-hand side depends both on $\{\lambda_{i,j} ;\ \forall(i,j)\in\mathbb N^2 \text{ such that }0\leq  i+j < \mfn+\mfl\}$ and on $\mfn$ additional parameters.
%\begin{itemize}
%\item indep of the first entries
%\item with RHS depends on lower level unknowns
%\end{itemize}
\end{itemize}
At each level $\mfl$, assuming that the unknowns of inferior levels are known and provided $\mfn$ given values for $\lambda_{i,\mfn+\mfl-i}$ for $0\leq i <\mfn$, Proposition \ref{prop:CS} provides an explicit formula to compute $\lambda_{i,\mfn+\mfl-i}$ for $\mfn\leq i\leq \mfn+\mfl $. 

%%%%%%%%%%%%%%%%%%%%%
\subsection{Algorithm}\label{ssec:algo}
The non-linear system \eqref{thesyst} had $N_{dof}^{ \eqref{thesyst}}=\frac{(\mfn+q)(\mfn +q+1)}{2}$ unknowns and $N_{eqn}^{ \eqref{thesyst}}=\frac{q(q+1)}{2}$ equations, whereas each linear triangular system introduced in the previous subsection has $N_{dof}^T=\mfn+\mfl+1$ unknowns and $N_{eqn}^T=\mfn+\mfl+1$ equations for each level $\mfl$ such that $0\leq\mfl\leq q-1$.
Therefore the hierarchy of triangular systems has a total of $N_{dof}^{H}=(\mfn+1)q+\frac{q(q-1)}{2}$ unknowns and $N_{eqn}^{H}=N_{eqn}^{ \eqref{thesyst}}+\mfn q=\mfn q+\frac{q(q+1)}{2}$ equations, including the $ \frac{q(q+1)}{2}$ equations of the initial non-linear system \eqref{thesyst}.

%% HERE FIG of the remaining unknowns 
The remaining $N_{dof}^{ \eqref{thesyst}}-N_{dof}^T=\frac{\mfn(\mfn+1)}{2}$ unknowns, which are unknowns of none of the triangular systems but appear only on the right hand side of these systems, are the $\{\lambda_{i,j},(i,j)\in\mathbb N^2,0\leq i+j <\mfn\}$. These are the unknowns with length of the multi-index at most equal to $\mfn-1$, and the corresponding indices $(i,j)$ are the only ones that are not marked on the right panel of Figure \ref{fig:indL}. It is therefore natural to add $\frac{\mfn(\mfn+1)}{2}$ constraints corresponding to fixing the values of the remaining unknowns $\{\lambda_{i,j},(i,j)\in\mathbb N^2,0\leq i+j <\mfn\}$. The final system we consider consists of these $\frac{\mfn(\mfn+1)}{2}$ constraints, guaranteeing that the unknowns  $\{\lambda_{i,j},(i,j)\in\mathbb N^2,0\leq i+j <\mfn\}$ are known, together with the hierarchy of triangular systems \eqref{trisyst} for  increasing values of $\mfl$ from $0$ to $q-1$; it has $N_{dof}^{F}=\frac{(\mfn+q)(\mfn +q+1)}{2} $ unknowns, namely the unknowns of the original system \eqref{thesyst}, and $N_{eqn}^{F}=\frac{(\mfn+q)(\mfn +q+1)}{2}$ equations, namely the equations of the original system split into linear subsytems together with a total of $\frac{\mfn(\mfn+1)}{2}+q\mfn$ additional constraints.
A counting summary is presented here:
\begin{equation*}%\label{tab:count}
\begin{array}{|c|c|c|}
\hline
 & \text{Number of unknowns} & \text{Number of equations}
 \\\hline
 \begin{array}{c}
 \text{Original non-linear system}\\
 \eqref{thesyst} 
 \end{array}
 & N_{dof}^{ \eqref{thesyst}}=\frac{(\mfn+q)(\mfn +q+1)}{2} &N_{eqn}^{ \eqref{thesyst}}=\frac{q(q+1)}{2}
 \\\hline
 \begin{array}{c}
 \text{Subsystem at level } \mfl\\
 \eqref{ssyst}
 \end{array}
 & N_{dof}^\mfl=\mfn+\mfl+1 & N_{eqn}^\mfl=\mfl+1
 \\\hline\hline
 \begin{array}{c}
 \text{Triangular system at level } \mfl\\
 \eqref{trisyst}
 \end{array}
 & N_{dof}^T=\mfn+\mfl+1 & N_{eqn}^T=\mfn+\mfl+1
 \\\hline
 \begin{array}{c}
 \text{Hierarchy of triangular systems} \\ \text{for $\mfl$ from $0$ to $q-1$}
 \end{array}
 & N_{dof}^{H}=(\mfn+1)q+\frac{q(q-1)}{2} & N_{eqn}^{H}=\mfn q+\frac{q(q+1)}{2}
 \\\hline\hline
 \begin{array}{c}
 \text{Final system} \\ \text{(initial constraints + triangular systems)}
 \end{array}
  & N_{dof}^{F}=\frac{(\mfn+q)(\mfn +q+1)}{2} & N_{eqn}^{F}=\frac{(\mfn+q)(\mfn +q+1)}{2}
 \\\hline
\end{array}
\end{equation*}
Thanks to the $\frac{\mfn(\mfn+1)}{2}$ constraints, for increasing values of $\mfl$ from $0$ to $q-1$, the hypothesis of Proposition \ref{prop:CS} is satisfied, the right hand side $\mathsf B^\mfl$ can be evaluated and the triangular system \eqref{trisyst} can be solved. So the unknowns $\{\lambda_{i,\mfn+\mfl-i} ;\ \forall i\in\mathbb N \text{ such that }0\leq  i\leq \mfn+\mfl\}$ can be computed by induction on $\mfl$, constructing a solution to the initial non-linear system \eqref{thesyst} by induction on $\mfl$.

The following algorithm requires the value of $\frac{\mfn(\mfn+1)}{2}+q\mfn$ parameters, to fix initially the set of unknowns $\{\lambda_{i,j},(i,j)\in\mathbb N^2,0\leq i+j <\mfn\}$ and then at each level $\mfl$ the set of unknowns  $\{\lambda_{i,\mfn+\mfl-i},i\in\mathbb N,0\leq i<\mfn\}$. Under Hypothesis \ref{hyp}, the algorithm presents a sequence of steps to construct explicitly a solution to Problem \eqref{thesyst} and requires no approximation process. 
\begin{algorithm}[H] 
\caption{Constructing a solution to Problem \eqref{thesyst}}
\label{algo}
\begin{algorithmic}[1]
\State\label{fix1} Fix $\{\lambda_{i,j},(i,j)\in\mathbb N^2,0\leq i+j <\mfn\}$\Comment{ $\frac{\mfn(\mfn+1)}{2}$ unknowns}
\For{$\mfl$ from $0$ to $q-1$}\label{Lloop}\Comment{ $q$ times $\phantom{unknoni}$}
%\State{\% Compute $\{\lambda_{i,j},i+j=\mfl +\mfn\}$ from equations $\{(I,\mfl-I),0\leq I\leq \mfl\}$}
\State\label{fix2} Fix $\{\lambda_{i,\mfn+\mfl-i},i\in\mathbb N,0\leq i<\mfn\}$\Comment{ $\mfn$ unknowns$\phantom{un}$}
\For{$I$ from $ 0$ to $\mfl$}\Comment{ $\mfl+ 1$ times$\phantom{un}$}
\State\label{ExFo} %$\displaystyle \lambda_{I+\mfn,\mfl-I} := \frac{1}{\mathsf T^\mfl_{I+\mfn+1,I+\mfn+1}}\left(\mathsf B^\mfl_{I+\mfn+1} - 
%\sum_{k=0}^{\mfn-1}\mathsf T^\mfl_{I+\mfn+1,I+k+1} \lambda_{I+k,\mfn+\mfl-I-k}\right)$ \Comment{ $1$ unknown}
%
$
\displaystyle \lambda_{I+\mfn,\mfl-I} := 
\displaystyle \frac{1}{\mathsf T^\mfl_{I+\mfn+1,I+\mfn+1}}\left(\mathsf B^\mfl_{I+\mfn+1} - 
\sum_{k=0}^{\mfn-1}\mathsf T^\mfl_{I+\mfn+1,I+k+1} \lambda_{I+k,\mfn+\mfl-I-k}\right)
$
\Comment{ $1$ unknown}
\EndFor
\EndFor
\end{algorithmic}
\end{algorithm}

%%FOLLOW UP EXAMPLES FOR A  $\mathfrak n=2$ and a $\mathfrak n = 4$ CASE
%% I DONT THINK SO - THE NUMERICAL SECTIONS ARE THERE FOR THAT
From the definitions of $\mathsf T^\mfl$ and $\mathsf B^\mfl$ we immediately see that the step $5$ boils down to
\begin{equation}\label{eq:ls}
\lambda_{I+\mfn,\mfl-I} =\frac{I!}{(I+\mfn)!\alpha_{\mfn,0}(x_0,y_0)}\left(\RHS_{I,\mfl-I} - 
\sum_{k=0}^{\mfn-1} \frac{(I+k)!(\mfn-k+\mfl-I)!}{I!(\mfl-I)!}\alpha_{k,\mfn-k}(x_0,y_0) \lambda_{I+k,\mfn+\mfl-I-k}\right)
\end{equation}
 
If the set of unknowns $\{\lambda_{i,j},(i,j)\in\mathbb N^2,0\leq i+j <\mfn+q-1\}$ is computed from Algorithm \ref{algo}, then the polynomial
 $\displaystyle P(x,y):= \sum_{0\leq i+j\leq q+\mfn-1} \lambda_{i,j}(x-x_0)^i(y-y_0)^j$ is a solution to Problem \eqref{thesyst}, and therefore the function $\displaystyle\varphi(\mathbf x) :=\exp P(\mathbf{x})$ satisfies \eqref{eq:Lhq}. This is true independently of the values fixed in lines \algref{algo}{fix1} and \algref{algo}{fix2} of the algorithm.

\begin{rmk}\label{rmk:chgtp}
 It is interesting to notice that the algorithm applies to a wide range of partial differential operators,
 %partial differential operators that have nothing to do with waves, 
 including type changing operators such as Keldysh operators, $L_K=\partial_x^2 + y^{2m+1}\partial_y^2+$ lower order terms, or Tricomi operators, $ L_T=\partial_x^2 + x^{2m+1}\partial_y^2 +$ lower order terms, that change from elliptic to hyperbolic type along a smooth parabolic curve.
 \end{rmk}

To conclude this section, we provide a formal definition of a GPW associated to an partial differential operator at a given point.
\begin{dfn}
\defOp
 A Generalized Plane Wave (GPW) associated to the differential operator $\Lal$ at the point $(x_0,y_0)$ is a function $\varphi$ satisfying
 $$
 \Lal \varphi{(x,y)} = O(\|(x-x_0,y-y_0)\|^q).
 $$

Under Hypothesis \ref{hyp}, a Generalized Plane Wave (GPW) can be constructed as function $\varphi (x,y)= \exp P(x,y) $, where the coefficients of the polynomial $P$ are computed by Algorithm \ref{algo}, independently of the values fixed in the algorithm. 
\end{dfn}

The crucial feature of the construction process is the {exact} solution provided in the algorithm: in practice, a solution to the initial non-linear rectangular system is computed without numerical resolution of any system, {\bf with an  explicit formula}.

The choice of the fixed values in Algorithm \ref{algo} will be discussed in the next paragraph. Even though these values does not affect the construction process, and the fact that the corresponding $\varphi (x,y)= \exp P(x,y) $ is a GPW, it will be key to prove the interpolation properties of the corresponding set of GPWs. 
\begin{rmk}\label{rmk:sim}
Under the hypothesis $\alpha_{0,\mfn}(x_0,y_0)\neq 0$ it would be natural to fix the values of $\{ \lambda_{i,j}, 0\leq j\leq \mfn -1, 0\leq i\leq q+\mfn -1 -j \}$ instead of those of $\{ \lambda_{i,j}, 0\leq i\leq \mfn -1, 0\leq j\leq q+\mfn -1 -i \}$, and an algorithm very similar to Algorithm \ref{algo}, exchanging the roles of $i$ and $j$ would construct the polynomial coefficients of a GPW.
\end{rmk}

%REMARK ON THE FACT THAT THE ALGO IS HIERARCHICAL SO EASY TO INCREASE $q$ IF NEEDED %it's true bue careful cause we are not likely to want to increase q for a fixed set of functions (fixed p) - rather we would increase both p and q with n, so we would need to recompute the whole functions

%%%%%%%%%%%%%%%%%%%%%%%%%%%%%%%%%%%%%%%%%%%%
\section{Normalization}\label{sec:norm}
We will refer to normalization as the choice of imposed values in Algorithm \ref{algo}. The discussion presented in this section will be summarized in Definition \ref{df:norm}.

Within the construction process presented in the previous section, only the design of the function $\varphi$ as the exponential of a polynomial is related to wave propagation, while Algorithm \ref{algo} works for partial differential operators not necessarily related to wave propagation. In particular, the property $\Lal \varphi(x,y)=O\left(\| (x,y)-(x_0,y_0) \|^q\right)$ of GPWs is independent of the choice of $(\luoz,\luzo)$. However, the normalization process described here carries on the idea of adding higher order terms to the phase function of a plane wave, see \eqref{eq:c/GPW}, as was proposed in \cite{LMinterp}.

We will now restrict our attention to a smaller set of partial differential operators that include several interesting operators related to wave propagation, thanks to an additional hypothesis on the highest order derivatives in $\Lal$, namely Hypothesis \ref{hyp2}.
Under this hypothesis we will be able to study the interpolation properties of associated GPWs in a unified framework. As we will see in this section, choosing only two non-zero fixed values in Algorithm \ref{algo} is sufficient to generate a set of linearly independent GPWs. It is then natural to study how the rest of the $\luij$s depend on those two values, and the related consequences of Hypothesis \ref{hyp2}. These rely on Hypothesis \ref{hyp2} extending the fact that for classical PWs 
%$\luoz^2+\luzo^2=-\kappa^2$.
$(i\kappa\cos\theta)^2+(i\kappa\sin\theta)^2 = -\kappa^2 $ is independent of $\theta$.

%%%%%%%%%%%%
\subsection{For every GPWs}\label{ssec:every}
%Simplifying choice
%Part II all the zero terms

In Algorithm \ref{algo}, the number of prescribed coefficients is $\frac{\mfn(\mfn+1)}{2}+\mfn q$, and the set of coefficients to be prescribed is the set $\{ \lambda_{i,j}, 0\leq i\leq \mfn -1, 0\leq j\leq q+\mfn -1 -i \}$.

For the sake of simplicity, it is natural to choose most of these values to be zero.
Since the unknown $\lambda_{0,0}$ never appears in the non-linear system, there is nothing more natural than setting it to zero: this ensures that any GPW $\varphi$ will satisfy $\varphi(x_0,y_0)=1$.
Concerning the subset of $\mfn q$ unknowns corresponding to step \algref{algo}{fix2} in Algorithm \ref{algo}, setting these values to zero simply reduces the amount of computation involved  in step \algref{algo}{ExFo} in the algorithm: indeed for $I=0$ then $\displaystyle \sum_{k=0}^{\mfn-1}\mathsf T^\mfl_{I+\mfn+1,I+k+1} \lambda_{I+k,\mfn+\mfl-I-k}=0$, while for $0<I<\mfn$ then $$\displaystyle \sum_{k=0}^{\mfn-1}\mathsf T^\mfl_{I+\mfn+1,I+k+1} \lambda_{I+k,\mfn+\mfl-I-k}=\sum_{k=\mfn-\mfl}^{\mfn-1}\mathsf T^\mfl_{I+\mfn+1,I+k+1} \lambda_{I+k,\mfn+\mfl-I-k}.$$

As for the unknowns $\lambda_{1,0}$ and $\lambda_{0,1}$, they will be non-zero to mimic the classical plane wave case, and their precise choice will be discussed in the next subsection. For the remaining unknowns to be fixed, that is to say the set $\{ \lambda_{i,j}, 2\leq i+j\leq \mfn -1 \}$, their values are set to zero, here again in order to reduce the amount of computation in computing the right hand side entries $\mathrm B^{\mathfrak L}_{\mfn+1+I}$ and in applying \algref{algo}{ExFo}.

%% FOLLOW UP EXAMPLES FOR A  $\mathfrak n=2$ and a $\mathfrak n = 4$ CASE
%[EX] 
For the operator $\mathfrak L_{2,\gamma}$  the non-linear terms in $\RHS_{1,0}$ and $\RHS_{0,1}$  respectively become with this normalization
\begin{equation*}
\begin{array}{rl}
\dx [\mathfrak L_{2,\gamma}^N P](x_0,y_0)=&
-2\lambda_{2,0}\lambda_{1,0}
+\gamma_{1,1}(x_0,y_0)%\left( 
2\lambda_{2,0}\lambda_{0,1}%+\lambda_{1,0}\lambda_{1,1} \right)
%+2\gamma_{0,2}(x_0,y_0)\lambda_{1,1}\lambda_{0,1}
%\\&
+\dx \gamma_{1,1}(x_0,y_0)\lambda_{1,0}\lambda_{0,1}
+\dx \gamma_{0,2}(x_0,y_0)\lambda_{0,1}^2,
\end{array}
\end{equation*}
\begin{equation*}
\begin{array}{rl}
\dy [\mathfrak L_{2,\gamma}^N P](x_0,y_0)=&
%-2\lambda_{1,1}\lambda_{1,0}
%+
\gamma_{1,1}(x_0,y_0)%\left( \lambda_{1,1}\lambda_{0,1}+
2\lambda_{1,0}\lambda_{2,0}% \right)
%+2\gamma_{0,2}(x_0,y_0)\lambda_{0,2}\lambda_{0,1}
%\\&
+\dy \gamma_{1,1}(x_0,y_0)\lambda_{1,0}\lambda_{0,1}
%+\dy \gamma_{0,2}(x_0,y_0)\lambda_{0,1}^2
.
\end{array}
\end{equation*}

Since all but two of the unknowns to be fixed in Algorithm \ref{algo} are set to zero, it is now natural to express the $\lambda_{i,j}$ unknowns computed from \algref{algo}{ExFo} in the algorithm as functions of the two non-zero prescribed unknowns, $\lambda_{1,0}$ and $\lambda_{0,1}$. 

\begin{lmm}
\defOp
Under Hypothesis \ref{hyp} consider a solution to Problem \eqref{thesyst} constructed thanks to Algorithm \ref{algo}  with all the prescribed values $\lambda_{i,j}$ such that $i<\mfn$ and $i+j\neq 1$ set to zero. Each $\lambda_{i+\mfn,j}$ can be expressed as an element of $\mathbb C[\lambda_{1,0},\lambda_{0,1}]$.
\end{lmm}
\begin{proof}
The fact that $\lambda_{i+\mfn,j}$ can be expressed as a polynomial in two variables with respect to $\lambda_{1,0}$ and $\lambda_{0,1}$ is a direct consequence from the explicit formula in step \algref{algo}{ExFo} in Algorithm \ref{algo} combining with setting $\lambda_{i,j}$ such that $i<\mfn$ and $i+j\neq 1$ to zero.
\end{proof}
Since unknowns are expressed as elements of $\mathbb C[\lambda_{1,0},\lambda_{0,1}]$, we will now study the degree of various terms from Algorithm \ref{algo} as polynomials with respect to  $\lambda_{1,0}$ and $\lambda_{0,1}$. To do so, we will start by inspecting the product terms appearing in Faa di Bruno's formula.
\begin{lmm}\label{lm:LNP}
\defOp
Consider a given polynomial $P\in\mathbb C[x,y]$. The non-linear terms $\LN P$, expressed as linear combinations of products of derivatives of $P$, namely $\prod_{m=1}^s \left( \dx^{i_m}\dy^{j_m} P \right)^{k_m}$, contain products of up to $\mfn$ derivatives of $P$, namely $\dx^{i_m}\dy^{j_m} P $, counting repetitions. The only products that have exactly $\mfn$ terms are 
$(\dx P)^k(\dy P)^{\mfn-k}$ for $0\leq k \leq \mfn$, %$(\dx P)^\mfn$, $(\dy P)^\mfn$, and $(\dx P)^k(\dy P)^{\mfn-k}$, 
whereas all the other products have less than $\mfn$ terms.
\end{lmm}
\begin{proof}
Since the operator $\LN$ is defined via Faa di Bruno's formula, we will proceed by careful examination of the summation and product indices in the latter.

The number of terms in the product term is $s$, with possible repetitions counted thanks to the $k_m$s, and the total number of terms counting repetitions is $\mu=\sum_{m=1}^s k_m$. Since in $\LN$ the indices are such that $1\leq \mu\leq \ell\leq\mfn$, there cannot be more than $\mfn$ terms counting repetitions in any of the $\prod_{m=1}^s \left( \dx^{i_m}\dy^{j_m} P \right)^{k_m}$.

For $s=1$, in the set $p_1((k,\ell-k),\mu)$, $(i_1,j_1)\in\mathbb N^2$ are such that $i_1+j_1
\geq 1$ and $k_1\in\mathbb N$ is such that $k_1(i_1+j_1)=\ell$. Since $\ell\leq \mfn$, such a term appears in Faa di Bruno's formula as a product of $\mu=\mfn$ terms if and only if $\ell=\mfn$, $k_1=\mfn$, and therefore $i_1+j_1= 1$. There are then only two possibilities: either $(i_1,j_1)=(1,0)$ corresponding to the term $(\dx P)^\mfn$, or $(i_1,j_1)=(0,1)$ corresponding to the term $(\dy P)^\mfn$.

For $s=2$, in the set $p_2((k,\ell-k),\mu)$, $(i_1,j_1,i_2,j_2)\in\mathbb N^4$ are such that $i_1+j_1\geq 1$, $i_2+j_2\geq 1$, $(i_1,j_1)\prec(i_2,j_2)$, and $(k_1,k_2)\in\mathbb N^2$ is such that $\mu=k_1+k_2$ and  $k_1(i_1+j_1)+k_2(i_2+j_2)=\ell$. Since $\ell\leq \mfn$ and $\ell=k_1(i_1+j_1)+k_2(i_2+j_2)\geq k_1+k_2=\mu$, such a term appears in Faa di Bruno's formula as a product of $\mu=\mfn$ terms if and only if $\ell=\mfn$ and $k_1+k_2=\mfn$. There are then two possible cases: either $i_2+j_2>1$, then $\mfn=k_1(i_1+j_1)+k_2(i_2+j_2)>k_1+k_2=\mfn$, so there is no such term in the sum, or $i_2+j_2=1$, then necessarily $(i_1,j_1)=(0,1)$ and $(i_2,j_2)=(1,0)$, corresponding to the terms $(\dx P)^k(\dy P)^{\mfn-k}$ for any $k$ from $0$ to $\mfn$.

For $s\geq 3$, in the set $p_s((k,\ell-k),\mu)$, for all $m\in\mathbb N$ such that $1\leq m\leq s$, $(i_m,j_m)\in\mathbb N^2$ and $k_m\in\mathbb N$ are such that $i_m+j_m\geq 1$, $\sum_{m=1}^s k_m(i_m+j_m)=\ell$, $\mu=\sum_{m=1}^s k_m$ and  $(i_1,j_1)\prec(i_2,j_2)\prec(i_3,j_3)$. Because of this last condition, it is clear that $i_3+j_3>1$. Since $\ell\leq \mfn$ and $\ell=\sum_{m=1}^s k_m(i_m+j_m)\geq \sum_{m=1}^s k_m=\mu$, such a term appears in Faa di Bruno's formula as a product of $\mu=\mfn$ terms if and only if $\ell=\mfn$ and $\sum_{m=1}^s k_m=\mfn$. But then $\mfn=\sum_{m=1}^s k_m(i_m+j_m)> \sum_{m=1}^s k_m=\mfn$, so there is no such term in the sum.

The claim is proved.
\end{proof}
\begin{lmm}\label{lm:IJLNP}
\defOp
Consider a given polynomial $P\in\mathbb C[x,y]$. The quantity $\dx^{I_0}\dy^{J_0}\LN P$ is a linear combination of terms $\dx^{I_0}\dy^{J_0}\left(\prod_{m=1}^s \left( \dx^{i_m}\dy^{j_m} P \right)^{k_m}\right)$, where the indices come from Faa di Bruno's formula. Each of these $\dx^{I_0}\dy^{J_0}\left(\prod_{m=1}^s \left( \dx^{i_m}\dy^{j_m} P \right)^{k_m}\right)$ can be expressed as  a linear combination  of products $\prod_{m=1}^t (\dx^{a_m}\dy^{b_m} P)^{c_m}$ where the indices satisfy $\sum_{m=1}^t c_m(a_m+b_m)\leq {I_0}+{J_0}+\mfn$.
\end{lmm}
\begin{proof}
Thanks to the product rule, the derivative $\dx^{I_0}\dy^{J_0}\left(\prod_{m=1}^s \left( \dx^{i_m}\dy^{j_m} P \right)^{k_m}\right)$ can be expressed as a linear combination of several terms $\prod_{m=1}^s \dx^{I_m}\dy^{J_m}\left[\left( \dx^{i_m}\dy^{j_m} P \right)^{k_m}\right]$, where $\sum_{m=1}^tI_m=I_0$ and $\sum_{m=1}^tJ_m=J_0$.

We can prove by induction on $k$ that $ \dx^{I}\dy^{J}\left[\left( \dx^{i}\dy^{j} P \right)^{k}\right]$ can be expressed, for all $(i,j,I,J)\in\mathbb N^4$, as  a linear combination  of products $\prod_{m=1}^M (\dx^{a_m}\dy^{b_m} P)^{c_m}$ where the indices satisfy $\sum_{m=1}^M c_m(a_m+b_m)\leq I+J+k(i+j)$:
\begin{enumerate}
\item it is evidently true for $k=1$;
\item suppose that it is true for $k_0\geq 1$, then for any $(i,j,I,J)\in\mathbb N^4$ the product rule applied  to $\dx^{i}\dy^{j} P\times \left(\dx^{i}\dy^{j} P \right)^{k_0}$ yields
$$
\dx^{I}\dy^{J}\left[\left( \dx^{i}\dy^{j} P \right)^{k_0+1}\right]
=
\sum_{\tilde i =0}^{I}
\sum_{\tilde j =0}^{J}
\begin{pmatrix}
I\\\tilde i
\end{pmatrix}
\begin{pmatrix}
J\\\tilde j
\end{pmatrix} 
\dx^{i+I-\tilde i}\dy^{j+J-\tilde j} P
\dx^{\tilde i}\dy^{\tilde j} \left[\left( \dx^{i}\dy^{j} P \right)^{k_0}\right],
$$
where by hypothesis each $\dx^{\tilde i}\dy^{\tilde j} \left[\left( \dx^{i}\dy^{j} P \right)^{k_0}\right]$can be expressed as  a linear combination  of products $\prod_{m=1}^M (\dx^{a_m}\dy^{b_m} P)^{c_m}$ with $\sum_{m=1}^M c_m(a_m+b_m)\leq \tilde i+\tilde j+k_0(i+j)$, so that each term in the double sum can be expressed  as  a linear combination  of products $ \prod_{m=1}^{M+1} (\dx^{a_m}\dy^{b_m} P)^{c_m}$
 where $a_{M+1} :=i+I-\tilde i$, $b_{M+1} := j+J-\tilde j$ and $c_{M+1}:=1$, which yields $\sum_{m=1}^{M+1} c_m(a_m+b_m)=\sum_{m=1}^{M} c_m(a_m+b_m)+(i+I-\tilde i+ j+J-\tilde j)$ and therefore 
 $\sum_{m=1}^{M+1} c_m(a_m+b_m)\leq k_0(i+j)+(i+I+ j+J)$. This concludes the proof by induction.
 \end{enumerate}
 Finally the derivative $\dx^{I_0}\dy^{J_0}\left(\prod_{m=1}^s \left( \dx^{i_m}\dy^{j_m} P \right)^{k_m}\right)$ can be expressed as a linear combination of several terms $\prod_{m=1}^s \prod_{\tilde m=1}^M (\dx^{a_{\tilde m}}\dy^{b_{\tilde m}} P)^{c_{\tilde m}}$, with $\sum_{\tilde m=1}^M c_{\tilde m}(a_{\tilde m}+b_{\tilde m})\leq I_m+J_m+k_m(i_m+j_m)$, in other words it can be expressed as a linear combination of several terms $\prod_{m=1}^{Ms} (\dx^{a_{ m}}\dy^{b_{m}} P)^{c_{m}}$, with $\sum_{ m=1}^{Ms} c_{ m}(a_{ m}+b_{m})\leq \sum_{m=1}^s I_m+J_m+k_m(i_m+j_m)=I_0+J_0+\sum_{m=1}^s k_m(i_m+j_m)$. For any $\dx^{I_0}\dy^{J_0}\left(\prod_{m=1}^s \left( \dx^{i_m}\dy^{j_m} P \right)^{k_m}\right)$ coming from $\dx^{I_0}\dy^{J_0}\LN P$, the summation indices from Faa di Bruno's formula satisfy $\sum_{m=1}^s k_m(i_m+j_m)=\ell$, so the products  $\prod_{m=1}^{Ms} (\dx^{a_{ m}}\dy^{b_{m}} P)^{c_{m}}$ are such that $\sum_{ m=1}^{Ms} c_{ m}(a_{ m}+b_{m})\leq I_0+J_0+\mfn$.
\end{proof}

The two following results now turn to $\lambda_{i+\mfn,j}$ computed in Algorithm \ref{algo}.
\begin{prop}\label{prop:ln0}
\defOp
Under Hypothesis \ref{hyp} consider a solution to Problem \eqref{thesyst} constructed thanks to Algorithm \ref{algo}  with all the fixed values $\lambda_{i,j}$ such that $i<\mfn$ and $i+j\neq 1$ set to zero. As an element of  $\mathbb C[\lambda_{1,0},\lambda_{0,1}]$, $\lambda_{\mfn,0}$ is of degree equal to $\mfn$.
\end{prop}
\begin{proof}
The formula to compute $\lambda_{\mfn,0}$ in Algorithm \ref{algo} comes from the $(I,J)=(0,0)$ equation in System \eqref{thesyst}, that is to say $\LA P(x_0,y_0)=-\alpha_{0,0}(x_0,y_0)$. It reads
$$
\lambda_{\mfn,0}=\frac{1}{\mathsf T^0_{\mfn+1,\mfn+1}}\left(\mathsf B^0_{\mfn+1} - 
\sum_{k=0}^{\mfn-1}\mathsf T^0_{\mfn+1,k+1} \lambda_{k,\mfn-k}\right),
$$
and the sum is actually zero since the $\lambda_{k,\mfn-k}$ unknowns are prescribed to zero for $k<\mfn$. The definitions of $\mathsf B^0$ and $\mathsf L^0$ then give
$$
\lambda_{\mfn,0}=\frac{1}{\mfn!\alpha_{\mfn,0}(x_0,y_0)}\left(- 
\sum_{\ell=0}^{\mfn-1} \sum_{k=0}^{\ell}
k!(\ell-k)! \alpha_{k,\ell-k}(x_0,y_0)\lambda_{k,\ell-k}
-\LN P(x_0,y_0)- \alpha_{0,0}(x_0,y_0)\right).
$$
Since the $\lambda_{k,\ell-k}$ unknowns are prescribed to zero for all $1<\ell <\mfn-1$ and all $k$, the double sum term reduces to $ \alpha_{0,1}(x_0,y_0)\lambda_{0,1}+\alpha_{1,0}(x_0,y_0)\lambda_{1,0}$. The non-linear terms from $\LN P$, namely $\prod_{m=1}^s (\dx^{i_m}\dy^{j_m}P)^{k_m}$, are products of at most $\mfn$ terms, counting repetitions, according to Lemma \ref{lm:LNP}. So $\LN P(x_0,y_0)$ is a linear combination of product terms reading $\prod_{m=1}^s (\lambda_{i_m,j_m})^{k_m}$ with at most $\mfn$ factors.  Moreover, since $P$ is constructed thanks to Algorithm \ref{algo},  from Corollary \ref{cor:RIJ} we know that these $ \lambda_{i_m,j_m}$s have a length of the multi-index at most equal to $\mfn-1$, so they are either $\lambda_{1,0}$ or $\lambda_{0,1}$ or prescribed to zero. This means that in $\mathbb C[\lambda_{1,0},\lambda_{0,1}]$ each one of these $ \lambda_{i_m,j_m}$ is at most of degree one. So in $\mathbb C[\lambda_{1,0},\lambda_{0,1}]$ each $\prod_{m=1}^s (\lambda_{i_m,j_m})^{k_m}$ is a product of at most $\mfn$ factors each of them of degree at most one, the product is therefore of degree at most $\mfn$. As a result
$$
\lambda_{\mfn,0}=\frac{1}{\mfn!\alpha_{\mfn,0}(x_0,y_0)}\left( 
-\alpha_{0,1}(x_0,y_0)\lambda_{0,1}
-\alpha_{1,0}(x_0,y_0)\lambda_{1,0}
-\LN P(x_0,y_0)- \alpha_{0,0}(x_0,y_0)\right)
$$
as an element of $\mathbb C[\lambda_{1,0},\lambda_{0,1}]$ is of degree at most $\mfn$.

Finally, the term $(\dx P)^\mfn$ from $\LN P$ identified in Lemma \ref{lm:LNP} corresponds to a term $\alpha_{\mfn,0}(x_0,y_0)(\lambda_{1,0})^\mfn$ in the expression of $\lambda_{\mfn,0}$, and this term is non-zero under Hypothesis \ref{hyp}. As a conclusion $\lambda_{\mfn,0}$ as an element of $\mathbb C[\lambda_{1,0},\lambda_{0,1}]$ is of degree equal to $\mfn$.
\end{proof}
\begin{prop}
\defOp
Under Hypothesis \ref{hyp} consider a solution to Problem \eqref{thesyst} constructed thanks to Algorithm \ref{algo}  with all the fixed values $\lambda_{i,j}$ such that $i<\mfn$ and $i+j\neq 1$ set to zero. As an element of $\mathbb C[\lambda_{1,0},\lambda_{0,1}]$, each $\lambda_{i+\mfn,j}$ has a total degree at most equal to the length of its multi-index $i+j+\mfn$.
\end{prop}
\begin{proof}
The formula to compute $\lambda_{I+\mfn,\mfl-I}$ in Algorithm \ref{algo} comes from the $(I,J)=(I,\mfl-I)$ equation in System \eqref{thesyst}, that is to say $\dx^I\dy^{\mfl-I}\LA P(x_0,y_0)=-\dx^I\dy^{\mfl-I}\alpha_{0,0}(x_0,y_0)$. It reads
\begin{equation}\label{eq:lipn}
\begin{array}{rl}
\lambda_{I+\mfn,\mfl-I}&\displaystyle
=\frac{1}{\mathsf T^\mfl_{I+\mfn+1,I+\mfn+1}}\left(\mathsf B^\mfl_{I+\mfn+1} - 
\sum_{k=0}^{\mfn-1}\mathsf T^\mfl_{I+\mfn+1,I+k+1} \lambda_{I+k,\mfn+\mfl-I-k}\right)\\
&\displaystyle
=\frac{I!}{(\mfn+I)!\alpha_{\mfn,0}(x_0,y_0)}\left(\mathsf \RHS_{I,\mfl-I} - 
\sum_{k=0}^{\mfn-1}
\frac{(I+k)!(\mfn-k+\mfl-1)!}{I!(\mfl-I)!} 
\alpha_{k,\mfn-k}(x_0,y_0)
\lambda_{I+k,\mfn+\mfl-I-k}\right).
\end{array}
\end{equation}
We will proceed by induction on $\mfl$:
\begin{enumerate}
\item the result has been proved to be true for $\mfl = 0$ in Proposition \ref{prop:ln0} ;
\item suppose the result is true for $\mfl\in\mathbb N$ as well as for all $\tilde\mfl\in\mathbb N$ such that $\tilde\mfl\leq\mfl$, then all the linear terms in $\RHS_{I,\mfl+1-I}$ have a length of the multi-index at most equal to $\mfn+\mfl$ so by hypothesis their degree as elements of $\mathbb C[\lambda_{1,0},\lambda_{0,1}]$ is at most equal to $\mfn+\mfl$, and thanks to Lemma \ref{lm:IJLNP} all the non-linear terms in $\RHS_{I,\mfl+1-I}$ can be expressed as  a linear combination  of products $\prod_{m=1}^t (\lambda_{a_m,b_m})^{c_m}$ where the indices satisfy $\sum_{m=1}^t c_m(a_m+b_m)\leq \mfl+1+\mfn$ so by hypothesis their degree as elements of $\mathbb C[\lambda_{1,0},\lambda_{0,1}]$ is at most equal to $\mfn+\mfl+1$ ; the last step is to prove that the $\lambda_{I+k,\mfn+\mfl+1-I-k}$ are also of degree at most equal to $\mfn+\mfl+1$, and we will proceed by induction on $I$:
\begin{enumerate}
\item for $I=0$, all  $\lambda_{I+k,\mfn+\mfl+1-I-k}$ for $0\leq k\leq \mfn-1$ satisfy the two conditions $I+k<\mfn$ and $I+k+\mfn+\mfl+1-I-k = \mfn+\mfl+1\neq 1$ so they are all prescribed to zero and their degree as element of $\mathbb C[\lambda_{1,0},\lambda_{0,1}]$ is at most equal to $\mfn+\mfl+1$ that ;
\item suppose that, for a given $I\in \mathbb N$, the $\lambda_{\tilde I+k,\mfn+\mfl+1-\tilde I-k}$  for all $\tilde I\in \mathbb N$ such that $\tilde I\leq I$  are also of degree at most equal to $\mfn+\mfl+1$ then it is clear from Equation \eqref{eq:lipn} that $\lambda_{I+1+\mfn,\mfl-I-1}$ is also of degree at most  equal to $\mfn+\mfl+1$.
\end{enumerate}
This concludes the proof.
\end{enumerate}
\end{proof}
As explained from an algebraic viewpoint in section 3.2 in \cite{LMinterp}, the degree of $\lambda_{i+\mfn,j}$ as an element of $\mathbb C[\lambda_{1,0},\lambda_{0,1}]$ will be affected by the choice of the last two prescribed values, namely $\lambda_{1,0}$ and $\lambda_{0,1}$. Indeed if $\lambda_{1,0}$ and $\lambda_{0,1}$ satisfy a polynomial identity $P_l(\lambda_{1,0},\lambda_{0,1})=0$, then we can consider the quotient ring
$\mathbb C [\lambda_{1,0},\lambda_{0,1}]/(P_l)$.

Note that choosing to set $\{ \lambda_{i,j}, 1< i+j\leq \mfn -1 \}$ to values different from zero may be useful to treat operators that do not satisfy Hypothesis \ref{hyp2} but this is not our goal here.

%%%%%%%%%%%%%%%%%%%%%%%%%%%%%%%%%%%%%%%%%%%%
\subsection{For each GPW}
\label{ssec:each}
%Part I the cos theta sin theta part
% Say all zero but $\lambda_{1,0},\lambda_{0,1}$ to mimic PW + higher order terms. 
In order to obtain a set of linearly independent GPWs, the values of  $\lambda_{1,0}$ and $\lambda_{0,1}$ will be chosen different for each GPW. However  the values of  $\lambda_{1,0}$ and $\lambda_{0,1}$ will satisfy a common property for every GPWs. Very much as the coefficients of any plane wave of wavenumber $\kappa$ satisfy $(\lambda_{1,0})^2+(\lambda_{0,1})^2=-\kappa^2$, independently of the direction of propagation $\theta$ since  $\lambda_{1,0}=\imath\kappa\cos\theta$ and  $\lambda_{0,1}=\imath\kappa\sin\theta$, under Hypothesis \ref{hyp2} the coefficients of each GPW will be chosen for the quantity
$$
\sum_{k=0}^\mfn \alpha_{k,\mfn-k}(x_0,y_0)
(\lambda_{1,0})^k(\lambda_{0,1})^{\mfn-k}
=
\left(\begin{pmatrix}
\lambda_{1,0}\\\lambda_{0,1}
\end{pmatrix}^t
\Gamma
\begin{pmatrix}
\lambda_{1,0}\\\lambda_{0,1}
\end{pmatrix}\right)^{\frac{\mfn}{2}}
$$
to be identical for every GPWs, as we will see in the following proposition and theorem. %independent of the particular choice of $\lambda_{1,0}$ and $\lambda_{0,1}$.

This will be crucial to prove interpolation properties of the corresponding set of functions, which will result from the consequence of this common property on the degree of each $\lambda_{i+\mfn,j}$ as an element of $\mathbb C[\lambda_{1,0},\lambda_{0,1}]$. As the plane wave case suggests, we will see that $\lambda_{i+\mfn,j}$ can be expressed as a polynomial of lower degree thanks to a judicious choice for $\lambda_{1,0}$ and $\lambda_{0,1}$.

We first need an intermediate result concerning the polynomial $\LN P$.
\begin{lmm}\label{lm:Rnorm}
\defOp
Consider a given polynomial $P\in\mathbb C[x,y]$.
For any $\mfl\in\mathbb N$ and any $I\in\mathbb N$ such that  $I\leq\mfl+1$, the quantity $\dx^I\dy^{\mfl+1-I} \left[ \LN P \right]$ can be expressed as a linear combination of products $\prod_{t=1}^\mu \dx^{i_t+I_t}\dy^{j_t+J_t}P $, with $\sum_{t=1}^\mu I_t =I$, $\sum_{t=1}^\mu J_t =\mfl+1-I$, $\sum_{t=1}^\mu i_t =k$, and $\sum_{t=1}^\mu j_t =\ell-k$. Moreover, for each product term, there exists $t_0\in\mathbb N$, $1\leq t_{0}\leq \mu$ such that $I_{t_0}\neq 0$ or $J_{t_0}\neq 0$.
\end{lmm}
\begin{proof}
The quantity $\LN P$ can be expressed, from Faa di Bruno's formula, as a linear combination of products $\prod_{m=1}^s\left( \dx^{i_m}\dy^{j_m}P \right)^{k_m}$, with  $(i_{m_1},j_{m_1})\neq (i_{m_2},j_{m_2})$ for all $m_1\neq m_2$, $\sum_{m=1}^sk_m = \mu$, $\sum_{m=1}^sk_mi_m =k$, and $\sum_{m=1}^sk_mj_m =\ell-k$. Therefore $\LN P$ can also be expressed, repeating terms, as a linear combination of products $\prod_{t=1}^\mu \dx^{i_t}\dy^{j_t}P $, with possibly $(i_{m_1},j_{m_1})= (i_{m_2},j_{m_2})$ for $m_1\neq m_2$, $\sum_{t=1}^\mu i_t =k$, and $\sum_{t=1}^\mu j_t =\ell-k$. So the quantity $\dx^I\dy^{\mfl+1-I} \left[ \LN P\right]$ can be expressed, from Leibniz's rule, as a linear combination of products $\prod_{t=1}^\mu \dx^{i_t+I_t}\dy^{j_t+J_t}P $, with $\sum_{t=1}^\mu I_t =I$ and $\sum_{t=1}^\mu J_t =\mfl+1-I$.

Consider such a given product term $\prod_{t=1}^\mu \dx^{i_t+I_t}\dy^{j_t+J_t}P $, and suppose that for all $t$ $I_t=J_t=0$. Then $I=\sum_{t=1}^\mu I_t =0$ and $\mfl+1-I=\sum_{t=1}^\mu J_t =0$, which is impossible since $\mfl+1>0$.
\end{proof}

The two following results gather the consequences of this choice on $\lambda_{i+\mfn,j}$s computed in Algorithm \ref{algo}.
\begin{prop}\label{prop:ln0norm}
\defOp
Under Hypotheses \ref{hyp} and \ref{hyp2} consider a solution to Problem \eqref{thesyst} constructed thanks to Algorithm \ref{algo}  with all the prescribed values $\lambda_{i,j}$ such that $i<\mfn$ and $i+j\neq 1$ set to zero, and 
\begin{equation}\label{norm:oz}
\begin{pmatrix}\lambda_{1,0}\\\lambda_{0,1}\end{pmatrix}
=
\cpxi\kappa A^{-1}D^{-1/2}
\begin{pmatrix}\cos\theta\\\sin\theta\end{pmatrix}
\end{equation} 
for some $\theta\in\mathbb R$ and $\kappa\in\mathbb C^*$.
As an element of  $\mathbb C[\lambda_{1,0},\lambda_{0,1}]$, $\lambda_{\mfn,0}$ can be expressed as a polynomial of degree at most equal to $\mfn-1$, and its coefficients are independent of $\theta$.
\end{prop}
Note that once we impose this condition on $\lambda_{1,0},\lambda_{0,1}$ any element of $\mathbb C[\lambda_{1,0},\lambda_{0,1}]$ can be expressed by different polynomials, possibly with different degrees, simply because under Hypothesis \ref{hyp2} and \eqref{norm:oz} we have
$$
\sum_{k=0}^\mfn \alpha_{k,\mfn-k} (x_0,y_0)\lambda_{1,0}^k\lambda_{0,1}^{\mfn-k}
=
\left( -\kappa^2\right)^{\frac{\mfn}2}.
$$
 See paragraph 3.2 in \cite{LMinterp} for an algebraic view point on this comment.
\begin{proof}
Since
\begin{equation}\label{eq:ln0}
\lambda_{\mfn,0}=\frac{1}{\mfn!\alpha_{\mfn,0}(x_0,y_0)}\left( 
-\alpha_{0,1}(x_0,y_0)\lambda_{0,1}
-\alpha_{1,0}(x_0,y_0)\lambda_{1,0}
-\LN P(x_0,y_0)- \alpha_{0,0}(x_0,y_0)\right),
\end{equation}
again the term to investigate is $\LN P(x_0,y_0)$. Lemma \ref{lm:LNP} identifies products of $\mfn$ terms in $\LN P$, and from the definition of  $\LN$ they appear in the following linear combination
$$
\sum_{k=0}^\mfn k!(\mfn-k)! \alpha_{k,\mfn-k} \frac{(\dx P)^k}{k!}\frac{(\dy P)^{\mfn-k}}{(\mfn-k)!}
=
\sum_{k=0}^\mfn \alpha_{k,\mfn-k} (\dx P)^k(\dy P)^{\mfn-k}.
$$
Back to the expression of $\lambda_{\mfn,0}$, and thanks to Hypothesis \ref{hyp2}, the only possible terms of degree $\mfn$ therefore appear in the following linear combination:
$$
\begin{array}{rl}
\displaystyle
\sum_{k=0}^\mfn \alpha_{k,\mfn-k}(x_0,y_0) (\lambda_{1,0})^k(\lambda_{0,1})^{\mfn-k}
%&=
%(\gamma_1(\lambda_{1,0})^2
%+\gamma_2\lambda_{1,0}\lambda_{0,1}
%+\gamma_3(\lambda_{0,1})^2)^{\frac{\mfn}{2}},\\
%&=\left(
%  \mu_1(A_{11}\lambda_{1,0}+A_{12}\lambda_{0,1})^2+
% +\mu_2(A_{21}\lambda_{1,0}+A_{22}\lambda_{0,1})^2
%\right)^{\frac{\mfn}{2}}\\
&=\left((\luoz\ \luzo)\Gamma  \begin{pmatrix}\lambda_{1,0}\\\lambda_{0,1}\end{pmatrix}\right)^{\frac{\mfn}{2}}
=\left((\cpxi \kappa)^2(\luoz\ \luzo)A^tD  A \begin{pmatrix}\lambda_{1,0}\\\lambda_{0,1}\end{pmatrix}\right)^{\frac{\mfn}{2}}\\
&=\left(-\kappa^2(\cos\theta\ \sin\theta) \begin{pmatrix}\cos\theta\\\sin\theta\end{pmatrix}\right)^{\frac{\mfn}{2}}
=(-\kappa)^\mfn
\end{array}
$$
Finally thanks to \eqref{norm:oz}, the only terms of degree $\mfn$ in \eqref{eq:ln0} can be expressed as a polynomial of degree at most equal $\mfn-1$.
\end{proof}
\begin{prop}\label{prop:linjnorm}
\defOp
Under Hypotheses \ref{hyp} and \ref{hyp2} consider a solution to Problem \eqref{thesyst} constructed thanks to Algorithm \ref{algo}  with all the fixed values $\lambda_{i,j}$ such that $i<\mfn$ and $i+j\neq 1$ set to zero, and 
$$
\begin{pmatrix}\lambda_{1,0}\\\lambda_{0,1}\end{pmatrix}
=
\cpxi \kappa A^{-1}D^{-1/2}
\begin{pmatrix}\cos\theta\\\sin\theta\end{pmatrix}
$$
%$$
%\begin{pmatrix}\lambda_{1,0}\\\lambda_{0,1}\end{pmatrix}
%=
%S^{-1}D^{-1/2}
%\begin{pmatrix}\cos\theta\\\sin\theta\end{pmatrix},
%$$ 
for some $\theta\in\mathbb R$ and $\kappa\in\mathbb C^*$.
 As an element of $\mathbb C[\lambda_{1,0},\lambda_{0,1}]$, each $\lambda_{i+\mfn,j}$ can be expressed as a polynomial of degree at most equal to  $i+j+\mfn-1$, and its coefficients are independent of $\theta$.
\end{prop}
\begin{proof}
From Algorithm \ref{algo} the expression of $\lambda_{I+\mfn,\mfl-I}$ reads
\begin{equation}\label{eq:lipnnorm}
\begin{array}{rl}
\lambda_{I+\mfn,\mfl-I}&\displaystyle
=\frac{1}{\mathsf T^\mfl_{I+\mfn+1,I+\mfn+1}}\left(\mathsf B^\mfl_{I+\mfn+1} - 
\sum_{k=0}^{\mfn-1}\mathsf T^\mfl_{I+\mfn+1,I+k+1} \lambda_{I+k,\mfn+\mfl-I-k}\right)\\
&\displaystyle
=\frac{I!}{(\mfn+I)!\alpha_{\mfn,0}(x_0,y_0)}\left(\mathsf \RHS_{I,\mfl-I} - 
\sum_{k=0}^{\mfn-1}
\frac{(I+k)!(\mfn-k+\mfl-1)!}{I!(\mfl-I)!} 
\alpha_{k,\mfn-k}(x_0,y_0)
\lambda_{I+k,\mfn+\mfl-I-k}\right).
\end{array}
\end{equation}
We will proceed again by induction on $\mfl$:
\begin{enumerate}
\item the result has been proved to be true for $\mfl = 0$ in Proposition \ref{prop:ln0norm} ;
\item suppose the result is true for $\mfl\in\mathbb N$ as well as for all $\tilde\mfl\in\mathbb N$ such that $\tilde\mfl\leq\mfl$, then we focus on $\RHS_{I,\mfl+1-I}$, given by
\begin{align*}
\begin{split}
\RHS_{0,\mfl+1}
&=  \sum_{k=0}^{\mfn} 
\sum_{\tilde j=0}^{\mfl}
 \left(k+\tilde i\right)!\frac{\left({\mfn}-k+\tilde j\right)!}{\tilde j !}
 \Dop{0}{\mfl+1-\tilde j}\alpha_{k,{\mfn}-k} (x_0,y_0) \lambda_{k,{\mfn}-k+\tilde j}\\
&\phantom =+\sum_{\ell = 1}^{\mfn-1} \sum_{k=0}^\ell 
\sum_{\tilde j=0}^{\mfl+1}
 \left(k\right)!\frac{\left({\ell}-k+\tilde j\right)!}{\tilde j !}
 \Dop{0}{\mfl+1-\tilde j}\alpha_{k,\ell-k} (x_0,y_0)  \lambda_{k,{\ell}-k+\tilde j} 
\\
&\phantom = - \Dop{0}{\mfl+1}\left[ \LN P \right](x_0,y_0) - \Dop{0}{\mfl+1}\alpha_{0,0}(x_0,y_0) \text{ for } I=0  \text{ ; and } 
\end{split}
\end{align*}
\begin{align*}
\begin{split}
&\RHS_{I,\mfl+1-I}
\\&=- \sum_{k=0}^{\mfn} 
\sum_{\tilde i=0}^{I-1}\sum_{\tilde j=0}^{\mfl-I}
 \frac{\left(k+\tilde i\right)!\left({\mfn}-k+\tilde j\right)!}{\tilde i!\tilde j!}
\Dop{I-\tilde i}{\mfl+1-I-\tilde j}\alpha_{k,{\mfn}-k} (x_0,y_0) \lambda_{k+\tilde i,{\mfn}-k+\tilde j}\\
&\phantom =-\sum_{\ell = 1}^{\mfn-1} \sum_{k=0}^\ell 
\sum_{\tilde i=0}^{I}\sum_{\tilde j=0}^{\mfl+1-I}
\frac{ \left(k+\tilde i\right)!\left({\ell}-k+\tilde j\right)!}{\tilde i!\tilde j!}
\Dop{I-\tilde i}{\mfl+1-I-\tilde j}\alpha_{k,\ell-k} (x_0,y_0)  \lambda_{k+\tilde i,{\ell}-k+\tilde j}
\\
&\phantom = -\Dop{I}{\mfl+1-I}\left[ \LN P \right](x_0,y_0) - \Dop{I}{\mfl+1-I}\alpha_{0,0}(x_0,y_0) \text{ otherwise ;}
\end{split}
\end{align*}

all the linear terms in $N_{I,\mfl+1-I}$, as elements of $\mathbb C[\lambda_{1,0},\lambda_{0,1}]$, by hypothesis have degree at most equal to $(I+\mfn)+(\mfl+1-I)-1=\mfn+\mfl$, and thanks to Lemma \ref{lm:Rnorm} all the non-linear terms in $N_{I,\mfl+1-I}$ can be expressed as  a linear combination  of products $\prod_{t=1}^\mu \lambda_{a_t,b_t}$ where the indices satisfy $\sum_{t=1}^\mu (a_t+b_t)\leq \mfl+1+\mfn$ ; in each such product, as element of $\mathbb C[\lambda_{1,0},\lambda_{0,1}]$, each $\lambda_{a_t,b_t}$ is either of degree $a_t+b_t=1$ if $(a_t,b_t)\in\{(0,1),(1,0)\}$, or of degree at most equal to $a_t+b_t-1$ otherwise by hypothesis ; from Lemma \ref{lm:Rnorm} there is at least one $t_0$ such that $(a_{t_0},b_{t_0})\notin\{(0,1),(1,0)\}$, therefore each product $\prod_{t=1}^\mu \lambda_{a_t,b_t}$, as element of $\mathbb C[\lambda_{1,0},\lambda_{0,1}]$, can be expressed as a polynomial of degree at most $\left(\sum_{t=1}^\mu (a_t+b_t)\right) -1\leq \mfl+\mfn$ ; so all terms in $N_{I,\mfl+1-I}$, as elements of $\mathbb C[\lambda_{1,0},\lambda_{0,1}]$, have degree at most equal to $\mfn+\mfl$ ; the last step is to prove that the $\lambda_{I+k,\mfn+\mfl+1-I-k}$ are also of degree at most equal to $\mfn+\mfl$, and we will proceed by induction on $I$:
\begin{enumerate}
\item for $I=0$, all  $\lambda_{I+k,\mfn+\mfl+1-I-k}$ for $0\leq k\leq \mfn-1$ satisfy the two conditions $I+k<\mfn$ and $I+k+\mfn+\mfl+1-I-k = \mfn+\mfl+1\neq 1$ so they are all prescribed to zero and their degree as element of $\mathbb C[\lambda_{1,0},\lambda_{0,1}]$ is at most equal to $\mfn+\mfl$ that ;
\item suppose that, for a given $I\in \mathbb N$, the $\lambda_{\tilde I+k,\mfn+\mfl+1-\tilde I-k}$  for all $\tilde I\in \mathbb N$ such that $\tilde I\leq I$  are also of degree at most equal to $\mfn+\mfl$ then it is clear from Equation \eqref{eq:lipnnorm} that $\lambda_{I+1+\mfn,\mfl-I-1}$ is also of degree at most  equal to $\mfn+\mfl$.
\end{enumerate}
This concludes the proof.
\end{enumerate}
\end{proof}

Finally, since we are interested in the local approximation properties of GPWs, it is natural to study their Taylor expansion coefficients, and how they can be expressed as elements of $\mathbb C[\lambda_{1,0},\lambda_{0,1}]$. In particular we will find what is the link between the Taylor expansion coefficients of a GPW, $ \dx^i\dy^j \varphi \left(x_0,y_0\right)/(i!j!)$, and that of the corresponding PW, $(\luzo)^j (\luoz)^i/(i!j!)$.

\begin{prop}\label{prop:derphi}
\defOp
Under Hypotheses \ref{hyp} and \ref{hyp2} consider a solution to Problem \eqref{thesyst} constructed thanks to Algorithm \ref{algo}  with all the fixed values $\lambda_{i,j}$ such that $i<\mfn$ and $i+j\neq 1$ set to zero, and 
$$
\begin{pmatrix}\lambda_{1,0}\\\lambda_{0,1}\end{pmatrix}
=
\cpxi \kappa A^{-1}D^{-1/2}
\begin{pmatrix}\cos\theta\\\sin\theta\end{pmatrix},
$$ 
for some $\theta\in\mathbb R$ and $\kappa\in\mathbb C^*$, and the corresponding $\displaystyle \varphi(x,y)=\exp \sum_{ 0\leq i+j\leq q+1} \luij (x-x_0)^i (y-y_0)^j$.
%Consider a GPW adapted to the operator $-\Delta+\beta$: $\displaystyle \varphi=\exp \sum_{ 0\leq i+j\leq q+1} \luij (x-x_0)^i (y-y_0)^j$.
 Then for all $(i,j) \in \mathbb N^2$ such that $i+j\leq q+1$ the difference
\begin{equation}\label{eq:derphi}
R_{i,j} : =  \dx^i\dy^j \varphi \left(x_0,y_0\right)%\frac{\partial^{i+j}\varphi \left(\overrightarrow g\right)}{\partial{x^i} \partial{y^j}}
 - (\luzo)^j (\luoz)^i%,
\end{equation}
can be expressed as an element of  $\mathbb C[\lambda_{1,0},\lambda_{0,1}]$ such that
\begin{itemize}
\item[$\bullet$] its total degree satisfies ${\rm d} R_{i,j} \leq i+j-1$,
\item[$\bullet$] its coefficients only depend on $i$, $j$, and on the derivatives of the PDE coefficients $\alpha$ evaluated at  $(x_0,y_0)$ but do not depend on $\theta$.
\end{itemize}
\end{prop}
\begin{proof} 
Applying the chain rule introduced in Appendix \ref{app:bivFDB} to the GPW $\varphi$ one gets for all $ (i,j) \in \mathbb N^2$,
\begin{equation*}{%\label{eq:CR}%chain rule
\dx^i\dy^j \varphi \left(x_0,y_0\right)% \frac{\partial^{i+j}\varphi \left(\overrightarrow g\right)}{\partial{x^i} \partial{y^j}}
 = i!j!\sum_{\mu=1}^{i+j} \sum_{s=1}^{i+j} \sum_{p_s((i,j),\mu)} \prod_{l=1}^s \frac{(\lu{i_l}{j_l})^{k_l}}{k_l!},
}
\end{equation*}
where $p_s((i,j),\mu)$ is the set of partitions of $(i,j)$ with length $\mu$:
\begin{equation*}{%\label{eq:p_s}
 \left\{ (k_l,(i_l,j_l))_{l\in [\![ 1,s ]\!]}:k_l\in\mathbb N^*, 0\prec (i_1,j_1)\prec \dots \prec(i_l,j_l), \sum_{l=1}^s k_l = \mu,  \sum_{l=1}^s k_l(i_l,j_l) = (i,j) \right\}.
}
\end{equation*}
%See Appendix \ref{app:bivFDB} for a definition of this ordering relation. 
For each partition $(k_l,(i_l,j_l))_{l\in [\![ 1,s ]\!]}$ of $(i,j)$,  the corresponding product term, considered as an element of   $\mathbb C[\lambda_{1,0},\lambda_{0,1}]$, has degree $\displaystyle Deg \ \prod_{l=1}^s (\lu{i_l}{j_l})^{k_l}=\sum_{l=1}^s k_l Deg\ \lu{i_l}{j_l}$. Combining Proposition \ref{prop:linjnorm}  and the fact that $\lambda_{i,j}=0$ for all $(i,j)$ such that $1<i+j<\mfn$, we can conclude that this degree is at most equal to
\begin{equation}{\label{eq:sumdeg}
\sum_{i_l=0,j_l=1} k_l j_l+\sum_{i_l=1,j_l=0} k_l i_l
+\sum_{1<i_l+j_l<\mfn} k_l \cdot 0 
+\sum_{i_l+j_l\geq\mfn} k_l (i_l+j_l-1) .
}
\end{equation}
%where the two first sums contain at most one term each. 

The partition with two terms $(i,j)=j(0,1)+i(1,0)$ corresponds to the term $(\luzo)^j(\luoz)^i$, which is the leading term in $\dx^i\dy^j \varphi \left(x_0,y_0\right)$. Indeed, any  other partition will include at least one term such that $i_l+j_l>1$, and the degree corresponding to this term within the product is either $ k_l \cdot 0$ or $ k_l (i_l+j_l-1)$, and in both case it is at most equal to $k_l (i_l+j_l)-1$. As a result, the degree of the product term in \eqref{eq:sumdeg} is necessarily less than $\displaystyle \sum_{l=1}^s k_l(i_l+j_l)= i+j$. So $R_{i,j}$, which is defined as the difference between $\dx^i\dy^j \varphi \left(x_0,y_0\right)$ and its leading term $(\luzo)^j (\luoz)^i$, is as expected of degree less than $i+j$.

Finally, the coefficients of $R_{i,j}$ share the same property as the coefficients of $\luij$s from Propositions \ref{prop:linjnorm}.
\end{proof}

\begin{rmk}
As mentioned in Remark \ref{rmk:sim}, under the hypothesis $\alpha_{0,\mfn}(x_0,y_0)\neq 0$, an algorithm very similar to Algorithm \ref{algo} would construct the polynomial coefficients of a GPW, fixing the values of $\{ \lambda_{i,j}, 0\leq j\leq \mfn -1, 0\leq i\leq q+\mfn -1 -j \}$. The corresponding version of Proposition \ref{prop:derphi} could then be proved essentially by exchanging the roles of $i$ and $j$ in all the proofs.
\end{rmk}

%%%%%%%%%%%%%%%%%%%%%%%%%%%%%%%%%%%%%%%%%%%%
\subsection{Local set of GPWs}
\label{sec:set}
At this point for a given value of $\theta\in\mathbb R$ we can construct a GPW as a function $\varphi=\exp P$ where the polynomial $P$ is a solution to Problem \eqref{thesyst} constructed thanks to Algorithm \ref{algo}  with all the fixed values $\lambda_{i,j}$ such that $i<\mfn$ and $i+j\neq 1$ set to zero, and 
$$
\begin{pmatrix}\lambda_{1,0}\\\lambda_{0,1}\end{pmatrix}
=
\cpxi \kappa A^{-1}D^{-1/2}
\begin{pmatrix}\cos\theta\\\sin\theta\end{pmatrix}.
$$ 
This parameter $\theta$ is then equivalent to the direction a classical plane wave, while $|\kappa|$ is equivalent to the wave number of a classical plane wave, and $\theta$ will now be used to construct a set of GPWs.  Under Hypotheses \ref{hyp} and \ref{hyp2}, by choosing $p$ different angles $\{\theta_l,l\in\mathbb N^*, l\leq p\}\in\mathbb R^p$, we can consider $p$ solutions to Problem \eqref{thesyst} to construct $p$ GPWs.
\begin{dfn}\label{df:norm}
\defOp
\defp
 Under Hypotheses \ref{hyp} and \ref{hyp2}, consider the normalization $\lambda_{i,j}$ such that $i<\mfn$ and $i+j\neq 1$ set to zero, and 
$$
\begin{pmatrix}\lambda_{1,0}^l\\\lambda_{0,1}^l\end{pmatrix}
=
\kappa A^{-1}D^{-1/2}
\begin{pmatrix}\cos\theta_l\\\sin\theta_l\end{pmatrix},
\text{ for }\{ \theta_l\in[0,2\pi), \forall l\in\mathbb N^*, l\leq p, \theta_{l_1}\neq\theta_{l_2} \ \forall l_1\neq l_2,\kappa\in\mathbb C^*\}.
$$
The set of corresponding GPWs contructed from Algorithm \ref{algo} will be denoted hereafter by
$$\VGPW =\{
\varphi_l:=\exp P_l, \forall l\in\mathbb N^*, l\leq p
\}.$$
\end{dfn}

%%%%%%%%%%%%%%%%%%%%%%%%%%%%%%%%%%
\section{Interpolation properties}\label{sec:int}
This section is restricted to operators of order $\mfn=2$. 

We now have built tools to turn to the interpolation properties of GPWs. In particular, since the GPWs are constructed locally, and will be defined separately on each mesh element, we focus on local interpolation properties. Given a partial differential operator $\mathcal L$, a point $ (x_0,y_0)\in\mathbb R^2$ and an integer $n\in\mathbb N$, the question is whether we can find a finite dimensional space $\mathbb V_h\subset \mathcal C^{\infty}$, with the following property:
\begin{equation}\label{IntPb}
\begin{array}{l}
\forall u%\in \mathcal C^\infty 
\text{ satisfying }\mathcal Lu=0,
\exists u_a\in\mathbb V_h \text{ s. t. }
\forall (x,y)\in\mathbb R^2,
|u(x,y)-u_a(x_0,y_0)|\leq C \| (x,y)-(x_0,y_0) \|^{n+1},
\end{array}
\end{equation}
that is to say there exists an element of $\mathbb V_h$ whose Taylor expansion at $(x_0,y_0)$ matches the Taylor expansion of $u$  at $(x_0,y_0)$ up to order $n$, for any solution $u$ of the PDE $\mathcal Lu=0$. If $\{f_i,i\in\mathbb N^*, i\leq p\}$ is a basis of $\mathbb V_h$, this can be expressed in terms of linear algebra. Consider the vector space $\mathbb F$ and the matrix $\mathsf M\in\mathbb C^{(n+1)(n+2)/2\times p}$ defined as follows:
$$
\mathbb F :=\left\{\mathsf F\in\mathbb C^{(n+1)(n+2)/2},\exists u\text{ satisfying }\mathcal Lu=0 \text{ s.t. } \mathsf F_{\frac{(k_1+k_2)(k_1+k_2+1)}{2}+k+2+1} = \dx^{k_1} \dy^{k_2} u (x_0,y_0) /(k_1!k_2!)\right\},
$$
\begin{equation}\label{eq:defM}
\mathsf M_{\frac{(k_1+k_2)(k_1+k_2+1)}{2}+k_2+1,i} := \dx^{k_1} \dy^{k_2} f_i (x_0,y_0)/(k_1!k_2!).
\end{equation}
Then \eqref{IntPb} is equivalent to 
$$
\forall \mathsf F \in\mathbb F, \exists \mathsf X\in\mathbb C^{p} \text{ s.t. } 
\mathsf M \mathsf X = \mathsf F,
$$
and the choice of $p$, the number of basis functions, will be crucial to our study.

Our previous work on GPWs was focused on the Helmholtz equation, i.e. $\mathcal L = -\Delta+\beta(x,y)$, and in that case the classical PWs are exact solutions to the PDE if the coefficient is constant $\beta(x,y)=-\kappa^2$. However, even though the proof of the interpolation properties of GPWs relies strongly on that of classical PWs, it is not required, in order to obtain the GPW result, for classical PW to be solutions of the constant coefficient equation \cite{LMinterp}. Indeed, what will be central to the proof that follows is the rank of the matrix $\mathsf M$ associated to a set of reference functions - not necessarily classical PWs - that are not required to satisfy any PDE.
For the Helmholtz equation, the reference functions used in \cite{LMinterp} were classical PWs if $\beta(x_0,y_0)<0$ and real exponentials if $\beta(x_0,y_0)>0$, and the structure of the proof provides useful guidelines for what follows. 

%%%%%%%%%%%
\subsection{Comments on a standard reference case}
\label{sec:PWinterp}
Interpolation properties of classical plane waves were already presented  for instance in \cite{LMinterp}, and in \cite{cess}, however the link between desired order of approximation $n$ and number $p$ of basis functions was simply provided as $p=2n+1$. We present here a new perspective, focusing on properties of trigonometric functions, to justify this choice. The corresponding set of trigonometric functions will constitute the reference case at the heart of the GPWs interpolation properties.

\begin{dfn}\label{dfn:MnCR} 
Consider a given $n\in\mathbb N^*$ and a given $p\in\mathbb N^*$.
Considering for some $\kappa\in\mathbb R^*$ a space $\mathbb V_h^\kappa=Span\{\exp \mathrm i \kappa (\cos\theta_l(x-x_0)+\sin\theta_l(y-y_0)), 1\leq l\leq p,\theta_l\in[0,2\pi),\theta_{l_1}\neq\theta_{l_2} \ \forall l_1\neq l_2\}$ of classical PWs, we define the corresponding matrix \eqref{eq:defM} for the plane wave functions spanning $\mathbb V_h^\kappa$, denoted $\mathsf M^C$, as well as the reference matrix $\mathsf M^R$, by
$$\forall (k_1,k_2)\in\mathbb N^2, k_1+k_2\leq n,\
\left\{\begin{array}{l}
\left(\mathsf M^C_n\right)_{\frac{(k_1+k_2)(k_1+k_2+1)}{2}+k_2+1,l} := (\mathrm i \kappa)^{k_1+k_2} \cos\theta_l^{k_1} \sin\theta_l^{k_2}/(k_1!k_2!),\\
\left(\mathsf M^R_n\right)_{\frac{(k_1+k_2)(k_1+k_2+1)}{2}+k_2+1,l} :=  \cos\theta_l^{k_1} \sin\theta_l^{k_2}/(k_1!k_2!).
\end{array}
\right.
$$
If we denote by $\mathsf D_n^{RC}=diag(d_k^{RC}, k\text{ from } 1 \text{ to }n+1)$ the block diagonal matrix with blocks of increasing size $d_k^{RC}=(\mathrm i \kappa)^{k-1} I_{k}\in\mathbb C^{k\times k}$, it is evident that $\mathsf M^C_n = \mathsf D_n^{RC}\mathsf M^R_n$, therefore trigonometric functions are closely related to interpolation properties of PWs.
\end{dfn}

Consider the two sets of functions
$$ \mathcal F_n=\{\theta\mapsto\cos^k \theta\sin^{K-k}\theta/(k!(K-k)!), 0\leq k\leq K\leq n\},\
\quad \text{and}\quad
\mathcal G_n=\{\theta\mapsto\exp i k\theta,-n\leq k\leq n\}.
$$
The first one, $\mathcal F_n$, is a set of $(n+1)(n+2)/2$ functions. The second one, $\mathcal G_n$, is a set of $2n+1$ linearly independent functions: indeed, any null linear combination of these functions $\sum_{-n \leq k \leq n} \nu_k \exp(i k \theta)$ would define a function $f(x) = \sum_{-n \leq k \leq n} \nu_k x^k$ that would be uniformly null on the circle $|x|=1$, implying that the polynomial $x^n.f(x)$ has an infinite number of roots ; hence all its coefficients $\nu_k$ are null. Moreover since
$$
\left\{\begin{array}{l}\displaystyle
\cos(\theta)^k\sin(\theta)^{K-k} = 
\left(\frac{e^{i\theta}+e^{-i\theta}}{2 }\right)^k
\left(\frac{e^{i\theta}-e^{-i\theta}}{2i}\right)^{K-k}
=
\frac 1 {2^K i^{K-k}} 
\sum_{l=0}^k\sum_{L=0}^{K-k}
\begin{pmatrix}k\\l\end{pmatrix}%C_k^l
\begin{pmatrix}K-k\\L\end{pmatrix}%C_{K-k}^L
 e^{ i (2l+2L-K)\theta},
\\\displaystyle
\phantom{\cos(\theta)^k\sin(\theta)^{K-k} =\left(\frac{e^{i\theta}+e^{-i\theta}}{2 }\right)^k } 
\text{with }-K\leq 2l+2L-K\leq K\Rightarrow  \mathcal F_n \subset Span\ \mathcal G_n,
\\\displaystyle
\exp \pm i k \theta = \sum_{s=0}^k 
\begin{pmatrix}k\\s\end{pmatrix}%C_k^s 
(\pm i)^s \cos(\theta)^{k-s}\sin(\theta)^s
\Rightarrow  \mathcal G_n \subset Span\ \mathcal F_n,
\end{array}\right.
$$
 we then have that $Span\ \mathcal F_n = Span\ \mathcal G_n$, and in particular the space spanned by $\mathcal F_n$ is of dimension $2n+1$.

Consider any matrix $\matM\in\mathbb C^{(n+1)(n+2)/2\times N_p}$ defined for some $\{\theta_l\}_{1\leq l\leq N_p}\in\left(\mathbb R\right)^{N_p}$, with $N_p>2n+1$, by
$$
\matM_{il} = f_i(\theta_l), \text{ where } f_i \text{ denotes the elements of }\mathcal F_n \text{ (independently of their numbering).}
$$
Its rank is at most $2n+1$. This is a simple consequence of the fact that the dimension of $Span\ \mathcal F_n$ is $2n+1<(n+1)(n+2)/2$: indeed, this implies that there exists a matrix $\matC\in\mathbb C^{((n+1)(n+2)/2-2n-1)\times (n+1)(n+2)/2}$ of rank $(n+1)(n+2)/2-2n-1$ such that
$$
\forall i\in\mathbb N, 1\leq i\leq (n+1)(n+2)/2-2n-1, \sum_{j=1}^{(n+1)(n+2)/2} \matC_{ij} f_j = 0,
$$
and therefore $\matC\matM =0_{((n+1)(n+2)/2-2n-1)\times N_p}$ ; as a result the $N_p$ columns of $\matM$ belong to the kernel of $\matC$, which is of dimension $2n+1$; so the rank of $\matM$ is at most $2n+1$. In particular %NEED SOME q I GUESS 
the matrix $\mathsf M_n^R$ introduced in Definition \ref{dfn:MnCR} is such a matrix $\mathsf A^{\mathcal F}$, and is therefore of rank at most $2n+1$. 

%BACK TO PW MATRIX 
We know that $\mathsf M^C_n=\mathsf D_n^{RC}\mathsf M^R_n$ and $\mathsf D_n^{RC}$ is non-singular, so $rk(\mathsf M_n^C)=rk(\mathsf M^R_n)$. The rank of $\mathsf M^C_n$ is at most equal to $2n+1$ for any choice of angles $\{\theta_l\in\mathbb R,1\leq l\leq p\}$. It was previously proved in Lemma 2 from \cite{LMinterp} that for $p=2n+1$ and directions such that $\{\theta_l\in [0,2\pi),1\leq l\leq p, l_1\neq l_2\Rightarrow \theta_{l_1}\neq\theta_{l_2}\}$ the matrix  $\mathsf M_n^C$ has rank $2n+1$. A trivial corollary of this proof is that, for any choice of $p$ distinct angles in $[0,2\pi)$,  
\begin{equation}\label{eq:2n+1}
rk  (\mathsf M^C_n)=2n+1=rk(\mathsf M^R_n) \Leftrightarrow p\geq 2n+1.
\end{equation}

In \cite{LMinterp} we also proved that the space $\mathbb F$ for the constant coefficient Helmholtz operator is equal to the range of $\mathsf M_n^C$ for the corresponding wave number $\kappa$.
As a direct consequence, a space  $\mathbb V_h^{\kappa}=Span\{\exp \mathrm i \kappa (\cos\theta_l(x-x_0)+\sin\theta_l(y-y_0)), 1\leq l\leq p\}$ for any choice of distinct angles in $[0,2\pi)$ satisfies the interpolation property \eqref{IntPb} for the Helmholtz equation if and only if $p\geq 2n+1$.

\subsection{Generalized Plane Wave case}
\label{sec:GPWinterp}
In order to prove that a GPW space $Span\ \VGPW$ (introduced in Definition \ref{df:norm}) satisfies the interpolation property \eqref{IntPb}, we will rely on Proposition \ref{prop:derphi} to study the rank of the matrix \eqref{eq:defM} built from GPWs. As in the Helmholtz case, the proof relates the GPW matrix to the reference matrix, but here via an intermediate transition matrix.
%Here now I guess just need to say that the proofs of HElm case apply directly here from our theorem on Rij

\begin{dfn}\label{dfn:Mn}
\defOp
For the corresponding set of GPWs, $\VGPW=\{\varphi_l:=\exp P_l, \forall l\in\mathbb N^*, l\leq p,\theta_l\in[0,2\pi),\theta_{l_1}\neq\theta_{l_2} \ \forall l_1\neq l_2,\kappa\in\mathbb C^*\}$, we define the corresponding matrix \eqref{eq:defM}, denoted $\mathsf M_n$, as well as the transition matrix $\mathsf M_n^{Tr}$, by
$$
\left\{\begin{array}{l}
\left(\mathsf M^{Tr}_n\right)_{\frac{(k_1+k_2)(k_1+k_2+1)}{2}+k_2+1,l} := (\luoz^l)^{k_1} (\luzo^l)^{k_2}/(k_1!k_2!),\\
\left(\mathsf M_n\right)_{\frac{(k_1+k_2)(k_1+k_2+1)}{2}+k_2+1,l} :=  \dx^{k_1}\dy^{k_2} \varphi_l(x_0,y_0)/(k_1!k_2!).
\end{array}
\right.
$$
\end{dfn}
%
%Consider a given $\mfn\in\mathbb N^*$ and a given set of complex-valued functions $\alpha=\{ \alpha_{k_1,k_2},0\leq k_1+k_2\leq \mfn \}$, as well as a given point $(x_0,y_0)$. Under Hypotheses \ref{hyp} and \ref{hyp2}, consider the normalization $\lambda_{i,j}$ such that $i<\mfn$ and $i+j\neq 1$ set to zero, and 

We first relate the transition matrix $\mathsf M_n^{Tr}$ to the reference matrix $\mathsf M_n^R$.
\begin{lmm}\label{lmm:DRT}
Consider an open set $\Omega\subset\mathbb R^2$, $(x_0,y_0)\in\Omega$, a given $(\mfn,n,p,q)\in(\mathbb N^*)^4$, $\mfn\geq 2$,  and a given set of complex-valued functions $\alpha=\{ \alpha_{k_1,k_2}\in\mathcal C^{q-1}(\Omega),0\leq k_1+k_2\leq \mfn \}$, the corresponding partial differential operator $\Lal$ and set of GPWs $\VGPW$. 
There exists a block diagonal non-singular matrix $\mathsf D_n^{RT}$ such that $\mathsf M_n^{Tr} = \mathsf D_n^{RT}\mathsf M^R_n$, independently of the number $p$ of GPWs in $\VGPW$.
\end{lmm}

\begin{proof} 
As long as there are four complex numbers $a,b,c,d$ such that
$$\forall p\in\mathbb N, 1\leq l\leq p,\
\begin{pmatrix} \luoz^l\\\luzo^l\end{pmatrix} 
=
\begin{pmatrix} a & b\\c&d\end{pmatrix} 
\begin{pmatrix} \cos\theta_l\\\sin\theta_l\end{pmatrix} ,
$$
then the diagonal blocks of $\mathsf D_n^{RT}=diag(d_K^{RT},K\text{ from } 0 \text{ to } n)$ of increasing size $d_K^{RT}\in\mathbb C^{(K+1)\times (K+1)}$ can be built thanks to the following binomial formula
$$
(\luoz^l)^{K-k}(\luzo^l)^k = \sum_{i=0}^{K-k}\sum_{j=0}^k 
\begin{pmatrix} K-k\\i\end{pmatrix}
\begin{pmatrix} k\\j\end{pmatrix}
a^ic^jb^{K-k-i}d^{K-k-j} (\cos\theta_l)^{i+j} (\sin\theta_l)^{K-i-j}
$$
since the coefficient of this linear combination of trigonometric functions are independent on $l$. 
\end{proof}

The following step is naturally to relate the GPW matrix $\mathsf M_n$ to the reference matrix $\mathsf M_n^R$.
\begin{prop}\label{prop:LRn}
Consider an open set $\Omega\subset\mathbb R^2$, $(x_0,y_0)\in\Omega$, a given $(\mfn,n,p,q)\in(\mathbb N^*)^4$, $\mfn\geq 2$, $q\geq n-1$, and a given set of complex-valued functions $\alpha=\{ \alpha_{k_1,k_2}\in\mathcal C^{\max(n,q-1)}(\Omega),0\leq k_1+k_2\leq \mfn \}$, the corresponding partial differential operator $\Lal$ and set of GPWs $\VGPW$. 
There exists a lower triangular matrix $L_n^R$, whose diagonal coefficients are equal all non-zero and whose other non-zero coefficients depend only on derivatives of the PDE coefficients $\alpha$ evaluated at $(x_0,y_0)$, such that
$$
\mathsf M_n = \mathsf L_n^R \cdot \mathsf M_n^R.
$$
As a consequence $rk(\mathsf M_n)=rk(\mathsf M_n^R)$ independently of the number $p$ of GPWs in $\VGPW$, and both $\| \mathsf L_n^R\|$ and $\|(\mathsf L_n^R)^{-1} \|$ are bounded by a constant depending only on the PDE coefficients $\alpha$.
\end{prop}
\begin{rmk}
If $n=1$, then the various matrices $\mathsf M$ belong to $\mathbb C^{3\times 3}$, and we have $\mathsf M_n = \mathsf M_n^{Tr}$ independently of the value of $q$.
\end{rmk}

\begin{proof}
Let's first relate $\mathsf M_n$ to $\mathsf M_n^{Tr}$.
The polynomials $R_{i,j}\in\mathbb C[X,Y]$ obtained in Proposition \ref{prop:derphi} have degree $\textrm d R_{i,j}\leq i+j-1$ and satisfy
\begin{equation}\label{eq:derR}
\forall (i,j)\in\mathbb N^2,i+j\leq q+1, \forall \varphi_l\in\VGPW, \dx^i\dy^j \varphi_l(x_0,y_0) = 
        (\luoz^l)^i(\luoz^l)^j
+R_{i,j} (\luoz^l  ,\luoz^l).
\end{equation} %Since $\mathsf M_n$ contains derivatives of the GPWs of order up to $n$
In order to apply this to all entries in the matrix $\mathsf M_n$, it is sufficient for $q$ to satisfy $n\leq q+1$, which explains the assumption on $q$. Therefore each  entry $(i,j)$ of the matrix $\mathsf M_n$ can be written as the sum of the $(i,j)$ entry of $\mathsf M_n^{Tr}$ and a linear combination of entries $(k,j)$ of $\mathsf M_n^{Tr}$ for $k<i$. In other words, the existence of a lower triangular matrix $\mathsf L_n^T$, whose diagonal coefficients are 1 and whose other non-zero coefficients depend only on the derivatives of the coefficients $\alpha$ evaluated at  $(x_0,y_0)$, such that $\mathsf M_n = \mathsf L_n^T\cdot\mathsf M_n^{Tr}$ is guaranteed since the coefficients of $R_{i,j}$ are independent of $l$ and any monomial in $R_{i,j}(\luoz,\luoz)$ has a degree lower than $i+j$. 

As a consequence, the existence of $\mathsf L_n^R$ is guaranteed by Lemma \ref{lmm:DRT} since $\mathsf L_n^R:=\mathsf L_n^T\cdot\mathsf D_n^{RT}$ satisfies the desired properties. 
\end{proof}

Everything is now in place to state and finally prove the necessary and sufficient condition on the number $p$ of GPWs for the space $\VGPW$ to satisfy the interpolation property \eqref{IntPb}. We here turn to the specific case of second order operators.
%\begin{thm}
%Assume $\Omega\subset\mathbb R^2$, and consider a given $(\mfn,n,p)\in(\mathbb N^*)^3$ and a given set of complex-valued functions $\alpha=\{ \alpha_{k_1,k_2}\in\mathcal C^{n}(\Omega),0\leq k_1+k_2\leq \mfn \}$, as well as a given point $(x_0,y_0)\in\Omega$ together with the linear operator $\mathcal L_{\mfn,\alpha}$ and the corresponding set of $p$ GPWs $\VGPW$. The space $\mathbb V_h^{G}:=span\VGPW$ satisfies the property
%\begin{equation}
%\begin{array}{l}
%\forall u\in \mathcal C^{\mfn+n}(\Omega)
%\text{ satisfying }\mathcal L_{\mfn,\alpha}u=0,
%\exists u_a\in\mathbb V_h^{G}, \exists \text{ a constant } C(\Omega,n) \text{ s. t. }\\
%\forall (x,y)\in\Omega,
%|u(x,y)-u_a(x_0,y_0)|\leq C(\Omega,n) \| (x,y)-(x_0,y_0) \|^{n+1},
%\end{array}
%\end{equation}
%if and only if $p\geq 2n+1$.
%\end{thm}
\begin{thm}\label{thethm}
\defParThm
The space $\mathbb V_h^{G}:=span\VGPW$ satisfies the property
\begin{equation}\label{eq:approx}
\begin{array}{l}
\forall u\in \mathcal C^{n+2}(\Omega)
\text{ satisfying }\mathcal L_{2,\alpha}u=0,
\exists u_a\in\mathbb V_h^{G}, \exists \text{ a constant } C(\Omega,n) \text{ s. t. }\\
\forall (x,y)\in\Omega,
|u(x,y)-u_a(x_0,y_0)|\leq C(\Omega,n) \| (x,y)-(x_0,y_0) \|^{n+1},
\end{array}
\end{equation}
if and only if $p\geq 2n+1$.
\end{thm}
\begin{proof}
Combining \eqref{eq:2n+1} with Proposition \ref{prop:LRn}, if $q\geq n-1$, we see immediately that $rk(\mathsf M_n) = 2n+1$ if and only if $p\geq 2n+1$.

It is then sufficient to prove that the space
$$
\begin{array}{l}
\displaystyle
\mathbb F_\alpha :=\left\{\mathsf F\in\mathbb C^{(n+1)(n+2)/2},\exists v\in\mathcal C^n(\Omega)\text{ satisfying }\mathcal L_{2,\alpha}v=0 
\right.
\\\displaystyle
\phantom{
\mathfrak K :=
\mathfrak K :=
}
\left.
\text{ s.t. } \mathsf F_{\frac{(k_1+k_2)(k_1+k_2+1)}{2}+k_2+1} = \dx^{k_1} \dy^{k_2} v (x_0,y_0) /(k_1!k_2!)
\right\}.
\end{array}
$$
satisfies $\mathbb F_\alpha \subset \mathcal R(\mathsf M_n)$, the range of $\mathsf M_n$. To this end, we now define the space
$$
\begin{array}{l}
\displaystyle
\mathfrak K :=
\left\{
\mathsf K\in\mathbb C^{(n+1)(n+2)/2},\exists f\in\mathcal C^n(\Omega)\text{ satisfying }\mathcal L_{2,\alpha}f(x,y)=O(\|(x,y)-(x_0,y_0)\|^{n-1})
\right.
\\\displaystyle
\phantom{
\mathfrak K :=
\mathfrak K :=
}
\left.
 \text{ s.t. } \mathsf K_{\frac{(k_1+k_2)(k_1+k_2+1)}{2}+k_2+1} = \dx^{k_1} \dy^{k_2} f (x_0,y_0) /(k_1!k_2!)
\right\}.
\end{array}
$$
We can now see that
\begin{itemize}
\item $\mathcal R(\mathsf M_n)\subset \mathfrak K$ independently of the value of $p$,
since by construction of GPWs, as long as $q\geq n-1$, each column of $\mathsf M_n$ belongs to $\mathfrak K$;
\item $\mathbb F_\alpha\subset \mathfrak K$, by definition of $\mathbb F_\alpha$;
\item $\dim \mathfrak K=2n+1$, since - from the condition involving the Taylor expansion coefficients of $\mathcal L_{2,\alpha}f$ of order up to $n-2$ at $x_0,y_0)$ set to zero - $\mathfrak K\subset \mathbb C^{(n+1)(n+2)/2}$ is the kernel of a matrix $\mathsf A\in\mathbb C^{n(n-1)/2\times (n+1)(n+2)/2}$ with
$$
\begin{array}{l}
\forall(i,j)\in\mathbb N^2,i+j<n-1,
 \mathsf A_{\frac{(i+j)(i+j+1)}{2}+j+1,\frac{(i+j+2)(i+j+3)}{2}+j+1}=\alpha_{2,0}(x_0,y_0)\neq 0  \text{ from Hypothesis \ref{hyp},}\\
\forall(\tilde i,\tilde j)\in\mathbb N^2,\tilde i+\tilde j<n-1, 
\text{ if } \tilde i+\tilde j >i+j 
\text{ or } 
\text{ if } \tilde i+\tilde j =i+j, \tilde j>j\\
\phantom{
\forall(i,j)\in\mathbb N^2,i+j<n-1,
} 
\mathsf A_{\frac{(i+j)(i+j+1)}{2}+j+1,\frac{(\tilde i+\tilde j+2)(\tilde i+\tilde j+3)}{2}+\tilde j+1}=0, 
\end{array}
$$ 
 so that $\mathsf A$ is of maximal rank while its kernel has dimension $\frac{(n+1)(n+2}{2}-\frac{n(n-1)}{2}=2n+1$.
\end{itemize}
Therefore, if $p\geq 2n+1$, we obtain that $\mathcal R(\mathsf M_n)=\mathfrak K$ and as a consequence $\mathbb F_\alpha\in\mathcal R(\mathsf M_n)$ as expected. This concludes the proof.
\end{proof}
The necessary and sufficient condition on the number $p$ of GPWs for the space $\VGPW$ to satisfy the interpolation property \eqref{IntPb} when $\mfn>2$ are still unknown.

\begin{rmk}
As in \cite{LMinterp}, the theorem holds for the Helmholtz equation with sign changing.
%Compare to Remark \ref{rmk:chgtp}
\end{rmk}

%%%%%%%%%%%%%%%%%%%%%%%%%%%%%%%%%%
%\section{Second order operators}
\section{Numerical experiments}\label{sec:NR}
In \cite{LMinterp}, GPWs where constructed and studied for the Helmholtz equation \eqref{Helm} with a variable and sign-changing coefficient $\beta$. The numerical experiments presented there were restricted to the Helmholtz equation at one point $(x_0,y_0)\in\mathbb R^2$, but considered a propagative case i.e. $\beta(x_0,y_0)<0$, an evanescent case i.e. $\beta(x_0,y_0)>0$, a cut-off case i.e. $\beta(x_0,y_0)=0$. They also considered a case not covered by the convergence theorem, but important for future applications: considering GPWs centered at points $(x_0,y_0)$ at a distance $h$ from the cut-off.

Here, we are interested in illustrating the results presented in Theorem \ref{thethm}. Since the well known case of classical PW for the constant-coefficient Helmholtz equation is included by the hypotheses of the theorem, we cannot expect any improvement on the required number of basis functions $p$. However, we are interested in exploring the impact of the order of approximation $q$ on the convergence of \eqref{eq:approx}, in particular for anisotropic problems.

%%%%%%%%%%%%%%%%%%%%%%%%%%%%%%%%%%%%
\subsection{Test cases}
We propose here four different test cases. Each test case consists of a partial differential operator of second order $\mathcal L$, an exact solution $u$ of the equation $\mathcal L u = 0$, as well as a computational domain $\Omega\subset\mathbb R^2$, such that Hypotheses \ref{hyp} and \ref{hyp2} hold at all $(x_0,y_0)\in\Omega$. The characteristics of the partial differential operators that we consider here are:
\begin{itemize}
\item polynomial coefficients $\alpha$,
\item non-polynomial coefficients $\alpha$,
\item anisotropy in the first order terms as $\overrightarrow a (x,y) \cdot\nabla$ for a vector-valued function $\overrightarrow a$;
\item anisotropy in the second order terms as $\nabla\cdot (A(x,y)\nabla)$ for a matrix-valued function $A$.
\end{itemize}

We consider one partial differential operator is isotropic with polynomial coefficients:
$$\left\{
\begin{array}{l}
\mathcal L_{Ad}:=-\Delta+2(x+y),\\
u_{Ad}:(x,y)\mapsto Ai(x+y), \\
\Omega_{Ad}:=(-2,2)^2.
\end{array}\right.
$$
We have $\mathcal L_{Ad}u_{Ad}=0$ on $\mathbb R^2$, all the coefficients of $\mathcal L_{Ad}$ belong to $\mathcal C^\infty(\mathbb R^2)$ and the coefficients $\{\alpha_{k,2-k}^{Ad};k=0,1,2\}$ satisfy 
$$\sum_{k=0}^2 \alpha_{k,2-k}^{Ad}(x_0,y_0)X^kY^{2-k} = X^2+Y^2\quad \forall (x_0,y_0)\in\mathbb R^2,$$
so $\mathcal L_{Ad}$ satisfies Hypotheses \ref{hyp} and \ref{hyp2} on $\mathbb R^2$. Note that the sign of the coefficient $\alpha_{0,0}^{Ad}(x,y) = 2(x+y)$ changes in the computational domain along the curve $x+y=0$.
%The function $u_{Ad}:(x,y)\mapsto Ai(x+y) $ is an exact solution of the equation $\mathcal L_{Ad}u=0$ on $\mathbb R^2$ where
%$$
%\mathcal L_{Ad}:=-\Delta+2(x+y)
%$$
%is an isotropic operator with polynomial coefficients. All the coefficients of $\mathcal L_{Ad}$ belong to $\mathcal C^\infty(\mathbb R^2)$ and the coefficients $\{\alpha_{k,2-k};k=0,1,2\}$ satisfy 
%$$\sum_{k=0}^2 \alpha_{k,2-k}(x_0,y_0)X^kY^{2-k} = X^2+Y^2\quad \forall (x_0,y_0)\in\mathbb R^2,$$
%so $\mathcal L_{Ad}$ satisfies Hypotheses \ref{hyp} and \ref{hyp2} on $\mathbb R^2$. We will consider the domain $\Omega_{Ad}:=(-2,2)^2$. 

We consider a partial differential operator with non-polynomial coefficients of the terms of order $1$ and $0$, and anisotropy in the first order term:
$$\left\{
\begin{array}{l}
\mathcal L_{Jc}:=
\nabla \cdot(x^2\nabla )+\begin{pmatrix}-x\\\cos y\end{pmatrix}\cdot\nabla +(\nu^2-2x^2-\sin y ),\\
u_{Jc}:(x,y)\mapsto J_1(x)\cos y,\\
\Omega_{Jc}:=(1,4)\times(0,2\pi).
\end{array}\right.
$$
We have $\mathcal L_{Jc}u_{Jc}=0$ on $(0,\infty)\times\mathbb R$, all the coefficients of $\mathcal L_{Jc}$ belong to $\mathcal C^\infty(\mathbb R^+\times\mathbb R)$ and the coefficients $\{\alpha_{k,2-k}^{Jc};k=0,1,2\}$ satisfy 
$$\sum_{k=0}^2 \alpha_{k,2-k}^{Jc}(x_0,y_0)X^kY^{2-k} =x_0^2( X^2+Y^2) \quad \forall (x_0,y_0)\in\mathbb R^2,$$
so $\mathcal L_{Jc}$ satisfies Hypotheses \ref{hyp} and \ref{hyp2} as long as  $x>0$.
%The function $u_{Jc}:(x,y)\mapsto J_1(x)\cos y $ is an exact solution of the equation $\mathcal L_{Jc}u=0$ on $\mathbb R^+\times\mathbb R$ where
%$$
%\mathcal L_{Jc}:=
%\nabla \cdot(x^2\nabla )
%+\begin{pmatrix}
%-x\\\cos y
%\end{pmatrix}
%\cdot\nabla 
%+(\nu^2-2x^2-\sin y )
%$$
%contains anisotropy in the first order term and has non-polynomial coefficients of the terms of order $1$ and $0$. All the coefficients of $\mathcal L_{Jc}$ belong to $\mathcal C^\infty(\mathbb R^+\times\mathbb R)$ and the coefficients $\{\alpha_{k,2-k};k=0,1,2\}$ satisfy 
%$$\sum_{k=0}^2 \alpha_{k,2-k}(x_0,y_0)X^kY^{2-k} =x_0^2( X^2+Y^2) \quad \forall (x_0,y_0)\in\mathbb R^+\times\mathbb R,$$
%so $\mathcal L_{Jc}$ satisfies Hypotheses \ref{hyp} and \ref{hyp2} as long as  $x>0$. We will consider the domain $\Omega_{Jc}:=(1,4)\times(0,2\pi)$. 

We consider a partial differential operator with  polynomial coefficients and  anisotropy in the first and second order terms:
$$\left\{
\begin{array}{l}
\mathcal L_{JJ}:=
\nabla \cdot\begin{pmatrix}x^2&0\\0&y^2\end{pmatrix}\nabla -\begin{pmatrix}x\\y\end{pmatrix}\cdot\nabla+(x^2+y^2-1) ,\\
u_{JJ}:(x,y)\mapsto J_0(x)J_1( y),\\
\Omega_{JJ}:=(1,3)\times(1,3).
\end{array}\right.
$$
We have $\mathcal L_{JJ}u_{JJ}=0$ on $(\mathbb R^+)^2$, all the coefficients of $\mathcal L_{JJ}$ belong to $\mathcal C^\infty((\mathbb R^+)^2)$ and the coefficients $\{\alpha_{k,2-k}^{JJ};k=0,1,2\}$ satisfy 
$$\sum_{k=0}^2 \alpha_{k,2-k}^{JJ}(x_0,y_0)X^kY^{2-k} =x_0^2X^2+y_0^2Y^2 \quad \forall (x_0,y_0)\in\mathbb R^2,$$
so $\mathcal L_{JJ}$ satisfies Hypotheses \ref{hyp} and \ref{hyp2} as long as $xy\neq 0$.
%The function $u_{JJ}:(x,y)\mapsto J_0(x)J_1( y) $ is an exact solution of the equation $\mathcal L_{JJ}u=0$ on $(\mathbb R^+)^2$ where
%$$
%\mathcal L_{JJ}:=
%\nabla \cdot
%\begin{pmatrix}
%x^2&0\\0&y^2
%\end{pmatrix}
%\nabla 
%-\begin{pmatrix}
%x\\y
%\end{pmatrix}
%\cdot\nabla
%+(x^2+y^2-1) 
%$$
%contains anisotropy in the first and second order terms and has polynomial coefficients. All the coefficients of $\mathcal L_{JJ}$ belong to $\mathcal C^\infty((\mathbb R^+)^2)$ and the coefficients $\{\alpha_{k,2-k};k=0,1,2\}$ satisfy 
%$$\sum_{k=0}^2 \alpha_{k,2-k}(x_0,y_0)X^kY^{2-k} =x_0^2X^2+y_0^2Y^2 \quad \forall (x_0,y_0)\in(\mathbb R^+)^2,$$
%so $\mathcal L_{JJ}$ satisfies Hypotheses \ref{hyp} and \ref{hyp2} as long as $xy\neq 0$. We will consider the domain $\Omega_{JJ}:=(1,3)\times(2,5)$. 
%%Polynomial anisotropy The function $u(x,y):= J_\nu (x) J_\mu( y)$ satisfies
%%$$
%%\nabla \cdot
%%\begin{pmatrix}
%%x^2&0\\0&y^2
%%\end{pmatrix}
%%\nabla u
%%-\begin{pmatrix}
%%x\\y
%%\end{pmatrix}
%%\cdot\nabla u
%%+(x^2+y^2-\nu^2-\mu^2) u =0
%%$$

Finally we consider a partial differential operator with non-polynomial coefficients and  anisotropy in the second order term:
$$\left\{
\begin{array}{l}
\mathcal L_{cs}:=
\nabla \cdot\begin{pmatrix}1&0.1\cos x\sin y\\0.1\cos x\sin y&-2\end{pmatrix}\nabla  -0.1\begin{pmatrix}\cos x(\phantom +\cos y)\\sin y(-\sin x)\end{pmatrix}\cdot\nabla  +(0.2\sin x \cos y-1) ,\\
{\color{white}\mathcal L_{cs}:}=
\dx^2+0.2\cos x\sin y \ \dx\dy-2\dy^2 +(0.2\sin x \cos y-1) ,\\
u_{cs}:(x,y)\mapsto \cos x\sin y ,\\
\Omega_{cs}:=(-1,1)^2.
\end{array}\right.
$$
We have $\mathcal L_{cs}u_{cs}=0$ on $\mathbb R^2$, all the coefficients of $\mathcal L_{cs}$ belong to $\mathcal C^\infty(\mathbb R^2)$ and the coefficients $\{\alpha_{k,2-k}^{cs};k=0,1,2\}$ satisfy 
$$\sum_{k=0}^2 \alpha_{k,2-k}^{cs}(x_0,y_0)X^kY^{2-k} 
=\left(1-\frac{(0.1)^2}{2}\cos^2 x_0\sin^2 y_0\right)X^2
 -2\left(Y-\frac{0.1}2 \cos x_0\sin y_0 X\right)^2
%=(X+0.1\cos x_0\sin y_0 Y)^2-(2+0.01\cos^2 x_0\sin^2 y_0)Y^2
 \quad \forall (x_0,y_0)\in\mathbb R^2,$$
so $\mathcal L_{cs}$ satisfies Hypotheses \ref{hyp} and \ref{hyp2} on $\mathbb R^2$.
%The function $u_{cs}:(x,y)\mapsto \cos x\sin y $ is an exact solution of the equation $\mathcal L_{cs}u_{cs}=0$ on $\mathbb R^2$ where
%$$
%\mathcal L_{cs}:=
%\nabla \cdot
%\begin{pmatrix}
%1&0.1\cos x\sin y\\0.1\cos x\sin y&-2
%\end{pmatrix}
%\nabla  
%-\begin{pmatrix}
%\cos x
%(
%\phantom +%y^5+
%\cos y)\\
%sin y
%(
%%x^3
%-\sin x)
%\end{pmatrix}
%\cdot\nabla  
%+(0.2\sin x \cos y-1
%% -x^3\cos y +y^5\sin x
%) 
%$$
%contains anisotropy in the first and second order terms and has non-polynomial coefficients. All the coefficients of $\mathcal L_{cs}$ belong to $\mathcal C^\infty(\mathbb R^2)$ and the coefficients $\{\alpha_{k,2-k};k=0,1,2\}$ satisfy 
%$$\sum_{k=0}^2 \alpha_{k,2-k}(x_0,y_0)X^kY^{2-k} =(X+\cos x_0\sin y_0 Y)^2-(2+\cos^2 x_0\sin^2 y_0)Y^2 \quad \forall (x_0,y_0)\in\mathbb R^2,$$
%$$\sum_{k=0}^2 \alpha_{k,2-k}(x_0,y_0)X^kY^{2-k} =X^2-(2+\cos^2 x_0\sin^2 y_0)Y^2 \quad \forall (x_0,y_0)\in\mathbb R^2,$$
%so $\mathcal L_{cs}$ satisfies Hypotheses \ref{hyp} and \ref{hyp2} on $\mathbb R^2$. We will consider the domain $\Omega_{cs}:=(-1,1)^2$. 

%%%%%%%%%%%%%%%%%%%%%%%%%%%%%%%%%%%%
\subsection{Implementation of the construction algorithm}
For a linear second order operator 
$$
\mathfrak L_{2,\alpha} = 
\alpha_{2,0} \dx^2  + \alpha_{1,1} \dx\dy +\alpha_{0,2} \dy^2 
+\alpha_{1,0} \dx     + \alpha_{0,1}\dy
+\alpha_{0,0} 
$$ 
 the associated operator $\mathfrak L_{2,\alpha}^A$ is defined by
\begin{equation*}
\mathfrak L_{2,\alpha}^A P=
 \underbrace{\alpha_{2,0} \dx^2 P  + \alpha_{1,1} \dx \dy P + \alpha_{0,2} \dy^2 P}_{T_1}
+\underbrace{\alpha_{2,0}(\dx P)^2 + \alpha_{1,1}\dx P\dy P + \alpha_{0,2}(\dy P)^2}_{T_2}
+\underbrace{\alpha_{1,0} \dx P    + \alpha_{0,1}\dy P}_{T_3}.
\end{equation*}
The implementation of Algorithm \ref{algo} simply requires, at each level $\mfl$, the evaluation of $\{N_{I,\mfl - I},0\leq I\leq \mfl\}$ %from \eqref{eq:D-Rij} 
to apply formula \eqref{eq:ls}. At each level $\mfl$ the coefficeints $\{\mu_{ij},(i,j)\in\mathbb N^2, i+j\leq q+1\}$ of $Q_\mfl:=\sum_{0\leq i+j\leq \mfn+\mfl-1}\lambda_{i,j}(x-x_0)^i(y-y_0)^j$ are computed as
$$
\mu_{i,j}:=\left\{\begin{array}{l}
\lambda_{i,j} \text{ if } i+j\leq\mfl+1\\
0\text{ otherwise},
\end{array}\right.
$$
and for $0\leq I\leq \mfl$ the different contributions to $N_{I,\mfl - I}$ can be described as follows:
\begin{itemize}
\item the linear contributions from first order terms $T_3$
$$
-\sum_{i=0}^I\sum_{j=0}^{\mfl-I}
\left( 
\Dop{I-i}{\mfl-I-j}\alpha_{1,0}(x_0,y_0)
(i+1)\mu_{i+1,j}
+
\Dop{I-i}{\mfl-I-j}\alpha_{0,1}(x_0,y_0)
(j+1)\mu_{i,j+1}
\right)
$$
\item the non-linear contributions from the terms $T_2$
$$
\begin{array}{rl}
\displaystyle-\sum_{i_1=0}^I\sum_{j_1=0}^{\mfl-I}\sum_{i_2=0}^{i_1}\sum_{j_2=0}^{j_1}
&\displaystyle
\left( 
\Dop{I-i_1}{\mfl-I-j_1}\alpha_{2,0}(x_0,y_0)
(i_1-i_2+1)(i_2+1)\mu_{i_1-i_2+1,j_1-j_2}\mu_{i_2+1,j_2}
\right.\\
&\displaystyle\left.+
\Dop{I-i_1}{\mfl-I-j_1}\alpha_{1,1}(x_0,y_0)
(i_1-i_2+1)(j_2+1)\mu_{i_1-i_2+1,j_1-j_2}\mu_{i_2,j_2+1}
\right.\\
&\displaystyle\left.+
\Dop{I-i_1}{\mfl-I-j_1}\alpha_{0,2}(x_0,y_0)
(j_1-j_2+1)(j_2+1)\mu_{i_1-i_2,j_1-j_2+1}\mu_{i_2,j_2+1}
\right),
\end{array}
$$
\item the linear contributions from the second order terms $T_1$
$$
\begin{array}{rl}
\displaystyle-\sum_{i=0}^I\sum_{j=0}^{\mfl-I}
&\displaystyle
\left( 
\Dop{I-i}{\mfl-I-j}\alpha_{2,0}(x_0,y_0)
(i+2)(i+1)\mu_{i+2,j}
\right.\\
&\displaystyle\left.+
\Dop{I-i}{\mfl-I-j}\alpha_{1,1}(x_0,y_0)
(j+1)(i+1)\mu_{i+1,j+1}
\right.\\
&\displaystyle\left.+
\Dop{I-i}{\mfl-I-j}\alpha_{0,2}(x_0,y_0)
(j+2)(j+1)\mu_{i,j+2}
\right),
\end{array}
$$
\item the contribution from the zeroth order term $\alpha_{0,0}$
$$
-\Dop{I}{\mfl-I}\alpha_{0,0}(x_0,y_0).
$$
\end{itemize}

Moreover, all experiments are conducted with the following choice of angles $\theta_l$ and $\kappa$ parameters to build the GPW space $\VGPW$:
$$\left\{
\begin{array}{l}
\theta_l := \frac{\pi}{6}+\frac{2(l-1)\pi}{p},\ \forall l\in\mathbb N, 1\leq l\leq p,
\\
\kappa = \sqrt{-\alpha_{0,0}(x_0,y_0)}
\end{array}\right.
$$

%%%%%%%%%%%%%%%%%%%%%%%%%%%%%%%%%%
\subsection{Numerical results}
The $h$-convergence results presented in Theorem \ref{thethm} are stated as local properties at a given point. In order to illustrate them, for each test case, we consider the following procedure.
\begin{itemize}
\item At each of 50 random points $(x_0,y_0)$ in the computational domain  $\Omega$
\begin{enumerate}
%\item At a random point in the computational domain 
\item Construct the set of GPWs from Algorithm \ref{algo} with the normalization proposed in section \ref{sec:norm}.
\item Compute $u_a$ the linear combination of GPWs studied in the theorem's proof, matching its Taylor expansion to that of the exact solution.
\end{enumerate}
\item Estimate as a function of $h$ the maximum $L^\infty$ error on a disk of radius $h$ centered at the random point: $\max_{(x_0,y_0)\in\Omega}\|u-u_a\|_{L^\infty(\{(x,y)\in\mathbb R^2, |(x,y)-(x_0,y_0)|<h \})}$.
\end{itemize}
%For each test case we repeat this procedure 50 times and save the maximum error among the 50 corresponding errors.  
We always consider a space $\VGPW$ of $p=2n+1$ GPWs. According to the theorem, we expect to observe convergence of order $n+1$ if the approximation parameter $q$ in the construction of the basis functions is at least equal to $n-1$. For each of the four test cases proposed, we present: on the one hand results for $n$ from $1$ to $5$ with $q=\max(1,n-1)$ (Left panel); on the other hand results for $q$ from $1$ to $4$ with $n=4$ (Right panel). Hence with the first choice of parameters the theorem predicts convergence of order $n+1$, while with the second choice the theorem does not cover these cases.

The results are presented in Figure \ref{AdNR} for the approximation of $u_{Ad}$, Figure \ref{JcNR} for the approximation of $u_{Jc}$, Figure \ref{JJNR} for the approximation of $u_{JJ}$, and Figure \ref{csNR} for the approximation of $u_{cs}$. 
We observe on Figures \ref{AdNR} and \ref{csNR} the effect of the large condition number of the matrix $\mathsf M_n$: even though the expected orders of convergence are observed for large values of $h$, the error stagnates at a threshold for smaller values of $h$. 
%about the constant in the theorem
We also observe, on the left panels of Figures  \ref{AdNR}, \ref{JcNR}, \ref{JJNR} and \ref{csNR}, that  for all of our test cases  the constant $C(\Omega,n)$ from \eqref{eq:approx} in Theorem \ref{thethm} does not seem to depend on $n$, even though the Theorem predicts that it does.
\begin{figure}
\begin{tikzpicture}
\begin{loglogaxis}[height=7cm,xlabel=$h$, ylabel=max error on disks of radius h,
legend pos=north west,
xmin=10^(-8),xmax=20,ymin=10^(-16),ymax=10,
xtick={1,0.01,.0001,.000001,.00000001},
very thick,cycle list name=color,grid=major]
\addplot table [x=h, y=errn1]{./GPWBAE2D2OthmCaseAixpy50pts.dat};
\addlegendentry{$n=1$ \& $q=1$}
\addplot table [x=h, y=errn2]{./GPWBAE2D2OthmCaseAixpy50pts.dat};
\addlegendentry{$n=2$ \& $q=1$}
\addplot table [x=h, y=errn3]{./GPWBAE2D2OthmCaseAixpy50pts.dat};
\addlegendentry{$n=3$ \& $q=2$}
\addplot table [x=h, y=errn4]{./GPWBAE2D2OthmCaseAixpy50pts.dat};
\addlegendentry{$n=4$ \& $q=3$}
\addplot table [x=h, y=errn5]{./GPWBAE2D2OthmCaseAixpy50pts.dat};
\addlegendentry{$n=5$ \& $q=4$}
\addplot[dotted] coordinates {(10^-6, 10^-11) (10^-1, 10^-1)};
\addlegendentry{order 2}
\addplot[dashed] coordinates {(3*10^-2, .5*10^-12) (3*10^-0, .5*10^-0)};
\addlegendentry{order 6}
\end{loglogaxis}
\end{tikzpicture}
\begin{tikzpicture}
\begin{loglogaxis}[height=7cm,xlabel=$h$, ylabel=max error on disks of radius h,
legend pos=north west,
xmin=10^(-8),xmax=20,ymin=10^(-16),ymax=10,
xtick={1,0.01,.0001,.000001,.00000001},
very thick,cycle list name=color,grid=major]
\addplot table [x=h, y=errn1]{./GPWBAE2D2On4CaseAixpy50pts.dat};
\addlegendentry{$n=4$ \& $q=1$}
\addplot table [x=h, y=errn2]{./GPWBAE2D2On4CaseAixpy50pts.dat};
\addlegendentry{$n=4$ \& $q=2$}
\addplot table [x=h, y=errn3]{./GPWBAE2D2On4CaseAixpy50pts.dat};
\addlegendentry{$n=4$ \& $q=3$}
\addplot table [x=h, y=errn4]{./GPWBAE2D2On4CaseAixpy50pts.dat};
\addlegendentry{$n=4$ \& $q=4$}
%\addplot table [x=h, y=errn5]{./GPWBAE2D2On4CaseAixpy50pts.dat};
%\addlegendentry{$n=5$}
\addplot[dotted] coordinates {(10^-4, 10^-11) (10^-1, 10^-2)};
\addlegendentry{order 3}
\addplot[dashed] coordinates {(3*10^-2, 10^-12) (3*10^-0, 10^-2)};
\addlegendentry{order 5}
\end{loglogaxis}
\end{tikzpicture}
%$$
%\begin{array}{c||c|c|c|c|c|}
%q\backslash n & 1&2&3&4&5\\\hline\hline
%1&\bf 2 & \bf 3 &     3 &     3 &     3 \\\hline
%2&\bf 2 & \bf 3 & \bf 4 &     4 &     4 \\\hline
%3&\bf 2 & \bf 3 & \bf 4 & \bf 5 &     5 \\\hline
%4&\bf 2 & \bf 3 & \bf 4 & \textcolor{red}{\bf 5} & \bf 6 \\\hline
%\end{array}$$
\caption{\capfig{{Ad}}}\label{AdNR}
\end{figure}
\begin{figure}
\begin{tikzpicture}
\begin{loglogaxis}[height=7cm,xlabel=$h$, ylabel=max error on disks of radius h,
legend pos=north west,
xmin=10^(-8),xmax=20,ymin=10^(-16),ymax=10,
xtick={1,0.01,.0001,.000001,.00000001},
very thick,cycle list name=color,grid=major]
\addplot table [x=h, y=errn1]{./GPWBAE2D2OthmCaseJ1xcosy50pts.dat};
\addlegendentry{$n=1$ \& $q=1$}
\addplot table [x=h, y=errn2]{./GPWBAE2D2OthmCaseJ1xcosy50pts.dat};
\addlegendentry{$n=2$ \& $q=1$}
\addplot table [x=h, y=errn3]{./GPWBAE2D2OthmCaseJ1xcosy50pts.dat};
\addlegendentry{$n=3$ \& $q=2$}
\addplot table [x=h, y=errn4]{./GPWBAE2D2OthmCaseJ1xcosy50pts.dat};
\addlegendentry{$n=4$ \& $q=3$}
\addplot table [x=h, y=errn5]{./GPWBAE2D2OthmCaseJ1xcosy50pts.dat};
\addlegendentry{$n=5$ \& $q=4$}
\addplot[dotted] coordinates {(10^-6, 10^-11) (10^-3, 10^-5)};
\addlegendentry{order 2}
\addplot[dashed] coordinates {(5*10^-2, .5*10^-12) (5*10^-1, .5*10^-6)};
\addlegendentry{order 6}
\end{loglogaxis}
\end{tikzpicture}
\begin{tikzpicture}
\begin{loglogaxis}[height=7cm,xlabel=$h$, ylabel=max error on disks of radius h,
legend pos=north west,
xmin=10^(-8),xmax=20,ymin=10^(-16),ymax=10,
xtick={1,0.01,.0001,.000001,.00000001},
very thick,cycle list name=color,grid=major]
\addplot table [x=h, y=errn1]{./GPWBAE2D2On4CaseJ1xcosy50pts.dat};
\addlegendentry{$n=4$ \& $q=1$}
\addplot table [x=h, y=errn2]{./GPWBAE2D2On4CaseJ1xcosy50pts.dat};
\addlegendentry{$n=4$ \& $q=2$}
\addplot table [x=h, y=errn3]{./GPWBAE2D2On4CaseJ1xcosy50pts.dat};
\addlegendentry{$n=4$ \& $q=3$}
\addplot table [x=h, y=errn4]{./GPWBAE2D2On4CaseJ1xcosy50pts.dat};
\addlegendentry{$n=4$ \& $q=4$}
%\addplot table [x=h, y=errn5]{./GPWBAE2D2On4CaseJ1xcosy50pts.dat};
%\addlegendentry{$n=5$}
\addplot[dotted] coordinates {(10^-4, 1*10^-11) (10^-1, 1*10^-2)};
\addlegendentry{order 3}
\addplot[dashed] coordinates {(2*10^-2, .5*10^-14) (2*10^-0, .5*10^-2)};
\addlegendentry{order 6}
\end{loglogaxis}
\end{tikzpicture}
%$$
%\begin{array}{c||c|c|c|c|c|}
%q\backslash n & 1&2&3&4&5\\\hline\hline
%1&\bf 2 & \bf 3 &     3 &     3 &     3 \\\hline
%2&\bf 2 & \bf 3 & \bf 4 &     4 &     4 \\\hline
%3&\bf 2 & \bf 3 & \bf 4 & \bf 5 &     5 \\\hline
%4&\bf 2 & \bf 3 & \bf 4 & \textcolor{red}{\bf 6} & \bf 6 \\\hline
%\end{array}$$
\caption{%$Jc$ test case.
\capfig{{Jc}} }\label{JcNR}
\end{figure}
\begin{figure}
\begin{tikzpicture}
\begin{loglogaxis}[height=7cm,xlabel=$h$, ylabel=max error on disks of radius h,
legend pos=north west,
xmin=10^(-8),xmax=20,ymin=10^(-16),ymax=10,
xtick={1,0.01,.0001,.000001,.00000001},
very thick,cycle list name=color,grid=major]
\addplot table [x=h, y=errn1]{./GPWBAE2D2OthmCaseJ0xJ1y50pts.dat};
\addlegendentry{$n=1$ \& $q=1$}
\addplot table [x=h, y=errn2]{./GPWBAE2D2OthmCaseJ0xJ1y50pts.dat};
\addlegendentry{$n=2$ \& $q=1$}
\addplot table [x=h, y=errn3]{./GPWBAE2D2OthmCaseJ0xJ1y50pts.dat};
\addlegendentry{$n=3$ \& $q=2$}
\addplot table [x=h, y=errn4]{./GPWBAE2D2OthmCaseJ0xJ1y50pts.dat};
\addlegendentry{$n=4$ \& $q=3$}
\addplot table [x=h, y=errn5]{./GPWBAE2D2OthmCaseJ0xJ1y50pts.dat};
\addlegendentry{$n=5$ \& $q=4$}
\addplot[dotted] coordinates {(10^-6, 10^-10) (10^-3, 10^-4)};
\addlegendentry{order 2}
\addplot[dashed] coordinates {(2*10^-2, 10^-14) (2*10^-0, 10^-2)};
\addlegendentry{order 6}
\end{loglogaxis}
\end{tikzpicture}
\begin{tikzpicture}
\begin{loglogaxis}[height=7cm,xlabel=$h$, ylabel=max error on disks of radius h,
legend pos=north west,
xmin=10^(-8),xmax=20,ymin=10^(-16),ymax=10,
xtick={1,0.01,.0001,.000001,.00000001},
very thick,cycle list name=color,grid=major]
\addplot table [x=h, y=errn1]{./GPWBAE2D2On4CaseJ0xJ1y50pts.dat};
\addlegendentry{$n=4$ \& $q=1$}
\addplot table [x=h, y=errn2]{./GPWBAE2D2On4CaseJ0xJ1y50pts.dat};
\addlegendentry{$n=4$ \& $q=2$}
\addplot table [x=h, y=errn3]{./GPWBAE2D2On4CaseJ0xJ1y50pts.dat};
\addlegendentry{$n=4$ \& $q=3$}
\addplot table [x=h, y=errn4]{./GPWBAE2D2On4CaseJ0xJ1y50pts.dat};
\addlegendentry{$n=4$ \& $q=4$}
%\addplot table [x=h, y=errn5]{./GPWBAE2D2On4CaseJ0xJ1y50pts.dat};
%\addlegendentry{$n=5$}
\addplot[dotted] coordinates {(5*10^-5, 10^-12) (5*10^-3, 10^-6)};
\addlegendentry{order 3}
\addplot[dashed] coordinates {(10^-2, .5*10^-14) (10^-0, .5*10^-2)};
\addlegendentry{order 6}
\end{loglogaxis}
\end{tikzpicture}
%$$
%\begin{array}{c||c|c|c|c|c|}
%q\backslash n & 1&2&3&4&5\\\hline\hline
%1&\bf 2 & \bf 3 &     3 &     3 &     3 \\\hline
%2&\bf 2 & \bf 3 & \bf 4 &   > 4 &    >4 \\\hline
%3&\bf 2 & \bf 3 & \bf 4 & \bf >5 &    6 \\\hline
%4&\bf 2 & \bf 3 & \bf 4 & \textcolor{red}{\bf >5} & \bf 6 \\\hline
%\end{array}$$
\caption{%$JJ$ test case.
\capfig{{JJ}} }\label{JJNR}
\end{figure}
\begin{figure}
\begin{tikzpicture}
\begin{loglogaxis}[height=7cm,xlabel=$h$, ylabel=max error on disks of radius h,
legend pos=north west,
xmin=10^(-8),xmax=20,ymin=10^(-16),ymax=10,
xtick={1,0.01,.0001,.000001,.00000001},
very thick,cycle list name=color,grid=major]
\addplot table [x=h, y=errn1]{./GPWBAE2D2OthmCasecosxsiny50pts.dat};
\addlegendentry{$n=1$ \& $q=1$}
\addplot table [x=h, y=errn2]{./GPWBAE2D2OthmCasecosxsiny50pts.dat};
\addlegendentry{$n=2$ \& $q=1$}
\addplot table [x=h, y=errn3]{./GPWBAE2D2OthmCasecosxsiny50pts.dat};
\addlegendentry{$n=3$ \& $q=2$}
\addplot table [x=h, y=errn4]{./GPWBAE2D2OthmCasecosxsiny50pts.dat};
\addlegendentry{$n=4$ \& $q=3$}
\addplot table [x=h, y=errn5]{./GPWBAE2D2OthmCasecosxsiny50pts.dat};
\addlegendentry{$n=5$ \& $q=4$}
\addplot[dotted] coordinates {(2*10^-6, 10^-10) (2*10^-3, 10^-4)};
\addlegendentry{order 2}
\addplot[dashed] coordinates {(1*10^-2, 5*10^-14) (1*10^-0, 5*10^-2)};
\addlegendentry{order 6}
\end{loglogaxis}
\end{tikzpicture}
\begin{tikzpicture}
\begin{loglogaxis}[height=7cm,xlabel=$h$, ylabel=max error on disks of radius h,
legend pos=north west,
xmin=10^(-8),xmax=20,ymin=10^(-16),ymax=10,
xtick={1,0.01,.0001,.000001,.00000001},
very thick,cycle list name=color,grid=major]
\addplot table [x=h, y=errn1]{./GPWBAE2D2On4Casecosxsiny50pts.dat};
\addlegendentry{$n=4$ \& $q=1$}
\addplot table [x=h, y=errn2]{./GPWBAE2D2On4Casecosxsiny50pts.dat};
\addlegendentry{$n=4$ \& $q=2$}
\addplot table [x=h, y=errn3]{./GPWBAE2D2On4Casecosxsiny50pts.dat};
\addlegendentry{$n=4$ \& $q=3$}
\addplot table [x=h, y=errn4]{./GPWBAE2D2On4Casecosxsiny50pts.dat};
\addlegendentry{$n=4$ \& $q=4$}
%\addplot table [x=h, y=errn5]{./GPWBAE2D2Op4Casecosxsiny50pts.dat};
%\addlegendentry{$n=5$}
\addplot[dotted] coordinates {(2*10^-4, 10^-10) (2*10^-1, 10^-1)};
\addlegendentry{order 3}
\addplot[dashed] coordinates {(10^-2, 10^-14) (10^-0, 10^-2)};
\addlegendentry{order 6}
\end{loglogaxis}
\end{tikzpicture}
%$$
%\begin{array}{c||c|c|c|c|c|}
%q\backslash n & 1&2&3&4&5\\\hline\hline
%1&\bf 2 & \bf 3 &     3 &     3 &     3 \\\hline
%2&\bf 2 & \bf 3 & \bf 4 &     4 &     4 \\\hline
%3&\bf 2 & \bf 3 & \bf 4 & \bf 5 &     5 \\\hline
%4&\bf 2 & \bf 3 & \bf 4 & \textcolor{red}{\bf 6} & \bf 6 \\\hline
%\end{array}$$
\caption{%$cs$ test case. 
\capfig{{cs}}}\label{csNR}
\end{figure}

We summarize in the following table the orders of convergence observed, always using $\VGPW$ with $p=2n+1$. The bold entries correspond to cases covered by Theorem \ref{thethm} i.e. $n+1$ for $q\leq n-1$, and the red entries correspond to cases with order of convergence observed higher than the theorem predicts.
$$
\begin{array}{c||c|c|c|c|c|}
q\backslash n & 1&2&3&4&5\\\hline\hline
1&\bf 2 & \bf 3 &     3 &     3/4 &     3 \\\hline
2&\bf 2 & \bf 3 & \bf 4 & \geq 4 & \geq 4 \\\hline
3&\bf 2 & \bf 3 & \bf 4 & \textcolor{red}{\bf \geq 5} &     5 \\\hline
4&\bf 2 & \bf 3 & \bf 4 & \textcolor{red}{\bf \geq 5} & \bf 6 \\\hline
\end{array}$$
%\begin{tabular}[r]{cc}
%lol&bab
%\end{tabular}
We can see from this table that in all cases covered by the theorem, we observe a convergence rate equal or slightly better than predicted. But it would seem that the hypotheses of the theorem are sharp.

%%%%%%%%%%%%%%%%%%%%%%%%%%%%%%%%
\section{Conclusion}
In this work we have considered local properties in the neighborhood of a point $(x_0,y_0)\in\mathbb R^2$, for an operator $\Lal$. To summarize, we followed the steps announced in the introduction:
\begin{enumerate}
\item construction of GPWs $\varphi$ such that $\Lal\varphi(x,y)= O\left(\| (x,y)-(x_0,y_0) \|^q\right)$
\begin{enumerate}
\item choose an ansatz for $\displaystyle \varphi(x,y)=\exp  \sum_{0\leq i+j\leq d P} \luij (x-x_0)^i(y-y_0)^j$
\item identify the corresponding $N_{dof} = \frac{(dP+1)(dP+2)}{2}$ degrees of freedom, and $N_{eqn} = \frac{q(q+1)}{2}$ constraints, namely respectively
$$ 
\begin{array}{l}
\{\luij ; (i,j)\in\mathbb N^2,0\leq i+j\leq d P \},\\
\{\mathcal D^{(I,J)} \Lal\varphi(x_0,y_0) = 0 ;  (I,J)\in\mathbb N^2,0\leq I+J< q\}.
\end{array}
$$
\item %choose the number of degrees of freedom adequately,
 for $dP=q+\mfn-1$,  the number of degrees of freedom is $N_{dof}=\frac{(\mfn+q)(\mfn+q+1)}{2}> N_{eqn}$ and this ensures that there are linear terms in all the constraints
\item %study the structure of the resulting system and 
identify $N_{dof}-N_{eqn}= \mfn q+ \frac{\mfn(\mfn+1)}{2}$ additional constraints, namely
$$
\text{Fixing } \{\lambda_{i,j}, (i,j)\in\mathbb N^2, i+j<q+\mfn \text{ and } i<\mfn\}
$$
to obtain a global system that can be split into a hierarchy of linear triangular subsystems
\item compute the remaining $N_{eqn}$ degrees of freedom by forward substitution for each triangular subsystem, therefore at minimal computational cost
\end{enumerate}
\item\label{interp} interpolation properties
\begin{enumerate}
\item\label{props} thanks to the normalization, in particular $\{\lambda_{i,j} = 0, (i,j)\in\mathbb N^2, i+j<\mfn +q\text{ and } i<\mfn, i+j\neq 1\}$,
study the properties of the remaining $N_{eqn}$ degrees of freedom, that is $\{\lambda_{i,j} , (i,j)\in\mathbb N^2, i+j<\mfn +q\text{ and } i\geq \mfn\}$, with respect to $(\luoz,\luzo)$ 
\item identify a simple reference case depending only on two parameters, that is basis functions $\displaystyle \phi(x,y)=\exp  \luoz (x-x_0) + \luzo(y-y_0)$ depending only on the choice of  $(\luoz,\luzo)$, independently of $\phi$ being an exact solution to the constant coefficient equation
\item study the interpolation properties of this reference case with classical PW techniques
\item relate the general case to the reference case thanks to \ref{props}
\item prove the interpolation properties of the GPWs from those of the reference case
\end{enumerate}
\end{enumerate}
This construction process guarantees that the GPW function $\varphi$ satisfies the approximate Trefftz property $\Lal\varphi(x,y)= O\left(\| (x,y)-(x_0,y_0) \|^q\right)$ independently of the normalization, that is the values chosen for $\{\lambda_{i,j}, (i,j)\in\mathbb N^2, i+j<\mfn \}$, while the proof of interpolation properties heavily rely on the normalization.

This work focuses on interpolation of solutions of a PDE, and is limited to local results, in the neighborhood of a given point. In order to address the convergence of a numerical method for a boundary value problem on a domain $\Omega$ with a GPW-discretized Trefftz method, on a mesh $\mathcal T_h$ of $\Omega$, we will consider a space $\mathbb V_h$ of GPWs built element-wise, at the centroid $(x_0,y_0)=(x_K,y_K)$ of each element $K\in\mathcal T_h$, to study interpolation properties on $\Omega$. In particular, meshing the domain $\Omega$ to respect any discontinuity in the coefficients, the interpolation error on $\Omega$,  $\|(I-P_{\mathbb V_h})\|$, will converge at the same order as the local interpolation error on each element, and the crucial point will be to investigate the behavior of the constant $C(\Omega,n)$ from Theorem \ref{thethm}.
Related computational aspects of the construction of GPWs proposed in this work are currently under study.

%%%%%%%%%%%%%%%%%%%%%%%%%%%%%%%%%%%%%%%%%%%%%%%%%%%%%%%%%%%%%%%%%
%%%%%%%%%%%%%%%%%%%%%%%%%%%%%%%%%%%%%%%%%%%%%%%%%%%%%%%%%%%%%%%%%
\appendix
%%%%%%%%%%%%%%%%%%%%%%%%%%%%%%%%%%%%%%%%%%%%%%%%%%%%%%%%%%%%%%%%%
%%%%%%%%%%%%%%%%%%%%%%%%%%%%%%%%%%%%%%%%%%%%%%%%%%%%%%%%

\section{Chain rule in dimension 1 and 2}
For the sake of completeness, this section is dedicated to describing the formula to derive a composition of two functions, in dimensions one and two. 
A wide bibliography about this formula is to be found in \cite{ma}. It is linked to the notion of partition of an integer or the one of a set. The 1D version is not actually used in this work but is displayed here as a comparison with a 2D version, mainly concerning this notion of partition.
\subsection{Faa Di Bruno Formula}\label{app:FDBf}
Faa Di Bruno formula gives the $m$th derivative of a composite function with a single variable. It is named after Francesco Faa Di Bruno, but was stated in earlier work of Louis F.A. Arbogast around 1800, see \cite{007}.

If $f$ and $g$ are functions with sufficient derivatives, then
$${
\frac{d^m}{dx^m}f(g(x)) = m!\sum f^{(\sum_k b_k)} (g(x)) \prod_{k=1}^{m} \frac{1}{b_k!}\left(  \frac{g^{(k)}(x)}{k!} \right)^{b_k},
}$$
where the sum is over all different solutions in nonnegative integers $(b_k)_{k\in[\![1,m]\!]}$ of $\sum_k k b_k = m$. These solutions are actually the partitions of $m$.

\subsection{Bivariate version}\label{app:bivFDB}
The multivariate formula has been widely studied, the version described here is the one from \cite{CS} applied to dimension $2$. A linear order on $\mathbb N^2$ is defined by: $\forall (\mu,\nu)\in\left(\mathbb N^2\right)^2$, the relation $\mu\prec\nu$ holds provided that 
\begin{enumerate}
\item $\mu_1+\mu_2<\nu_1+\nu_2 $; or
\item $\mu_1+\mu_2=\nu_1+\nu_2 $ and $\mu_1<\nu_1$.
\end{enumerate}
If $f$ and $g$ are functions with sufficient derivatives, then
$${
\dx^i \dy^j %\frac{\partial^{i+j}}{\dx^i \dy^j} 
f(g(x,y)) =i!j! \sum_{1\leq\mu\leq i+j} f^{\mu}(g(x,y)) \sum_{s=1}^{i+j} \sum_{p_s((i,j),\mu)} \prod_{l=1}^s \frac{1}{k_l!}\left( \frac{1}{i_l!j_l!} \dx^{i_l}\dy^{j_l} %\frac{\partial^{i_l+j_l}f}{\dx^{i_l}\dy^{j_l}}
 (g(x,y))\right)^{k_l},
}$$
where the partitions of $(i,j)$ are defined by the following sets: $\forall \mu\in[\![1,i+j]\!]$, $\forall s\in[\![1,i+j]\!]$, $p_s((i,j),\mu)$ is equal to
$${
\left\{ (k_1,...,k_s;(i_1,j_1),\cdots,(i_s,j_s)):k_i>0,0\prec(i_1,j_1)\prec\cdots\prec(i_s,j_s),%\phantom{\sum_{l=1}^s }\right.}
%\enn{\left. 
\sum_{l=1}^s k_l=\mu,\sum_{l=1}^s k_li_l=i,\sum_{l=1}^s k_lj_l=j\right\}.
}$$
See \cite{ardi} for a proof of the formula interpreted in terms of collapsing partitions.

%
%\section{Notation}
%Summary of all the notations used throughout the paper - in an effort of single use of a notation (for a single meaning)

\section{Faa di Bruno}
The multivariate formula has been widely studied, the version described here is the one from \cite{CS} applied to dimension $2$. A linear order on $\mathbb N^2$ is defined by: $\forall (\mu,\nu)\in\left(\mathbb N^2\right)^2$, the relation $\mu\prec\nu$ holds provided that 
\begin{enumerate}
\item $\mu_1+\mu_2<\nu_1+\nu_2 $; or
\item $\mu_1+\mu_2=\nu_1+\nu_2 $ and $\mu_1<\nu_1$.
\end{enumerate}
If $f$ and $g$ are functions with sufficient derivatives, then
$$
\dx^i \dy^j %\frac{\partial^{i+j}}{\dx^i \dy^j} 
f(g(x,y)) =i!j! \sum_{1\leq\mu\leq i+j} f^{(\mu)}(g(x,y)) \sum_{s=1}^{i+j} \sum_{p_s((i,j),\mu)} \prod_{l=1}^s \frac{1}{k_l!}\left( \frac{1}{i_l!j_l!} \dx^{i_l}\dy^{j_l} %\frac{\partial^{i_l+j_l}f}{\dx^{i_l}\dy^{j_l}}
 (g(x,y))\right)^{k_l},
$$
$$
\dx^k \dy^{\ell-k}  
e^{P(x,y)} =k!{(\ell-k)}! \sum_{1\leq\mu\leq \ell} e^{P(x,y)} \sum_{s=1}^{\ell} \sum_{p_s((k,\ell-k),\mu)} \prod_{m=1}^s \frac{1}{k_m!}\left( \frac{1}{i_m!j_m!} \dx^{i_m}\dy^{j_m}  P(x,y)\right)^{k_m},
$$
where the partitions of $(i,j)$ are defined by the following sets: $\forall \mu\in[\![1,i+j]\!]$, $\forall s\in[\![1,i+j]\!]$, $p_s((i,j),\mu)$ is equal to
$$
\left\{ (k_1,...,k_s;(i_1,j_1),\cdots,(i_s,j_s)):k_i>0,0\prec(i_1,j_1)\prec\cdots\prec(i_s,j_s),%\phantom{\sum_{l=1}^s }\right.}
%\enn{\left. 
\sum_{l=1}^s k_l=\mu,\sum_{l=1}^s k_li_l=i,\sum_{l=1}^s k_lj_l=j\right\}.
$$
Note that $s$ is the number of different terms appearing in  the product, while $\mu$ is the number of terms in the product counting multiplicity, $k_m$ is the multiplicity of the $m$th term in the product, while $p_s$ represents the possible partitions of $(i,j)$.

Note that since $k_m>0$, the condition $\sum_{m=1}^s k_m=\mu$ implies that $\mu = \sum_{m=1}^s k_m\geq  \sum_{m=1}^s 1 = s$.

\section{Polynomial formulas}\label{PolFor}
Here are two important comments.
The first one concerns the product of polynomials. Assume that $\min(D_1,D_2)\geq q$. Then the product of two polynomials, respectively of degree $D_1$ and $D_2$, satisfies:
\begin{equation*}
\left(\sum_{i_1=0}^{D_1}\sum_{j_1=0}^{D_1-i_1} p_{i_1,j_1}x^{i_1}y^{j_1}\right)
\left(\sum_{i_2=0}^{D_2}\sum_{j_2=0}^{D_2-i_2} q_{i_2,j_2}x^{i_2}y^{j_2}\right)
=
\sum_{i=0}^{q-1}\sum_{j=0}^{q-1-i}
\left(\sum_{\tilde i=0}^i\sum_{\tilde j=0}^j p_{i-\tilde i,j-\tilde j}q_{\tilde i,\tilde j} \right)
x^iy^j+O(h^q).
\end{equation*}
 Since in particular the summation indices are such that $0\leq \tilde i\leq i$, $0\leq i-\tilde i\leq i$, $0\leq \tilde j\leq j$, and $0\leq j-\tilde j\leq j$, the only coefficients $p_{i,j}$ and $q_{i,j}$ appearing in the $(I_0,J_0)$ coefficient of the product have a length of the multi-index $i+j\leq I_0+J_0$. As a consequence, the only coefficients of several polynomials appearing in the $(I_0,J_0)$ coefficient of the product these several polynomials have a length of the multi-index $i+j\leq I_0+J_0$. The second comment turns to the derivative of a polynomial:
\begin{equation*}
\dx^I\dy^J\left(\sum_{i=0}^{D}\sum_{j=0}^{D-i} p_{i,j}x^{i}y^{j}\right)
=
\sum_{i=0}^{D-I-J}\sum_{j=0}^{D-I-J-i}
\frac{(i+I)!}{i!}\frac{(j+J)!}{j!} 
p_{i+I,j+J}
x^{i}y^{j}.
\end{equation*}
In particular the only coefficients $p_{i,j}$ appearing in the $(I_0,J_0)$ coefficient of the derivative has a length of the multi-index $i+j = I+J+I_0+J_0$.

%%%%%%%%%%
% biblio %
%%%%%%%%%%

\end{document}